\documentclass[10pt]{amsart}
\usepackage{a4wide}
\usepackage{amsfonts,graphicx,
amssymb,stmaryrd,amscd,amsmath,latexsym,amsbsy,xypic}
\xyoption{all}
\usepackage{marginnote}
\usepackage{hyperref} 
\numberwithin{equation}{section}
\theoremstyle{plain}

\usepackage{pst-node}
\usepackage{tikz-cd} 

\usepackage{pst-node}
\usepackage{pst-intersect}
\newcommand\rr{\mathbf{r}}
\newcommand\dd{d}
\newcommand\lm{\lambda}

\usepackage{etex}
\usepackage{amsfonts}
\usepackage{amsmath}
\usepackage[latin1]{inputenc}
\usepackage{setspace}
\usepackage{mathrsfs}
\usepackage{graphicx}
\graphicspath{ {./pics/} }
\usepackage{caption}
\usepackage{subcaption}
\usepackage{stmaryrd}
\usepackage[english]{babel}
\usepackage{enumitem}
\usepackage{amssymb}
\usepackage{amsthm}
\usepackage{bbding}
\usepackage[all]{xy}
\usepackage[a4paper]{geometry} 
\usepackage{tikz-cd}
\usepackage{pstricks}
\usepackage{pst-slpe}
\usepackage{multido}							% \begin{pspicture}...\end{pspicture}
\usepackage{pstricks-add}
\usepackage{pst-plot}

\newtheorem{thm}{Theorem}[section]

\newtheorem{prop}[thm]{Proposition}

\newtheorem{defi}[thm]{Definition}
\newtheorem{lem}[thm]{Lemma}
\newtheorem{cor}[thm]{Corollary}
\newtheorem{conv}[thm]{Convention}

\newtheorem{eg}[thm]{Example}

\theoremstyle{remark}
\newtheorem{rema}[thm]{Remark}

\title[Coupled classical dynamical Yang-Baxter and reflection equations]{Folded and contracted solutions of coupled classical dynamical Yang-Baxter and reflection equations}
\author{Jasper V. Stokman}
\address{KdV Institute for Mathematics, University of Amsterdam,
Science Park 105-107, 1098 XG Amsterdam, The Netherlands.}
\email{j.v.stokman@uva.nl}
\begin{document}
%%%%%%%%%%%%%%%%%%%%%%%%%%%%%%%
%%%%%%%%%%%%%%%%%%%%%%%%%%%%%%%
\maketitle
\begin{abstract}
In this paper we give a concrete recipe how to construct triples of
algebra-valued meromorphic functions on a complex vector space $\mathfrak{a}$ satisfying three coupled classical dynamical Yang-Baxter equations and an associated classical dynamical reflection equation. Such triples provide the local factors of a consistent system of first order differential operators on $\mathfrak{a}$, generalising asymptotic boundary Knizhnik-Zamolodchikov-Bernard (KZB) equations. 

The recipe involves folding and contracting $\mathfrak{a}$-invariant and $\theta$-twisted symmetric classical dynamical $r$-matrices along an involutive automorphism $\theta$. In case of the universal enveloping algebra of a simple Lie algebra $\mathfrak{g}$ we determine the subclass of Schiffmann's classical dynamical $r$-matrices which are 
$\mathfrak{a}$-invariant and $\theta$-twisted.

The paper starts with a section highlighting the connections between asymptotic (boundary) KZB equations, representation theory of semisimple Lie groups, and integrable quantum field theories. 
 \end{abstract}
%%%%%%%%%%%%%%%%%%%%%%%%%%%%%%%
%\tableofcontents
%%%%%%%%%%%%%%%%%%%%%%%%%%%%%%%%%
\section{Introduction}
%%%%%%%%%%%%%%%%%%%%%%%%%%%%%%%%%%
Let $U,A_\ell,A_r$ be complex associative algebras, and $\mathfrak{a}\subseteq U$ a finite dimensional subspace. In this paper we give a folding and contraction procedure which allows to construct functions $r^{\pm}: \mathfrak{a}^*\rightarrow U\otimes U$ and $\kappa: \mathfrak{a}^*\rightarrow A_\ell\otimes U\otimes A_r$ with $(r^+,r^-)$ satisfying three coupled classical dynamical Yang-Baxter equations $\textup{CYB}[t](r^+,r^-)=0$ ($1\leq t\leq 3$) and $\kappa$ satisfying a classical dynamical reflection equation $\textup{CR}(r^+,r^-;\kappa)=0$ relative to $(r^+,r^-)$. The explicit formulas for $\textup{CYB}[t]$ and $\textup{CR}$ are given by \eqref{CYB} and \eqref{CR}, respectively. 

The core construction is as follows. Consider a pair $(\theta,r)$ with $\theta\in\textup{Aut}(U)$ an involution containing $\mathfrak{a}$ in its $(-1)$-eigenspace, and 
$r: \mathfrak{a}^*\rightarrow U\otimes U$ an $\mathfrak{a}$-invariant $\theta$-twisted symmetric classical dynamical $r$-matrix
($\theta$-twisted symmetric means that $\widetilde{r}:=(\theta\otimes\textup{id})r$ takes values in $S^2U$). Denote by $m: U\otimes U\rightarrow U$ the multiplication map of $U$.
Then folding and contracting $r$ along $\theta$ produces a triple
\begin{equation}\label{ff}
(r^+,r^-,\kappa):=\Bigl(\frac{\widetilde{r}+r}{2},\frac{\widetilde{r}-r}{2},\frac{m(\widetilde{r})}{2}\Bigr)
\end{equation}
with $\textup{CYB}[t](r^+,r^-)=0$ ($1\leq t\leq 3$) and $\textup{CR}(r^+,r^-;\kappa)=0$. We also discuss extensions in which the folded solution $\kappa$ of the classical dynamical reflection equation takes values in $U^\theta\otimes U\otimes U^\theta$, with $U^\theta\subseteq U$ the subalgebra of elements in $U$ fixed by $\theta$. 

Before describing the contents of the paper in more detail, we will first explain how these integrability equations and the folding and contraction procedure naturally arise in representation theory and in conformal field theory. 

Triples $(r^+,r^-,\kappa)$ satisfying $\textup{CYB}[t](r^+,r^-)=0$ ($1\leq t\leq 3$) and $\textup{CR}(r^+,r^-;\kappa)=0$
serve as the local building blocks for consistent systems of first-order linear differential equations acting on functions $\mathfrak{a}^*\rightarrow V$, where $V$ is a $U^{\otimes N}$-module. Examples of such systems of equations
are asymptotic boundary Knizhnik-Zamo\-lod\-chi\-kov-Bernard (KZB) equations, in which case we have $(U,\mathfrak{a})=(U(\mathfrak{g}),\mathfrak{h})$
with $U(\mathfrak{g})$ the universal enveloping algebra of a complex semisimple Lie algebra $\mathfrak{g}$ and $\mathfrak{h}$ a Cartan subalgebra of $\mathfrak{g}$.
These equations are satisfied by restrictions to the split Cartan subgroup of the $N$-point spherical functions, which form a special class of
vector-valued spherical functions on split real connected semisimple Lie groups associated to $\mathfrak{g}$, see \cite{SR} and Subsection \ref{22}.

The terminology ``asymptotic boundary KZB equations'' is motivated by the fact that the equations are asymptotic remnants of 
the consistency equations for correlation functions in Wess-Zumino-Witten (WZW) conformal field theories on an elliptic curve with conformally invariant 
boundary conditions. 
This limit transition will be shortly discussed in Subsection \ref{DegScheme}. The boundary KZB equations themselves and the representation theoretic interpretation of $N$-point correlation functions in WZW conformal field theory on the elliptic curve with conformally invariant boundary conditions will be discussed in a forthcoming paper of N. Reshetikhin and the author. 

As mentioned before, solutions of asymptotic KZB equations are 
restrictions to the split Cartan subgroup of a special class of spherical functions on split real connected semisimple Lie groups. The split real form of $\mathfrak{g}$ determines a Chevalley involution $\sigma\in\textup{Aut}(\mathfrak{g})$ satisfying $\sigma\vert_{\mathfrak{h}}=-\textup{id}_{\mathfrak{h}}$, which extends to an involutive automorphism of $U(\mathfrak{g})$. It was already observed in \cite{SR} that the core triple $(r^+,r^-,\kappa)$ underlying the asymptotic boundary KZB equations is obtained from Felder's \cite{F} $\mathfrak{h}$-invariant and $\sigma$-twisted symmetric trigonometric classical dynamical $r$-matrix $r: \mathfrak{h}^*\rightarrow \mathfrak{g}\otimes\mathfrak{g}$ (see \eqref{Frrr}) by folding and contracting along $\sigma$. Recall here that Felder's classical dynamical $r$-matrix itself forms the fundamental building block of the asymptotic KZB equations, 
which are asymptotic remnants of the consistency equations for correlation functions in WZW conformal field theory on an elliptic curve without boundaries, see \cite{F,FW,E,ES}. Their solutions are given in terms of restrictions to Cartan subgroups of generalised trace functions on semisimple Lie groups, see \cite{ES}.

So the context of spherical functions and asymptotic boundary KZB equations leads to a natural example of the general folding and contraction procedure that we develop in this paper. The other explicit examples we will obtain in this paper also concern $U=U(\mathfrak{g})$ with involution $\sigma$. In this case $\mathfrak{a}$ will be a subspace of $\mathfrak{h}$ and the classical dynamical $r$-matrices $r:\mathfrak{a}^*\rightarrow \mathfrak{g}\otimes\mathfrak{g}$ underlying the folding and contraction are from an appropriate subclass of Schiffmann's \cite{S} family of classical dynamical $r$-matrices (see Definition \ref{tripledef}). 

For non-split real semisimple Lie groups a similar folding and contraction takes place in the theory of $N$-point spherical functions, but the integrability conditions for the building blocks of the associated asymptotic boundary KZB equations are more complicated, see \cite{RS} for details.\\

The contents of the paper is as follows.

Section \ref{Se1} provides a self-contained exposition explaining the role of Felder's trigonometric classical dynamical $r$-matrix, its folded and contracted versions, and the asymptotic (boundary) KZB equations in representation theory of semisimple Lie groups.
We furthermore place the asymptotic (boundary) KZB equations in a degeneration scheme containing
KZB type equations and Gaudin type Hamiltonians of types $A_N$ and $C_N$.

In Section \ref{CommSection} the coupled classical dynamical Yang-Baxter equations and associated classical dynamical reflection equation are introduced. We explain how solutions of these equations give rise to commuting first order differential operators, thereby extending results from \cite[\S 6.7]{SR}. 

In Section \ref{S2} we relate the usual classical dynamical Yang-Baxter equation for $r^+-r^-$ to the coupled classical dynamical Yang-Baxter equations for the pair $(r^+,r^-)$. This  involves a fourth classical {\it non-dynamical} Yang-Baxter type equation. We also show how extra terms can be added to the core solutions of the associated classical dynamical reflection equation. In case of the solution of the classical dynamical reflection equation appearing in the asymptotic boundary KZB equations, these extensions relate to adding non-trivial spin reflection terms at the boundaries, see \cite[\S 6]{SR}. 
 
 Section \ref{foldingSection} establishes the folding and contraction procedure described in the second paragraph of the introduction. 
 It is also shown that these folded pairs $(r^+,r^-)$ are solutions of a fourth coupled classical non-dynamical Yang-Baxter equation. 

In Section \ref{ExampleSection} we consider the cases where $U=U(\mathfrak{g})$ is the universal enveloping algebra of a complex simple Lie algebra $\mathfrak{g}$ and $\mathfrak{a}$ is a subspace of a fixed Cartan subalgebra $\mathfrak{h}$ of $\mathfrak{g}$. Etingof and Schiffmann \cite{ES,S} classified the associated
quasi-unitary $\mathfrak{a}$-invariant solutions of the classical dynamical Yang-Baxter equation with values in $\mathfrak{g}\otimes\mathfrak{g}$ and coupling constant $1$. Schiffmann's \cite{S} explicit representatives $r^{\textup{Sch}}: \mathfrak{a}^*\rightarrow\mathfrak{g}\otimes\mathfrak{g}$ of the gauge equivalence classes of this set of solutions are 
labeled by $\mathfrak{a}$-admissible generalised Belavin-Drinfeld triples. Extremal cases are Felder's \cite{F} solution, and the half-Casimirs. The latter ones are the standard non-dynamical solutions of the classical dynamical Yang-Baxter equation.

We show that 
if  $\theta\in\textup{Aut}(\mathfrak{g})$ is an $\mathfrak{h}$-stabilising involution for which there exists a $\theta$-twisted symmetric solution $r^{\textup{Sch}}: \mathfrak{a}^*\rightarrow U(\mathfrak{g})^{\otimes 2}$, then $\theta$ is a Chevalley involution relative to $\mathfrak{h}$.
In this case we describe all $\theta$-twisted symmetric solutions $r^{\textup{Sch}}$. By the folding and contraction procedure from Section \ref{foldingSection}, this provides new families of solutions of the coupled classical dynamical Yang-Baxter equations and of the associated classical dynamical reflection equation depending on a choice $\Gamma$ of the set $\Delta$ of simple roots and on a suitable class of subspaces $\mathfrak{a}\subseteq\mathfrak{h}$. The example related to $N$-point spherical functions on a split real simple Lie group corresponds to $(\Gamma,\mathfrak{a})=(\Delta,\mathfrak{h})$. 

Finally, in Section \ref{AsGa}, we consider the twisted symmetric solution $r^{\textup{Sch}}$ for $\mathfrak{a}=0$ and use the folding and contraction procedure,
together with the results from Section \ref{CommSection}, to construct commuting asymptotic Gaudin Hamiltonians of type $C_N$.
 \vspace{.2cm}\\
\noindent
{\it Conventions:} For a group $L$ and a subset $S\subset L$ we write $N_L(S)$ and $Z_L(S)$ for the normaliser and centraliser of $S$ in $L$, respectively. The set of conjugacy classes in $L$ is denoted by $\textup{Conj}_L$. All modules and representations are over $\mathbb{C}$ unless stated otherwise. Tensor products over $\mathbb{C}$ are denoted by $\otimes$. The universal enveloping algebra of
a complex Lie algebra $\mathfrak{g}$ is denoted by $U(\mathfrak{g})$. 
\vspace{.2cm}\\
\noindent
{\it Acknowledgements:} 
It is a pleasure to thank Nicolai Reshetikhin and Pavel Etingof for interesting discussions and comments.
The author was supported by the Netherlands Organization for Scientific Research (NWO). 

%%%%%%%%%%%%%%%%%%%%%%%%
\section{Representation theoretic context}\label{Se1}
%%%%%%%%%%%%%%%%%%%%%%%%

%%%%%%%%%%%%%%%%%%%%%%%%%%%%%%%%%%%
\subsection{Generalised trace functions and asymptotic KZB equations}\label{ab}
%%%%%%%%%%%%%%%%%%%%%%%%%%%%%%%%%%%%
In this subsection we revisit the Etingof-Kirillov-Schiffman theory on generalised trace functions for semisimple Lie groups (see, e.g, \cite{E0,E,EK,Ki,ES,EL}). The only new result in this subsection is a formula for the Schr{\"o}dinger operator of the spin Calogero-Moser system in terms of the radial component action of Felder's \cite{F} trigonometric classical dynamical $r$-matrix (Proposition \ref{rL}). The subsection also includes a short new derivation of the asymptotic Knizhnik-Zamolodchikov-Bernard (KZB) equations for generalised 
trace functions clearly separating the arguments involving the ``bulk'' (intertwiners) from the arguments involving the ``twisted cyclic boundary conditions'' 
(weighted traces), see Theorem \ref{KZBthm}.

Let $\mathfrak{g}$ be a complex semisimple Lie algebra, $\mathfrak{h}\subset\mathfrak{g}$ a Cartan subalgebra, and 
\[
\mathfrak{g}:=\mathfrak{h}\oplus\bigoplus_{\alpha\in R}\mathfrak{g}_\alpha
\]
the corresponding root space decomposition, with root system $R\subset\mathfrak{h}^*$. Let $\sigma\in\textup{Aut}(\mathfrak{g})$ be a Chevalley involution relative to the Cartan subalgebra $\mathfrak{h}$ (so $\sigma\vert_{\mathfrak{h}}=-\textup{id}_{\mathfrak{h}}$). Let $K:\mathfrak{g}\times\mathfrak{g}\rightarrow\mathbb{C}$ be the Killing form of $\mathfrak{g}$, and denote by 
$(\cdot,\cdot)$ the restriction of $K$ to a nondegenerate symmetric bilinear form on $\mathfrak{h}$. We also write $(\cdot,\cdot)$ for the bilinear form on $\mathfrak{h}^*$ obtained by dualising $(\cdot,\cdot)$.  Let $\{x_j\}_{j=1}^n$ be a basis of $\mathfrak{h}$ such that $(x_i,x_j)=\delta_{i,j}$ (it exists since $(\cdot,\cdot)$ is the complex bilinear extension of a scalar product on a suitable real form of $\mathfrak{h}$).
For $\nu\in\mathfrak{h}^*$ denote by $t_\nu\in\mathfrak{h}$ the element such that $\nu(\cdot)=(t_\nu,\cdot)$.
Choose root vectors 
$e_\alpha\in\mathfrak{g}_\alpha$ such that $[e_\alpha,e_{-\alpha}]=t_\alpha$ and $\sigma(e_\alpha)=-e_{-\alpha}$ (see, e.g., \cite[\S 25.2]{Hu}).
Let $R^+=\{\beta_1,\ldots,\beta_m\}$ be a choice of positive roots. Write $\mathfrak{n}:=\bigoplus_{\alpha\in R^+}\mathfrak{g}_\alpha$ for the corresponding nilpotent
subalgebra of $\mathfrak{g}$. 

Fix a complex connected Lie group $G$ with Lie algebra $\mathfrak{g}$, and denote by $T\subset G$ the connected analytic subgroup with Lie algebra $\mathfrak{h}$. 
Let $\Lambda$ be the group of linear functionals $\mu\in\mathfrak{h}^*$ that integrate to an algebraic character $\xi_\mu: T\rightarrow\mathbb{C}^\times$. Then $\Lambda$ is a lattice satisfying $Q\subseteq\Lambda\subseteq P$, with $Q$ and $P$ the root and weight lattice of $R$, respectively. The algebra of regular functions on $T$
is $\mathcal{P}_\Lambda=\bigoplus_{\mu\in\Lambda}\mathbb{C}\xi_\mu$. The multiplication rules are $\xi_\mu\xi_\nu=\xi_{\mu+\nu}$ and $\xi_0=1$. We write $\mathcal{P}$ for the subalgebra spanned by $\xi_\mu$ ($\mu\in Q$), and $\mathcal{Q}$ for its quotient field. 
 
Let $U_{\mathcal{Q}}(\mathfrak{g})$ be the $\mathcal{Q}$-algebra $\mathcal{Q}\otimes U(\mathfrak{g})$,
and denote by $\mathcal{D}$ the algebra of differential operators on $T$ with coefficients in the ring $U_{\mathcal{Q}}(\mathfrak{g})$. We write $\partial_h\in\mathcal{D}$ for the directional derivative in direction $h\in\mathfrak{h}$.

Let $\mathcal{B}$ be the Poincar{\'e}-Birkhoff-Witt (PBW) basis of $U(\mathfrak{g})$ relative to the ordered basis
$(x_1,\ldots,x_n,e_{-\beta_m},\ldots,e_{-\beta_1},e_{\beta_1},\ldots,e_{\beta_m})$ of $\mathfrak{g}$. 

%%%%%%%%%%%%%%%%%%%%%%%%%%%%%%%%%%%%%%%%%%%%%%%%%%%%%%%%%%%%%%%
\begin{lem}\label{lemmaL}
There exists a unique complex linear map
\begin{equation}\label{Lmap}
U(\mathfrak{g})\rightarrow \mathcal{D}, \qquad X\mapsto L_X
\end{equation}
satisfying
\begin{enumerate}
\item[\textup{(1)}] $L_1=1$, 
\item[\textup{(2)}] if $be_\alpha\in\mathcal{B}$ is a $\textup{PBW}$-word ending with the letter
$e_\alpha$ \textup($\alpha\in R$\textup{)} then
\begin{equation*}
L_{be_\alpha}=
\frac{e_\alpha}{\xi_{-\alpha}-1}L_b+\frac{1}{\xi_\alpha-1}L_{[e_\alpha,b]},
\end{equation*}
\item[\textup{(3)}] if $bx_j\in\mathcal{B}$ is a $\textup{PBW}$-word ending with the letter $x_j$ \textup{(}$1\leq j\leq n$\textup{)} then
\[
L_{bx_j}=\partial_{x_j}L_b.
\]
\end{enumerate}
\end{lem}
%%%%%%%%%%%%%%%%%%%%%%%%%%%%%%%%%%%%%%%%%%%%%%%%%%%%%%%
\begin{proof}
The existence and uniqueness 
of the map \eqref{Lmap} follows easily by induction to the degree of $X\in U(\mathfrak{g})$ with respect to the standard filtration on $U(\mathfrak{g})$. 
\end{proof}
%%%%%%%%%%%%%%%%%%%%%%%%%%%%%%%%%%%%%%%%
%%%%%%%%%%%%%%%%%%%%
\begin{eg}\label{ire}
\textup{(i)} $X\mapsto L_X$ assigns to $X\in U(\mathfrak{h})$ the associated constant coefficient differential operator on $T$.\\
\textup{(ii)} A direct computation using \textup{(1)-(3)} gives for $\alpha\in R$,
\begin{equation*}
L_{e_\alpha}=\frac{e_{\alpha}}{\xi_{-\alpha}-1},\qquad
L_{x_je_\alpha}=\frac{e_\alpha}{\xi_{-\alpha}-1}\Bigl(\partial_{x_j}+\frac{\alpha(x_j)}{1-\xi_\alpha}\Bigr)
\end{equation*}
and
\begin{equation*}
L_{e_{-\alpha}e_\alpha}=\frac{1}{\xi_\alpha-1}\partial_{t_\alpha}-\frac{\xi_\alpha e_\alpha e_{-\alpha}}{(\xi_\alpha-1)^2}.
\end{equation*}
\end{eg}
%%%%%%%%%%%%%%%%%%%%%%
In the following proposition we will relate the differential operator $L_X$
with the action $X\cdot f$ of $X\in U(\mathfrak{g})$ as left-invariant differential operator on a special class of vector-valued smooth vector-valued functions $f$ on $G$.

Let $(V,\tau)$ be a finite dimensional $G$-representation. Endow $V$ with the left $U(\mathfrak{g})$-module structure obtained by differentiating
the $G$-action on $V$. The corresponding representation map $U(\mathfrak{g})\rightarrow\textup{End}(V)$ is again denoted by $\tau$.
The space $\Xi_V(G)$ of $V$-valued class functions on $G$ is defined to be the space of smooth functions $f: G\rightarrow V$ satisfying the equivariance property
\begin{equation}\label{twistedclass}
f(gg^\prime g^{-1})=\tau(g)f(g^\prime)\qquad \forall\, g,g^\prime\in G.
\end{equation}
%%%%%%%%%%%%%%%%%%%%%%%%
\begin{prop}
For $X\in U(\mathfrak{g})$ and $f\in \Xi_V(G)$ we have
\begin{equation}\label{radialcomponent}
(X\cdot f)\vert_{T}=L_X(f\vert_{T}).
\end{equation}
\end{prop}
%%%%%%%%%%%%%%%%%%%%%%%%
\begin{proof}
It suffices to show that 
\begin{equation}\label{td}
(1-t^{-\alpha})\bigl((Xe_\alpha)\cdot f\bigr)(t)=-\tau(e_\alpha)\bigl(X\cdot f\bigr)(t)+t^{-\alpha}\bigl([e_\alpha,X]\cdot f\bigr)(t)
\end{equation}
for $\alpha\in R$, $X\in U(\mathfrak{g})$, $f\in \Xi_V(G)$ and $t\in T$. Formula \eqref{td} follows from the following explicit computation,
\begin{equation*}
\begin{split}
\bigl((Xe_\alpha)\cdot f\bigr)(t)&=\frac{d}{ds}\biggr\vert_{s=0}
\tau(\exp(-se_\alpha))\bigl(X\cdot f\bigr)(\exp(se_\alpha)t)\\
&=-\tau(e_\alpha)(X\cdot f)(t)+\frac{d}{ds}\biggr\vert_{s=0}\bigl(X\cdot f\bigr)(t\exp(st^{-\alpha}e_\alpha))\\
&=-\tau(e_\alpha)\bigl(X\cdot f\bigr)(t)+t^{-\alpha}\bigl((e_\alpha X)\cdot f\bigr)(t).
\end{split}
\end{equation*}
Here we used $f(g^\prime g)=\tau(g^{-1})f(gg^\prime)$ ($g,g^\prime\in G$) for the first equality, cf. \eqref{twistedclass}.
\end{proof}
%%%%%%%%%%%%%%%%%%%%%%%
\begin{rema}\label{classfunctionrem}
The interpretation of $L_X$ as the radial component of the action of $X\in U(\mathfrak{g})$ on $\Xi_V(G)$ as left-invariant differential operator
extends to $V$-valued class functions on (the regular part of) a split real form $G_{\mathbb{R}}$ of $G$, cf. \cite{CM}. It implies that the map \eqref{Lmap} is canonical and independent of the choices.
\end{rema}
%%%%%%%%%%%%%%%%%%%%%%%%%
The space $\Xi_V(G)$ of $V$-valued class functions is preserved by the action of the center $Z(\mathfrak{g})$ of $U(\mathfrak{g})$. Common $Z(\mathfrak{g})$-eigenfunctions in $\Xi_V(G)$ can be constructed as follows.

Let $(\pi,M)$ be a finite dimensional $G$-representation. We will denote by $X_M$ the corresponding infinitesimal action of $X\in U(\mathfrak{g})$ on $M$, and write $M[\mu]$ for the weight space of $M$ at weight $\mu\in\Lambda$. Let $V$ be a finite dimensional complex vector space.
The partial trace over $M[\mu]$ of a complex linear map
$\Psi\in\textup{Hom}(M,M\otimes V)$ is defined by
\[
\textup{Tr}_{M[\mu]}(\Psi):=\sum_i(\phi_i\otimes\textup{id}_V)\Psi(m_i)\in V,
\]
with $\{m_i\}_i$ a linear basis of $M[\mu]$ and $\{\phi_i\}_i$ its dual basis. Here we have extended $\phi_i\in M[\mu]^*$ to a linear functional on $M$ by requiring $\phi_i$ to vanish on the remaining weight spaces $M[\nu]$ ($\nu\not=\mu$). We write $\textup{Tr}_M(\Psi)\in V$ for the partial trace of $\Psi$ over the whole representation space $M$.
Assume now that $(\tau,V)$ is also a finite dimensional $G$-representation.
Denote by $\textup{Hom}_G(M,M\otimes V)$ the space of $G$-intertwiners $M\rightarrow M\otimes V$.

The {\it $V$-valued trace functions} $f_\Phi\in\Xi_V(G)$ ($\Phi\in\textup{Hom}_G(M,M\otimes V)$) of Etingof, Kirillov Jr. and Schiffmann \cite{E0,EK,Ki,ES} are defined by
\begin{equation}\label{fPhiglobal}
f_\Phi(g):=\textup{Tr}_M(\Phi\pi(g))\qquad (g\in G).
\end{equation}
Note that its restriction to $T$ is given by
\begin{equation}\label{fPhi}
f_\Phi\vert_T=\sum_{\mu\in\Lambda}\textup{Tr}_{M[\mu]}(\Phi)\xi_\mu\in\mathcal{P}_\Lambda\otimes V[0].
\end{equation}
If $V=\mathbb{C}$ is the trivial representation then $f_{\textup{Id}_M}=\chi_M$ is the character of $(\pi,M)$, and
\[
\chi_M\vert_T=f_{\textup{id}_M}\vert_T=\sum_{\mu\in\Lambda}\textup{dim}(M[\mu])\xi_\mu
\]
is the formal character of $M$. 

For irreducible finite dimensional $G$-representations $M$ we denote by $c_M: Z(\mathfrak{g})\rightarrow \mathbb{C}$ the central character of $M$ (the Harish-Chandra isomorphism allows to describe it explicitly in terms of the highest weight of $M$). By \eqref{radialcomponent} we have
%%%%%%%%%%%%%%%%%%%%%%
\begin{prop}\label{evProp}
If $M$ is irreducible and $\Phi\in\textup{Hom}_G(M,M\otimes V)$ then
\begin{equation}\label{ev}
L_Z(f_\Phi\vert_T)=c_M(Z)f_\Phi\vert_T\qquad \forall\, Z\in Z(\mathfrak{g}).
\end{equation}
\end{prop}
%%%%%%%%%%%%%%%%%%%%%%%%
The weight decomposition of $U(\mathfrak{g})$, viewed as $G$-represen\-ta\-tion with respect to the
adjoint action, is  
\[
U(\mathfrak{g})=\bigoplus_{\mu\in Q}U[\mu],\qquad U[\mu]:=\{X\in U(\mathfrak{g}) \,\, | \,\, \textup{Ad}(t)X=t^\mu X\quad \forall\, t\in T\}.
\]
Denote by $\mathcal{D}_\mu\subset\mathcal{D}$ the subspace of differential operators with coefficients
in $U_{\mathcal{Q}}[\mu]:=\mathcal{Q}\otimes U[\mu]$. {}From the definition of $L_X$ it immediately follows that $L_X\in\mathcal{D}_\mu$ if $X\in U[\mu]$. In particular, $L_X\in\mathcal{D}_0$ for $X\in Z(\mathfrak{g})$. 

The Weyl group $W:=N_G(T)/T$ acts naturally on $T$, $\mathcal{P}_\Lambda$, $\mathcal{Q}$, $U[0]$ and $\mathcal{D}_0$. 
The fact that the intersection of a regular conjugacy class in $G$ with $T$ is a $W$-orbit in $T$ implies that $f_\Phi\vert_T$ and 
$L_Z$ are $W$-invariant for $\Phi\in\textup{Hom}_G(M,M\otimes V)$ and
$Z\in Z(\mathfrak{g})\subset U[0]$. 

Proposition \ref{evProp} and its extension to $V$-valued class functions on the regular part of $G_{\mathbb{R}}$ (see Remark \ref{classfunctionrem}) then implies that
%%%%%%%%%%%%%%%%%%%%%%%%%%%%%%%%%%%%%%%
\begin{thm}\label{ir}
The map $X\mapsto L_X$ \textup{(}see \eqref{Lmap}\textup{)} restricts to an algebra embedding $Z(\mathfrak{g})\hookrightarrow \mathcal{D}_0^W$.
\end{thm}
%%%%%%%%%%%%%%%%%%%%%%%%%%%%%%%%%%%%%%%
A special role for applications in harmonic analysis, quantum integrable systems and asymptotic integrable quantum field theories is played by the
radial component $L_\Omega\in\mathcal{D}_0^W$ of the quadratic Casimir element
\begin{equation}\label{Omega}
\Omega:=\sum_{j=1}^nx_j^2+\sum_{\alpha\in R}e_{-\alpha}e_\alpha\in Z(\mathfrak{g}).
\end{equation}
The differential operator $L_\Omega$ can be computed by first expressing $\Omega$ in ``normally ordered form'',
\[
\Omega=\sum_{j=1}^nx_j^2+2t_\rho+2\sum_{\alpha\in R^+}e_{-\alpha}e_\alpha
\]
with $\rho:=\frac{1}{2}\sum_{\alpha\in R^+}\alpha$ the half-sum of positive roots, and subsequently applying the rules from Lemma \ref{lemmaL} (cf. Example \ref{ire}). It results in
%%%%%%%%%%%%%%%%%%%%%%%%%%%%%%%%%
\begin{cor}\label{evc}
We have
\begin{equation}\label{LO}
L_\Omega=\Delta+\sum_{\alpha\in R^+}\Bigl(\frac{1+\xi_{-\alpha}}{1-\xi_{-\alpha}}\Bigr)\partial_{t_\alpha}-2\sum_{\alpha\in R^+}\frac{\xi_{-\alpha}e_\alpha e_{-\alpha}}{(1-\xi_{-\alpha})^2}
\end{equation}
with Laplacian $\Delta:=\sum_{j=1}^n\partial_{x_j}^2$. 
\end{cor}
%%%%%%%%%%%%%%%%%%%%%%%%%%%%%%%%%

%%%%%%%%%%%%%%%%%%%%%%%%%%%%%%%%%%%%%%%%%%%%%%%%%%%%%%%%%%%%%%%
\begin{rema}
As a special case of Proposition \ref{evProp} one obtains the explicit second order differential equation
\begin{equation}\label{evfor}
L_\Omega(\chi_M\vert_T)=c_M(\Omega)\chi_M\vert_T
\end{equation}
for $\chi_M\vert_T$ if $M$ is irreducible.
In Freudenthal's \cite{Fre} algebraic proof of Weyl's \cite{We1,We2} character formula for $\chi_M\vert_T$, 
the second order differential equation \eqref{evfor} appeared after direct computations. 
Springer \cite{Sp} realised from Harish-Chandra's \cite{H1,H2} analytic theory on distributional characters of real reductive groups, which appeared around the same time as Freudenthal's paper \cite{Fre}, that \eqref{evfor} is a consequence of
the action of $\Omega$ as a biinvariant differential operator on class functions. This insight allowed Springer \cite{Sp} to derive an analogue of Weyl's character formula for connected semisimple groups over algebraically closed fields of sufficiently large characteristic. 
\end{rema}
%%%%%%%%%%%%%%%%%%%%%%%%%%%%%%%%%%%%%%%%%%%%%%%%%%%%%%%%%%%%%%%
\begin{rema}
The radial component method of Harish-Chandra and Casselman-Mili\v{c}i{\'c} \cite{CM} provides an explicit procedure to compute the radial component of the action of left-invariant differential operators on vector-valued spherical functions. The map $X\mapsto L_X$ can be recovered by applying the radial component method to a particular space of vector-valued spherical functions associated to the symmetric pair
$(G\times G,\textup{diag}(G))$, with $\textup{diag}(G)$ the group $G$ diagonally embedded into $G\times G$. 

Concretely, the relevant space of vector-valued spherical functions consists of the smooth functions $f: G\times G\rightarrow V$ satisfying 
\[
f(g_1g,g_2g)=f(g_1,g_2)\quad \&\quad f(gg_1,gg_2)=\tau(g)f(g_1,g_2)\qquad \forall\, g,g_1,g_2\in G.
\]
By the right-sided trivial $G$-equivariance, this space of spherical functions is isomorphic to the space $\Xi_V(G)$ of generalised trace functions through the map $f\mapsto \widetilde{f}$, $\widetilde{f}(g):=f(g,1)$ ($1$ the neutral element of $G$).
\end{rema}
%%%%%%%%%%%%%%%%%%%%%%%%

%%%%%%%%%%%%%%%%%%%%%%%%%%%%%%%%%%%%%%%%%%%%%%%%%%%%%%%%%%%%%%

The first order term of $L_\Omega$
can be removed by gauging with the Weyl denominator 
\begin{equation}\label{Wd}
q:=\xi_\rho\prod_{\alpha\in R^+}(1-\xi_{-\alpha}).
\end{equation}
Concretely, 
 \begin{equation}\label{gaugeD}
L_\Omega=q^{-1}\circ\Bigl(\Delta-2\sum_{\alpha\in R^+}\frac{\xi_{-\alpha}e_\alpha e_{-\alpha}}{(1-\xi_{-\alpha})^2}\Bigr)\circ q+(\rho,\rho).
\end{equation}

The removal of the first order term of $L_\Omega$ can also be achieved in the following manner. 
Let $m: U_{\mathcal{Q}}(\mathfrak{g})\otimes_{\mathcal{Q}} U_{\mathcal{Q}}(\mathfrak{g})
\rightarrow U_{\mathcal{Q}}(\mathfrak{g})$ be the $\mathcal{Q}$-bilinear extension of the multiplication map of $U(\mathfrak{g})$. 
Extend 
\eqref{Lmap} to a $\mathcal{Q}$-linear map $U_{\mathcal{Q}}(\mathfrak{g})\rightarrow\mathcal{D}$ by $f\otimes X\mapsto fL_X$ for $f\in\mathcal{Q}$
and $X\in U(\mathfrak{g})$. 
%%%%%%%%%%%%%%%%%%%%%%%%%%%%%%%%%%%%%%%%%%%%%%%%%%%%%
\begin{prop}\label{rL}
We have
\[
L_{m(\rr)}=-\frac{1}{2}\Delta+\sum_{\alpha\in R^+}\frac{\xi_{-\alpha}e_\alpha e_{-\alpha}}{(1-\xi_{-\alpha})^2}
\]
with 
\begin{equation}\label{Felderrr}
 \rr:=-\frac{1}{2}\sum_{j=1}^nx_j\otimes x_j-\sum_{\alpha\in R}\frac{e_{-\alpha}\otimes e_\alpha}{1-\xi_{-\alpha}}\in\mathcal{Q}\otimes \mathfrak{g}^{\otimes 2}.
 \end{equation}
 \end{prop}
%%%%%%%%%%%%%%%%%%%%%%%%%%%%%%%%%%%%%%%%%%%%%%%%%%%%%%%
\begin{proof}
In view of \eqref{LO} it suffices to show that
\begin{equation}\label{Omegar}
\Omega=2(y-m(\rr))
\end{equation}
with 
\begin{equation}\label{y}
y:=\frac{1}{2}\sum_{\alpha\in R^+}\Bigl(\frac{1+\xi_{-\alpha}}{1-\xi_{-\alpha}}\Bigr)t_\alpha,
\end{equation}
i.e., it suffices to show the identity
\begin{equation}\label{Ch}
 \Omega=\sum_{j=1}^nx_j^2+\sum_{\alpha\in R^+}\Bigl(\frac{1+\xi_{-\alpha}}{1-\xi_{-\alpha}}\Bigr)t_\alpha+2\sum_{\alpha\in R}
 \frac{e_{-\alpha}e_\alpha}{1-\xi_{-\alpha}}
 \end{equation}
in $U_{\mathcal{Q}}(\mathfrak{g})$. This is verified by a direct computation.
\end{proof}
%%%%%%%%%%%%%%%%%%%%%%%%%%%%%%%%%%%%%%%%%%%%%%%%%%%%%%%
The action of $L_{m(\rr)}$ on $\bigl(C^\infty(T)\otimes V[0]\bigr)^W$ will be denoted 
by $L_{m(\rr)}^V$. It is the Schr{\"o}dinger operator of the quantum spin Calogero-Moser system.
For classical Lie algebras $\mathfrak{g}$ and special representations $V$ this quantum system describes one-dimensional quantum spin particles, with $V[0]$ the combined internal spin spaces. Higher order quantum integrals are obtained by gauged radial components $L_Z^V$ of higher 
Casimir elements $Z\in Z(\mathfrak{g})$. This turns it into a quantum integrable system  (see, e.g., \cite[Chpt. 7]{EL}).

%%%%%%%%%%%%%%%%%%%%%%%%%%%%%%%%%%
\begin{rema}\label{remarkrrr}
The map $\mathfrak{h}^*\rightarrow\mathfrak{g}\otimes\mathfrak{g}$, defined by
\begin{equation}\label{Frrr}
\lambda\mapsto \rr(\exp(t_\lambda))=-\frac{1}{2}\sum_{j=1}^nx_j\otimes x_j-\sum_{\alpha\in R}\frac{e_{-\alpha}\otimes e_\alpha}{1-e^{-(\lambda,\alpha)}},
\end{equation}
is Felder's \cite{F} trigonometric solution of the classical dynamical Yang-Baxter equation. We abuse notation and simply write
$\rr(\lambda)$ for $\rr(\exp(t_\lambda))$ in the remainder of the paper.
Felder's $r$-matrix $\rr$ is the classical limit of the dynamical universal $R$-matrix $R$ of the quantised universal
enveloping algebra $U_q(\mathfrak{g})$ (see the proof of \cite[Prop. 7.10]{EL}). The quantum dynamical Yang-Baxter equation satisfied by $R$ reduces to the classical dynamical Yang-Baxter equation for $\rr$.
\end{rema}
%%%%%%%%%%%%%%%%%%%%%%%%%%%%%%%%%
\begin{rema}
In \cite[\S 7.3]{EL} the Schr{\"o}dinger operator $L_{m(\rr)}$ and its higher order quantum Hamiltonians are obtained as
limits of the
quantum Hamiltonians of relativistic versions of the quantum spin Calogero-Moser systems. 
The Schr{\"o}dinger operator of the relativistic quantum systems
can be expressed as a partial trace of the quantum dynamical $R$-matrix, see \cite[Thm. 6.9]{EL}. Proposition \ref{rL} is a nonrelativistic analogue of this formula. 
\end{rema}
%%%%%%%%%%%%%%%%%%%%%%%%%%%%%%%

The classical dynamical $r$-matrix $\rr$ satisfies the quasi-unitarity condition
\begin{equation}\label{qu}
\rr+\rr_{21}=-\omega,\qquad \omega:=\sum_{j=1}^nx_j\otimes x_j+\sum_{\alpha\in R}e_{-\alpha}\otimes e_\alpha,
\end{equation}
where $\rr_{21}$ is obtained from $\rr$ by interchanging the two tensor components. The element $\omega\in (S^2\mathfrak{g})^G$ is called the mixed Casimir element, and satisfies
$m(\omega)=\Omega$.

For the derivation of the asymptotic KZB equations we also need a twisted analogue of \eqref{qu}. The twisting is with respect to the following formal analogue of the action $t\mapsto \textup{Ad}(t^{-1})$ ($t\in T$).
%%%%%%%%%%%%%%%%%%%%%%%%
\begin{defi}
We write $\eta\in\textup{Aut}(U_{\mathcal{Q}}(\mathfrak{g}))$ for the $\mathcal{Q}$-linear automorphism satisfying 
\[
\eta(x):=\xi_{-\mu}x\qquad (x\in U[\mu])
\]
for all $\mu\in Q$.
\end{defi}
%%%%%%%%%%%%%%%%%%%%%%%%
The twisted analogue of \eqref{qu} is
%%%%%%%%%%%%%%%%%%%%%%%%
\begin{lem}\label{rsymmlemma}
We have
\begin{equation}\label{twistedcyclicr}
\rr+(\rr^{\eta_2})_{21}=-\sum_{j=1}^nx_j\otimes x_j,
\end{equation}
where $\rr^{\eta_2}:=(\textup{id}\otimes\eta)(\rr)$.
\end{lem}
%%%%%%%%%%%%%%%%%%%%%%%%%%%%%%%%%%%
\begin{proof}
This follows by a direct computation.
\end{proof}
%%%%%%%%%%%%%%%%%%%%%%%%%%%%%%%%%%%
We now turn to asymptotic analogues of KZB equations. Up to this moment we looked at $f_\Phi\vert_T$ as the restriction to the Cartan subgroup $T$ of an (elementary)
spherical function on $G$. Instead we will now think of $\Phi\in\textup{Hom}_G(M,M\otimes V)$ as an asymptotic vertex operator, and view the operation
\begin{equation}\label{btb}
\Phi\mapsto f_\Phi\vert_T
\end{equation}
as imposing twisted cyclic boundary conditions. 

Of special interest is the case that $\Phi$ is the composition of $N$ intertwiners.
We then have $V=V_1\otimes\cdots\otimes V_N$ with $(\tau_i,V_i)$ ($1\leq i\leq N$) finite dimensional $G$-representations, and 
\begin{equation}\label{Nvertexoperator}
\Phi=(\Phi_1\otimes\textup{id}_{V_2\otimes\cdots\otimes V_N})\cdots (\Phi_{N-1}\otimes\textup{id}_{V_N})\Phi_N\in\textup{Hom}_G(M_N,M_0\otimes V)
\end{equation}
with $\Phi_j\in\textup{Hom}_G(M_j,M_{j-1}\otimes V_j)$ and $M:=M_0=M_N$.
Then Etingof and Schiffmann \cite[Thm. 3.1]{ES} determined the following system of first order linear differential equations for $f_\Phi\vert_T$.
%%%%%%%%%%%%%%%%%%%%%%%%%%%%%%%%%%%%%%%%%%%%%%%%%%%%%%%
\begin{thm}[\cite{ES}]\label{KZBthm}
Suppose that the $M_j$ \textup{(}$0\leq j\leq N$\textup{)} are finite dimensional irreducible $G$-representations and that $M:=M_0=M_N$. Let $\Phi\in\textup{Hom}_G(M,M\otimes V)$ be a $G$-intertwiner of the form \eqref{Nvertexoperator}, then
\begin{equation}\label{KZBformula}
\begin{split}
\Bigl(\sum_{j=1}^n(x_j)_{V_i}\partial_{x_j}-\sum_{j=1}^{i-1}\rr_{V_jV_i}+&\sum_{j=i+1}^N\rr_{V_iV_j}+y_{V_i}\Bigr)f_\Phi\vert_T=\\
&=\frac{1}{2}\bigl(c_{M_i}(\Omega)-c_{M_{i-1}}(\Omega)\bigr)f_\Phi\vert_T
\end{split}
\end{equation}
for $i=1,\ldots,N$, with the sub-labels indicating on which tensor components of $V$ the elements $x_j$, $\rr$ and $y$ are acting. 
\end{thm}
%%%%%%%%%%%%%%%%%%%%%%%%%%%%%%%%%%%%%%%%%%%%%%%%%%%%%%%%
\begin{proof}
We give a proof that separates the ``bulk'' computations, only involving the product $\Phi$ of $N$ intertwiners $\Phi_j$, from the computations involving the  ``boundary condition'' \eqref{btb}.  

The space $\textup{Hom}(M_i,M_{i-1}\otimes V_i)$ of $\mathbb{C}$-linear maps $M_i\rightarrow M_{i-1}\otimes V_i$ admits a right $U(\mathfrak{g})\otimes U(\mathfrak{g})$-action by
\[
\Psi\ast(a\otimes b):=S(a)_{V_i}\Psi b_{M_i},
\]
with $S$ the anti-automorphism of $\mathfrak{g}$ such that $S\vert_{\mathfrak{g}}=-\textup{id}_{\mathfrak{g}}$. Extending scalars to $\mathcal{Q}$, the decomposition \eqref{Omegar} of $\Omega$ and $\Omega_{M_j}=c_{M_j}(\Omega)\textup{id}_{M_j}$ give
\begin{equation}\label{KZBop}
\bigl(\frac{1}{2}\bigl(c_{M_i}(\Omega)-c_{M_{i-1}}(\Omega)\bigr)-y_{V_i}\bigr)\Phi_i=-\rr\Phi_i+\Phi_i\ast \rr.
\end{equation}
Consider the composition \eqref{Nvertexoperator} with $\Phi_i$ replaced by \eqref{KZBop}, and push the first leg of $\rr$ in $\rr\Phi_i$ 
(respectively the second leg of $\rr$ in $\Phi_i\ast \rr$)
to the far left (respectively far right), we get the identity
\begin{equation}\label{ttdd}
\begin{split}
\bigl(\frac{1}{2}\bigl(c_{M_i}(\Omega)-&c_{M_{i-1}}(\Omega)\bigr)-y_{V_i}\bigr)\Phi=\\
=&\Bigl(-\sum_{j=1}^{i-1}\rr_{V_jV_i}+\sum_{j=i+1}^N\rr_{V_iV_j}\Bigr)\Phi+\bigl(\Phi\ast_i\rr-\rr_{M_0V_i}\Phi\bigr),
\end{split}
\end{equation}
where we write $\Psi\ast_i(a\otimes b):=S(a)_{V_i}\Psi b_{M_N}$ for $\Psi\in\textup{Hom}(M_N,M_0\otimes V)$ and $a,b\in U(\mathfrak{g})$. 
We now assume $M:=M_0=M_N$ and consider the effect of the boundary condition \eqref{btb} on the term
$\Phi\ast_i\rr-\rr_{M_0V_i}\Phi$. 

For a complex linear map $\Psi\in\textup{Hom}(M,M\otimes V)$ the cyclicity of the trace implies
\[
\textup{Tr}_{M[\mu-\nu]}\bigl(\Psi X_M\bigr)=\textup{Tr}_{M[\mu]}\bigl(X_M\Psi\bigr)\qquad (X\in U[\nu]).
\] 
As a consequence we have
\[
\sum_\mu\textup{Tr}_{M[\mu]}\bigl(\Psi X_M\bigr)\xi_\mu=\sum_\mu\textup{Tr}_{M[\mu]}(\eta(X)_M\Psi)\qquad \forall\, X\in U(\mathfrak{g}).
\]
This implies 
\begin{equation}\label{todoo}
\begin{split}
\sum_\mu\textup{Tr}_{M[\mu]}\bigl(\Phi\ast_i\rr-\rr_{MV_i}\Phi\bigr)\xi_\mu&=
-\sum_\mu\textup{Tr}_{M[\mu]}
\bigl((\rr_{MV_i}+(\rr^{\eta_2})_{V_iM})\Phi\bigr)\xi_\mu\\
&=\sum_{j=1}^n(x_j)_{V_i}\partial_{x_j}\bigl(f_\Phi\vert_T\bigr),
\end{split}
\end{equation}
where the last equality follows from \eqref{twistedcyclicr}. The differential equation \eqref{KZBformula} now follows by applying the boundary operation \eqref{btb} to \eqref{ttdd} and substituting \eqref{todoo}.
\end{proof}
%%%%%%%%%%%%%%%%%%%%%%%%%%%%%%%%%%%%%%%%%%%%%%%%%%%%%%%%
We call \eqref{KZBformula} the asymptotic Knizhnik-Zamolodchikov-Bernard (KZB) equations (see Subsection \ref{DegScheme} for a justification of the terminology). 
The consistency of the asymptotic KZB equations follows either directly from the classical dynamical Yang-Baxter equation satisfied by $\rr$ (see Section \ref{CommSection}), or from an argument similar to the one leading to Theorem \ref{ir}. 

%%%%%%%%%%%%%%%%%%%%%%%%%%%%%%%%%%%%
\subsection{Spherical functions and boundary asymptotic KZB equations}\label{22}
%%%%%%%%%%%%%%%%%%%%%%%%%%%%%%%%%%%%
In this subsection we will discuss some of the recent results in \cite[\S 6]{SR}. We focus here on the emergence of folding and contraction and on the parallels with the theory of generalised trace functions, as discussed in the previous subsection.

Consider the decomposition 
\[
\mathfrak{g}=\mathfrak{k}\oplus\mathfrak{p}
\]
of $\mathfrak{g}$ in the $(+1)$-eigenspace $\mathfrak{k}$ and $(-1)$-eigenspace $\mathfrak{p}$ of the Chevalley involution $\sigma$. The fix-point Lie subalgebra $\mathfrak{k}$ has linear basis $\{e_\alpha-e_{-\alpha}\,\, | \,\, \alpha\in R^+\}$ and 
\[
\mathfrak{p}=\mathfrak{h}\oplus\bigoplus_{\alpha\in R^+}\mathbb{C}(e_\alpha+e_{-\alpha}).
\] 
Denote by $K\subset G$ the analytic subgroup with Lie algebra $\mathfrak{k}$.

To derive a spherical/boundary analogue of the asymptotic KZB equations, we need to split off
an appropriate symmetric factor from $\rr$ to serve the altered boundary conditions. 

The bulk computations in the proof of Theorem \ref{KZBthm} depend on a particular choice of splitting of the action of the quadratic Casimir, see \eqref{KZBop}. 
The following lemma analyses the flexibility in the choice of splitting.

%%%%%%%%%%%%%%%%%
\begin{lem}\label{operatorKZB}
Let $M,M^\prime,V$ be three finite dimensional $\mathfrak{g}$-modules and fix an intertwiner $\Phi\in\textup{Hom}_{\mathfrak{g}}(M,M^\prime\otimes V)$.
Let $s,s^\prime\in\mathfrak{g}\otimes\mathfrak{g}$ and set 
\begin{equation}\label{rel}
s^+=\frac{1}{2}(s_{21}^\prime+s),\qquad s^-=\frac{1}{2}(s^\prime-s).
\end{equation}
Then
\begin{equation}\label{factor}
\frac{1}{2}\Bigl(\Phi m(s)-(m(s)\otimes 1)\Phi\Bigr)=s^+\Phi+\Phi\ast s^-+\Bigl(1\otimes \frac{m(s^\prime)}{2}\Bigr)\Phi.
\end{equation}
\end{lem}
%%%%%%%%%%%%%%%%%%
\begin{proof}
Start with the right hand side of \eqref{factor} and remove the action on $V$ using the intertwining property of $\Phi$. The right hand side of \eqref{factor} then
becomes
\[
\Phi m\Bigl(\frac{s^\prime}{2}-s^-\Bigr)-\Bigl(m\Bigl(s^+-\frac{s_{21}^\prime}{2}\Bigr)\otimes 1\Bigr)\Phi+\Phi\triangleright \Bigl(\frac{s^\prime}{2}+\frac{s_{21}^\prime}{2}-s^+-s^-\Bigr)
\]
with $\triangleright$ the right $U(\mathfrak{g})\otimes U(\mathfrak{g})$-action on $\textup{Hom}(M,M^\prime\otimes V)$ defined by
\[
\Psi\triangleright (a\otimes b):=(S(a)\otimes 1)\Psi b.
\]
The proof now immediately follows.
\end{proof}
%%%%%%%%%%%%%%%%%
\begin{rema}
(i) There is a natural version of the lemma with the ground field extended to $\mathcal{Q}$, which will be used in the proof of Theorem \ref{KZBthmb}.\\
(ii) Note that $s^{\pm}=\frac{1}{2}(s^\prime\pm s)$ if $s^\prime\in \mathcal{Q}\otimes S^2\mathfrak{g}$. 
The special case $s^\prime=0$ was used in the proof of Theorem \ref{KZBthm} with $s=\rr$.
\end{rema}
%%%%%%%%%%%%%%%%%%%%%%%%%
In the present boundary case we apply Lemma \ref{operatorKZB} to $s=\rr$ with a suitably chosen nonzero correction term $s^\prime$. The left hand side of \eqref{factor} then still admits an expression in terms
of the action of the quadratic Casimir up to first order correction terms, cf.
\eqref{Omegar}.

To have boundary conditions that are compatible with the Chevalley involution $\sigma$ 
we choose $s^\prime\in\mathcal{Q}\otimes\mathfrak{g}^{\otimes 2}$ in such a manner that $s^\prime_{21}+\rr$ lies in 
$\mathcal{Q}\otimes\mathfrak{k}\otimes\mathfrak{g}$. A natural choice is
\[
(\rr_{21})^{\sigma_2}=\frac{1}{2}\sum_{j=1}^nx_j\otimes x_j+\sum_{\alpha\in R}\frac{e_\alpha\otimes e_\alpha}{1-\xi_{-\alpha}}
\in\mathcal{Q}\otimes S^2\mathfrak{g}.
\]
In this case we denote $s^{+}$ and $s^-$ by $\rr^+$ and $\rr^-$. They are explicitly given by
\begin{equation}\label{rplusmin}
\begin{split}
\rr^+:=&\frac{1}{2}\bigl(\rr^{\sigma_1}+\rr\bigr)=\frac{1}{2}\sum_{\alpha\in R}\frac{(e_\alpha-e_{-\alpha})\otimes e_\alpha}{1-\xi_{-\alpha}},\\
\rr^-:=&\frac{1}{2}\bigl((\rr_{21})^{\sigma_2}-\rr\bigr)=\frac{1}{2}\sum_{j=1}^nx_j\otimes x_j+\frac{1}{2}\sum_{\alpha\in R}\frac{(e_\alpha+e_{-\alpha})\otimes e_\alpha}{1-\xi_{-\alpha}}.
\end{split}
\end{equation} 
Note that $(\rr_{21})^{\sigma_2}\in\mathcal{Q}\otimes S^2\mathfrak{g}$, so $(\rr_{21})^{\sigma_2}=\rr^{\sigma_1}$ and $\rr^{\pm}=\frac{1}{2}\bigl(\rr^{\sigma_1}\pm \rr\bigr)$.

Note that the restricted root system for the symmetric pair
$(G,K)$ is $2R$, since the underlying involution is the Chevalley involution $\sigma$.  
To compensate for the factor $2$ we define for $s\in\mathcal{Q}\otimes \mathfrak{g}^{\otimes 2}$, viewed as rational function on $T$ with values in $\mathfrak{g}^{\otimes 2}$,
the element $\widehat{s}\in\mathcal{Q}\otimes\mathfrak{g}^{\otimes 2}$ by
\[
\widehat{s}(t):=s(t^2)\qquad (t\in T).
\]

When removing boundary terms in the present context, the role of Lemma \ref{rsymmlemma} is replaced by the following
%%%%%%%%%%%%%%%%%%%%%%%%%%%%%%%%%%%%
\begin{lem}\label{boundaryrsymmlemma}
Set
\begin{equation}\label{zformula}
z:=\sum_{\alpha\in R^+}\frac{(e_\alpha-e_{-\alpha})\otimes (e_\alpha+e_{-\alpha})}{\xi_\alpha-\xi_{-\alpha}}\in\mathcal{Q}\otimes\mathfrak{k}\otimes\mathfrak{p}.
\end{equation}
Then 
\begin{equation}\label{decompr}
\widehat{\rr}=\frac{1}{2}\Bigl(-\sum_{j=1}^nx_j\otimes x_j+z^{\eta_2^{-1}}-(z_{21})^{\eta_2^{-1}}\Bigr),
\end{equation}
and hence
\begin{equation}\label{boundaryr}
\widehat{\rr^{+}}=\frac{1}{2}z^{\eta_2^{-1}},\qquad \widehat{\rr^{-}}=\frac{1}{2}\Bigl(\sum_{j=1}^nx_j\otimes x_j+(z_{21})^{\eta_2^{-1}}\Bigr).
\end{equation}
\end{lem}
%%%%%%%%%%%%%%%%%%%%%%%%%%%%%%%%%%%%
\begin{proof}
This follows by a direct computation.
\end{proof}
%%%%%%%%%%%%%%%%%%%%%%%%%%%%%%%%%%%%
\begin{rema}
It follows from \eqref{decompr} that
\[
\widehat{\rr}+\frac{1}{2}\sum_{j=1}^nx_j\otimes x_j\in\bigl(\mathfrak{k}_{\mathcal{Q}}\otimes_{\mathcal{Q}}\eta^{-1}(\mathfrak{p}_{\mathcal{Q}})\bigr)\oplus
\bigl(\mathfrak{p}_{\mathcal{Q}}\otimes_{\mathcal{Q}}\eta^{-1}(\mathfrak{k}_{\mathcal{Q}})\bigr)
\]
with $\mathfrak{k}_{\mathcal{Q}}:=\mathcal{Q}\otimes\mathfrak{k}$ and $\mathfrak{p}_{\mathcal{Q}}:=\mathcal{Q}\otimes\mathfrak{p}$. So it lies in the two anti-diagonal blocks
when splitting the tensor product space $\mathfrak{g}_{\mathcal{Q}}\otimes_{\mathcal{Q}}\eta^{-1}(\mathfrak{g}_{\mathcal{Q}})$ along $\mathfrak{g}_{\mathcal{Q}}=\mathfrak{k}_{\mathcal{Q}}\oplus\mathfrak{p}_{\mathcal{Q}}$. The transition $\widehat{\rr}\rightarrow\widehat{\rr^+}$ thus amounts to removing 
the $\bigl(\mathfrak{p}_{\mathcal{Q}}\otimes_{\mathcal{Q}}\eta^{-1}(\mathfrak{k}_{\mathcal{Q}})\bigr)$-block
contribution from $\widehat{\rr}+\frac{1}{2}\sum_{j=1}^nx_j\otimes x_j$.
\end{rema}
%%%%%%%%%%%%%%%%%%%%%%%%%%%%%%%%%%%
\begin{rema}\label{rcmGK}
Lemma \ref{lemmaL}, \eqref{radialcomponent} and Proposition \ref{rL} have natural analogues in the present setup. The $\mathcal{Q}$-linear map \eqref{Lmap} is now replaced by a $\mathcal{Q}$-linear map
$U_{\mathcal{Q}}(\mathfrak{g})\rightarrow \mathcal{D}^\prime$, $X\mapsto D_X$ with $\mathcal{D}^\prime$ the algebra of differential operators on $T$ with coefficients
in $U_{\mathcal{Q}}(\mathfrak{k}\oplus\mathfrak{k})$, and $D_X$ ($X\in U(\mathfrak{g})$) the radial component of the action of $X$ as left-invariant differential operator
on spherical functions on the split real semisimple Lie group $G_{\mathbb{R}}$ (see \cite{CM} and \cite[Thm. 3.4]{RS}). Then the analogue of Proposition \ref{rL} in the present context is 
\[
D_{m(\widehat{\rr})}=-\frac{1}{2}\Delta-\sum_{\alpha\in R^+}\frac{1}{(\xi_\alpha-\xi_{-\alpha})^2}\prod_{\epsilon\in\{\pm 1\}}
\bigl(y_\alpha\otimes 1+\xi_{\epsilon\alpha}(1\otimes y_\alpha)\bigr)
\]
with $y_\alpha:=e_\alpha-e_{-\alpha}\in\mathfrak{k}$, cf. \cite{SR}. See \cite{SR} for a further analysis of the associated quantum spin Calogero-Moser system.
\end{rema}
%%%%%%%%%%%%%%%%%%%%%%%%%%%%%%%%%%%
We are now prepared to consider the analogue of Theorem \ref{KZBthm} for special classes of spherical functions on $G$. We keep the notations and the setup of the previous subsection. In particular we have the composition $\Phi\in\textup{Hom}_G(M_N,M_0\otimes V)$ of $N$ intertwiners $\Phi_j\in\textup{Hom}_G(M_j,M_{j-1}\otimes V_j)$, see
\eqref{Nvertexoperator}. Suppose that
$m_N\in M_N^K$ is a nonzero $K$-fixed vector, and $\phi_0\in M_0^{*,K}$ a nonzero $K$-fixed co-vector in $M_0^*$. 

The {\it $N$-point spherical function}
$f_{\Phi}^{\phi_0,m_N}: G\rightarrow V$ is defined by
\[
f_{\Phi}^{\phi_0,m_N}(g):=\bigl(\phi_0\otimes\textup{id}_V\bigr)\Phi(\pi_N(g)m_N),
\]
cf. \cite[\S 6]{SR}. It is the natural analogue of the $V$-valued trace function $f_\Phi$ (see \eqref{fPhiglobal}) for the symmetric pair $(G,K)$, satisfying trivial transformation behaviour with respect to the right-regular $K$-action while the left-regular $K$-action transforms according to $V$. In \cite{Ob} Oblomkov related special cases of $1$-point spherical functions on complex Grassmannians to $\textup{BC}_n$-type Heckman-Opdam polynomials.

The restriction $f_\Phi^{\phi_0,m_N}\vert_T$ of $f_\Phi^{\phi_0,m_N}$ to $T$ amounts to applying the boundary operation
\begin{equation}\label{bttb}
\Phi\mapsto f_{\Phi}^{\phi_0,m_N}\vert_T=\sum_\mu(\phi_0\otimes\textup{id}_V)\Phi(m_N[\mu])\xi_\mu
\end{equation}
to $\Phi$, where $m_N=\sum_\mu m_N[\mu]$ is the decomposition of $m_N$ in weight components $m_N[\mu]\in M_N[\mu]$. 

We can now formulate the asymptotic boundary KZB equations (it is a special case of  \cite[Prop. 6.13]{SR}). 
%%%%%%%%%%%%%%
\begin{thm}\label{KZBthmb}
Suppose that $M_j$ \textup{(}$0\leq j\leq N$\textup{)} are finite dimensional irreducible $G$-representations.
Let $\Phi$ be a $N$-point $G$-intertwiner \textup{(}see \eqref{Nvertexoperator}\textup{)}, then
\begin{equation}\label{bKZBformula}
\begin{split}
\Bigl(\frac{1}{2}\sum_{j=1}^n(x_j)_{V_i}\partial_{x_j}-\sum_{j=1}^{i-1}\bigl(\widehat{\rr^+}\bigr)_{V_iV_j}-\frac{1}{2}m(\widehat{\rr}^{\,\sigma_1})_{V_i}-
&\sum_{j=i+1}^N\bigl(\widehat{\rr^-}\bigr)_{V_iV_j}+\frac{1}{2}\widehat{y}_{V_i}\Bigr)f_\Phi^{\phi_0,m_N}\vert_T=\\
&\,\,\,\,\,=\frac{1}{4}\bigl(c_{M_i}(\Omega)-c_{M_{i-1}}(\Omega)\bigr)f_\Phi^{\phi_0,m_N}\vert_T
\end{split}
\end{equation}
for $i=1,\ldots,N$.
\end{thm}
%%%%%%%%%%%%%%%%%%%%%%%%%%%%%%%%%%%%
\begin{proof}
In \cite[\S 6]{SR} two proofs are given. The first proof is based on the radial component map for $(G,K)$, mentioned in Remark \ref{rcmGK}. The second proof follows the proof
of Theorem \ref{KZBthm} in the bulk and then uses the $K$-fixed vectors $\phi_0$ and $m_N$ to reflect the boundary terms and merge them back together
(see the proof of \cite[Prop. 6.13]{SR}). We give here a third proof which is based on Lemma \ref{operatorKZB} and Lemma \ref{boundaryrsymmlemma}.

Take $s=\widehat{\rr}$ in Lemma \ref{operatorKZB}. For $s^\prime\in\mathcal{Q}\otimes \mathfrak{g}^{\otimes 2}$ 
set $s^+:=\frac{1}{2}(s^\prime_{21}+\widehat{\rr})$ and $s^-:=\frac{1}{2}(s^\prime-\widehat{\rr})$ to obtain the operator KZB-type equation \eqref{factor}. 
As in the first part of the proof of Theorem \ref{KZBthm} we then get
\begin{equation}\label{stepgeneral}
\begin{split}
\bigl(\frac{1}{4}\bigl(c_{M_{i-1}}(\Omega)-&c_{M_i}(\Omega)\bigr)+\frac{1}{2}\widehat{y}_{V_i}\bigr)\Phi=\\
&=
\Bigl(\sum_{j=1}^{i-1}s_{V_jV_i}^++\frac{1}{2}m(s^\prime)_{V_i}+\sum_{j=i+1}^Ns_{V_iV_j}^-\Bigr)\Phi+\bigl(s_{M_0V_i}^+\Phi+\Phi\ast_is^-\bigr).
\end{split}
\end{equation}

Now take $s^\prime:=(\widehat{\rr}_{21})^{\sigma_2}=\widehat{\rr}^{\,\sigma_1}\in\mathcal{Q}\otimes S^2\mathfrak{g}$, so that $s^{\pm}=\widehat{\rr^\pm}$, and impose the
boundary condition \eqref{bttb}.
Since $\widehat{\rr^+}\in\mathcal{Q}\otimes\mathfrak{k}\otimes\mathfrak{g}$ and $\phi_0$ is $K$-fixed the first of the two boundary terms vanishes,
\begin{equation}\label{todooo1}
\sum_\mu(\phi_0\otimes\textup{id}_V)\bigl(\bigl(\widehat{\rr^+}\bigr)_{M_0V_i}\Phi m_N[\mu]\bigr)\xi_\mu=0.
\end{equation}
For the second boundary term, note first that
\[
\sum_\mu X(m_N[\mu])\xi_\mu=\sum_\mu(\eta(X)m_N)[\mu]\xi_\mu\qquad (X\in U(\mathfrak{g}))
\]
in $\mathcal{Q}\otimes M_N$. Writing $\widehat{\rr^-}=\sum_\ell a_\ell\otimes b_\ell$, it follows that
\[
\sum_\mu(\phi_0\otimes\textup{id}_V)\bigl((\Phi\ast_i\widehat{\rr^-})m_N[\mu]\bigr)\xi_\mu=-\sum_{\ell,\mu}(a_\ell)_{V_i}(\phi_0\otimes\textup{id}_V)\Phi
\bigl((\eta(b_i)m_N)[\mu]\bigr)\xi_\mu.
\]
By \eqref{boundaryr} we have
\[
\sum_\ell a_\ell\otimes\eta(b_\ell)=(\widehat{\rr^-})^{\eta_2}=\frac{1}{2}\Bigl(\sum_{j=1}^nx_j\otimes x_j+z_{21}\Bigr)
\]
and $z_{21}\in\mathcal{Q}\otimes\mathfrak{p}\otimes\mathfrak{k}$, so the fact that $m_N$ is $K$-fixed gives
\begin{equation}\label{todooo}
\sum_\mu(\phi_0\otimes\textup{id}_V)\bigl((\Phi\ast_i\widehat{\rr^-})m_N[\mu]\bigr)\xi_\mu=
-\frac{1}{2}\sum_{j=1}^n(x_j)_{V_i}\partial_{x_j}\bigl(f_{\Phi}^{\phi_0,m_N}\vert_T\bigr).
\end{equation}
Now apply \eqref{bttb} to \eqref{stepgeneral} and substitute \eqref{todooo1} and \eqref{todooo}. 
\end{proof}
%%%%%%%%%%%%%%%%%%%%%%%%%%%%%%%%%%%%%%%%%%%%%%%%%%%%%%
\begin{rema}
The proof of Theorem \ref{KZBthmb} seems to suggest that another choice of $s^\prime$ might lead to an alternative system of 
first order differential equations. But the local factors $s^+:=\frac{1}{2}\bigl(s_{21}^\prime+\widehat{\rr}\bigr)$ and 
$s^-:=\frac{1}{2}(s^\prime-\widehat{\rr})$
only solve the same vanishing boundary conditions \eqref{todooo1} and \eqref{todooo} as $\widehat{\rr^+}$ and $\widehat{\rr^-}$ 
if $s^\prime=\widehat{\rr}^{\,\sigma_1}+s^{\prime\prime}$ with
\[
s^{\prime\prime}\in\bigl(\mathfrak{g}_{\mathcal{Q}}\otimes_{\mathcal{Q}}\mathfrak{k}_{\mathcal{Q}}\bigr)\cap
\bigl(\mathfrak{g}_{\mathcal{Q}}\otimes_{\mathcal{Q}}\eta^{-1}(\mathfrak{k}_{\mathcal{Q}})\bigr).
\]
The latter space is $\{0\}$ due to the infinitesimal Cartan decomposition
\[
\mathfrak{g}_{\mathcal{Q}}=\eta^{-1}(\mathfrak{k}_{\mathcal{Q}})\oplus\mathfrak{h}_{\mathcal{Q}}\oplus
\mathfrak{k}_{\mathcal{Q}}.
\]
The infinitesimal Cartan decomposition plays an important role in computing the radial components of $G$-invariant differential operators, see \cite{CM,SR,RS}.
\end{rema}
%%%%%%%%%%%%%%%%%%%%%%%%%%%%%%%%

In \cite{SR} sufficiently many solutions of the asymptotic boundary KZB equations \eqref{bKZBformula} were constructed to conclude consistency. The consistency translates in four explicit equations for its local factors $(\rr^{\,+},\rr^{\,-},m(\rr^{\,\sigma_1})/2)$, viewed as functions on $\mathfrak{h}^*$ (cf. Remark \ref{remarkrrr}). Three equations only involve
$(\rr^{\,+},\rr^{\,-})$ and are coupled versions of the classical dynamical Yang-Baxter equations. The fourth equation is a classical dynamical reflection
type equation for $m(\rr^{\,\sigma_1})/2$ relative to $(\rr^{\,+},\rr^{\,-})$ (see \cite[\S 6]{SR} for details). 

One of the goals of this paper is to obtain these four coupled classical dynamical Yang-Baxter and reflection equations by direct algebraic manipulations from the classical dynamical Yang-Baxter equation satisfied by $\rr$, and to reveal the basic underlying algebraic principles.

%%%%%%%%%%%%%%%%%%%%%%%%%%%%%%%%%

%%%%%%%%%%%%%%%%%%%%%%%%%
\subsection{Limits of KZB equations}\label{DegScheme}
%%%%%%%%%%%%%%%%%%%%%%%%%
Denote by $U(\widetilde{\mathfrak{g}})$ the universal enveloping algebra of the affine Lie algebra $\widetilde{\mathfrak{g}}$. 
Let $\widetilde{C}$ be the quadratic Casimir
element of $U(\widetilde{\mathfrak{g}})$. 
The asymptotic (boundary) KZB equations fit in the following degeneration scheme:
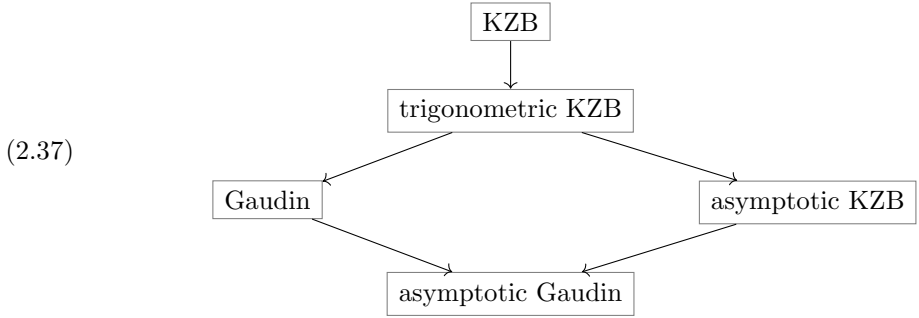
\vspace{.5cm}
\begin{equation}\label{dscheme}
\begin{tikzcd}[cells={nodes={draw=gray}}]
& \textup{KZB} \arrow[d, black] & \\
& \textup{trigonometric KZB}\arrow[ld,black]\arrow[rd,black,swap] & \\
\textup{Gaudin}\arrow[rd,black] & & \textup{asymptotic KZB}\arrow[ld,black,swap]\\
& \textup{asymptotic Gaudin} &
\end{tikzcd}
\end{equation}
\vspace{.3cm}\\
The box at the top level of the degeneration scheme contains the KZB equations and their boundary versions. The KZB equations form a consistent system of first order differential equations that describe how the different local insertions of $\widetilde{C}$ within $N$-point correlation functions for Bernard's \cite{B,FW} extension of the Wess-Zumino-Witten (WZW) conformal field theory to the elliptic curve are related. The $N$-point correlation functions can be expressed as restrictions of generalised trace functions on the affine Lie group (see, e.g., \cite{E,ES,EFK}). When conformally invariant boundary conditions are imposed, the $N$-point correlation functions are restrictions of affine spherical functions on the affine Lie group instead, and the consistency equations are called boundary KZB equations. The boundary case will be discussed in detail in a forthcoming paper by N. Reshetikhin and the author.

The (boundary) KZB operators are commuting first order differential operators acting on functions $f: \mathfrak{h}^*\times (\mathbb{C}^\times)^N\rightarrow V:=V_1\otimes\cdots\otimes V_N$ of the 
form
\begin{equation}\label{KZBoper}
(c+h^\vee)z_i\frac{\partial}{\partial{z_i}}+\sum_{j=1}^n(x_j)_{V_i}\partial_{\lm_j}+A_i^{(N)}(\cdot\vert\tau),\qquad 1\leq i\leq N.
\end{equation}
Here $\{\lm_j\}_{j=1}^n$ is the linear basis of $\mathfrak{h}^*$ dual to $\{x_j\}_{j=1}^n$, $\partial_{\lm_j}$ is the directional derivative in direction $\lm_j\in\mathfrak{h}^*$,
$h^\vee$ is the dual Coxeter number, $c\in\mathbb{C}$ the level, $\tau$ lies in the upper half plane (fixing an elliptic curve $\mathbb{C}^\times/p^{\mathbb{Z}}$ with $p:=e^{2\pi i\tau}$),
$z=(z_1,\ldots,z_N)\in (\mathbb{C}^\times)^N$ are the spectral parameters, $\partial/\partial z_i$ the associated partial derivatives, and $A_i^{(N)}(\cdot\vert\tau): \mathfrak{h}^*\times (\mathbb{C}^\times)^N\rightarrow \textup{End}(V)$ is constructed in terms of Felder's \cite{F} classical elliptic dynamical $r$-matrix $\rr_{\textup{ell}}: \mathfrak{h}^*\rightarrow\mathfrak{g}\otimes\mathfrak{g}$ with spectral parameter (we write $\rr_{\textup{ell}}$ explicitly down in the following paragraph). The explicit expressions for $A_i^{(N)}(\cdot\vert\tau)$ in case of the KZB equations are given in \cite{F}. For the boundary KZB equations the building blocks of $A_i^{(N)}(\cdot\vert\tau)$ are folded and contracted versions of $\rr_{\textup{ell}}$, with the folding and contraction relative to the Chevalley involution of the affine Lie algebra. This will be discussed in a forthcoming paper by N. Reshetikhin and the author.
The spectral parameters arise representation theoretically from the realisation of the spin space $V$ as a $N$-fold tensor product $V_1(z_1)\otimes\cdots\otimes V_N(z_N)$ of evaluation representations of the affine Lie algebra. Note that under the identification $\mathfrak{h}\simeq\mathfrak{h}^*$ via the Killing form, $\partial_{\lm_j}$ corresponds to $\partial_{x_j}$. 

The elliptic classical dynamical $r$-matrix $\rr_{\textup{ell}}(z;\cdot\vert\tau): \mathfrak{h}^*\rightarrow \mathfrak{g}^{\otimes 2}$ with spectral parameter $z\in\mathbb{C}^\times$ is explicitly given by 
\[
\rr_{\textup{ell}}(e^{2\pi it};\lm\vert\tau):=\frac{\rho(t\vert\tau)}{2\pi i}\sum_{j=1}^nx_j\otimes x_j+\frac{1}{2\pi i}\sum_{\alpha\in R}
\sigma\Bigl(\frac{(\alpha,\lm)}{2\pi i},t\vert\tau\Bigr)e_{-\alpha}\otimes e_\alpha
\]
with
\[
\rho(t\vert\tau):=\frac{\theta_1^\prime(t\vert\tau)}{\theta_1(t\vert\tau)},\qquad \sigma(w,t\vert\tau):=\frac{\theta_1(w-t\vert\tau)\theta_1^\prime(0\vert\tau)}{\theta_1(w\vert\tau)\theta_1(t\vert\tau)},
\]
see \cite{F}. Here $\theta_1(t\vert\tau)$ is Jacobi's theta function,
\[
\theta_1(t\vert\tau):=ie^{\pi i\tau/4}e^{-\pi it}\prod_{k=0}^{\infty}(1-p^k)(1-p^ke^{2\pi it})(1-p^{k+1}e^{-2\pi it})
\]
and $\theta_1^\prime(t\vert\tau):=\frac{d}{dt}\theta_1(t\vert\tau)$. See \cite[\S 2]{F} for the classical dynamical Yang-Baxter equation satisfied by
$\rr_{\textup{ell}}$. 

The limit from trigonometric (boundary) KZB operators to trigonometric (boundary) KZB operators is $\Im(\tau)\rightarrow\infty$. The role of $\rr_{\textup{ell}}(z;\lm\vert\tau)$ is then taken over by the trigonometric classical dynamical $r$-matrix
\begin{equation*}
\rr_{\textup{trig}}(z;\lm):=
\lim_{\Im(\tau)\rightarrow\infty}\rr_{\textup{ell}}(z;\lm\vert\tau)=\frac{1}{2}\Bigl(\frac{z+1}{z-1}\Bigr)\sum_{j=1}^nx_j\otimes x_j+
\sum_{\alpha\in R}\frac{(z-e^{(\alpha,\lm)})}{(z-1)(1-e^{(\alpha,\lm)})}e_{-\alpha}\otimes e_\alpha
\end{equation*}
with spectral parameter, cf. \cite[\S 4.3]{EV}. 

The limit from trigonometric (boundary) KZB operators to (boundary) Gaudin Hamiltonians is as follows.
Write $\mathfrak{h}_{\mathbb{R}}^*$ for the real subspace of $\mathfrak{h}^*$ generated by $R$. It is a real form of $\mathfrak{h}^*$.
In the asymptotic region of $\mathfrak{h}^*$ where the real part of $\lm$ tends to infinity in the Weyl chamber opposite to the fundamental Weyl chamber relative to $R^+$
(i.e., $\Re((\alpha,\lm))\rightarrow-\infty$ for $\alpha\in R^+$), the $r$-matrix
$\rr_{\textup{trig}}(z;\cdot)$ reduces to the Belavin-Drinfeld \cite{BD1,BD2} classical trigonometric $r$-matrix 
\[
\rr_{\textup{BD}}(z):=\frac{z\Omega_-+\Omega_+}{z-1}
\]
with spectral parameter, where $\Omega_{\pm}:=\frac{1}{2}\sum_{j=1}^nx_j\otimes x_j+\sum_{\alpha\in R^{\pm}}e_\alpha\otimes e_{-\alpha}$
are the half-Casimirs. In this asymptotic region the trigonometric (boundary) KZB operators reduce to a family of commuting Gaudin type Hamiltonians $B_j^{(N)}(z)$ ($1\leq j\leq N$) on $V$, which are of type $C_N$ in the boundary case. Gaudin Hamiltonians of this form appeared before in, e.g., \cite{Skr1,Skr2}. 

The limit from trigonometric KZB operators to asymptotic KZB operators is $|z_i/z_{i+1}|\rightarrow 0$ (and, in addition, $|z_N|\rightarrow 0$ in the boundary case).
The underlying limit of $r$-matrices is 
\[
\lim_{|z|\rightarrow 0}\rr_{\textup{trig}}(z;\lm)=\rr(\lm)
\]
with $\rr$ Felder's \cite{F} classical dynamical $r$-matrix \eqref{Frrr}. 

Finally, asymptotic Gaudin Hamiltonians are obtained from Gaudin Hamiltonians by the limit $|z_i/z_{i+1}|\rightarrow 0$ and $|z_N|\rightarrow 0$, 
and from the asymptotic KZB operators when the real part of $\lm\in\mathfrak{h}^*$ tends to infinity in the Weyl chamber opposite to the positive Weyl chamber. The corresponding limits of $r$-matrices are
\begin{equation}\label{aGaLimit}
\lim_{\lm\rightarrow-\infty}\rr(\lm)=-\Omega_+=\lim_{|z|\rightarrow 0}\rr_{\textup{BD}}(z),
\end{equation}
where $\lm\rightarrow -\infty$ stands for $\Re(\lm)\in\mathfrak{h}_{\mathbb{R}}^*$ tending to infinity in the Weyl chamber opposite to the positive Weyl chamber. See Section \ref{AsGa} for a direct approach to asymptotic Gaudin type Hamiltonians of type $C_N$.

%%%%%%%%%%%%%%%%%%%%%%%%%%%%%%%%%%
\section{Commuting first order differential operators}\label{CommSection}
%%%%%%%%%%%%%%%%%%%%%%%%%%%%%%%%%%
Let $A_\ell, U, A_r$ be three unital complex associative algebras. 
Let $\mathfrak{a}\subseteq U$ be a finite dimensional linear subspace. 
Let $\mathcal{M}$ be the field of meromorphic functions on $\mathfrak{a}^*$. We call elements in the algebraic tensor product $\mathcal{M}\otimes U$ meromorphic $U$-valued functions on $\mathfrak{a}^*$, and denote them as maps $\mathfrak{a}^*\rightarrow U$. Functions are assumed to be meromorphic unless stated explicitly otherwise.

Let $N\in\mathbb{Z}_{>0}$. If $\{i_1,\ldots,i_k\}\subseteq\{1,\ldots,N\}$ and $z\in U^{\otimes k}$ then we write $z_{i_1i_2\cdots i_k}\in U^{\otimes N}$ for the
element obtained from $z$ by placing the $j$-th tensor component of $z$ in the $i_j$-th tensor component and placing ones everywhere else
($N$ is implicit from context). For instance, if $z=\sum_ia_i\otimes b_i\in U^{\otimes 2}$ then $z_{31}=\sum_ib_i\otimes 1\otimes a_i$. 
We use the same convention when placing elements from $A_\ell\otimes U^{\otimes k}\otimes A_r$ in $A_\ell\otimes U^{\otimes N}\otimes A_r$, omitting 
sublabels for the $A_\ell$ and $A_r$-components unless confusion may arise (which may happen, for instance, when $A_\ell=A_r$ or $A_\ell=U=A_r$).

Denote by $\mathbb{D}_N$ the algebra of differential operators on $\mathfrak{a}^*$ with coefficients in the ring of $A_\ell\otimes U^{\otimes N}\otimes A_r$-valued meromorphic functions on $\mathfrak{a}^*$. We write $\partial_\lm\in\mathbb{D}_N$ for the directional derivative in direction $\lm\in\mathfrak{a}^*$. 
The gradient type first order differential terms occurring in the differential operators we define below, are defined as follows.
%%%%%%%%%%%%%%%%%%%%%%%%%%%%%%%%%%%%%%%%%%
\begin{defi}
We define
\begin{equation}\label{Ei}
E_i:=\sum_{j=1}^d(x_j)_i\partial_{\lm_j}\in\mathbb{D}_N\qquad (1\leq i\leq N),
\end{equation}
where $\{x_j\}_{j=1}^\dd$ is a linear basis of $\mathfrak{a}$ and $\{\lm_k\}_{k=1}^d$ is the corresponding dual basis of $\mathfrak{a}^*$. 
\end{defi}
%%%%%%%%%%%%%%%%%%%%%%%%%%%%%%%%%%%%%%%%%%%%
We used here the conventions of the previous paragraph, so 
\[
(x_j)_i=1_{A_\ell}\otimes 1_{U}^{\otimes (i-1)}\otimes x_j\otimes 1_U^{\otimes (N-i)}\otimes 1_{A_r}.
\]
Note that $E_i$ is independent of the choice of basis
$\{x_j\}_{j=1}^d$ of $\mathfrak{a}$.

We consider now a special class of first order differential operators in $\mathbb{D}_N$
whose constant terms are sums of local tensor product contributions $r^{\pm}: \mathfrak{a}^*\rightarrow U^{\otimes 2}$ and $\kappa: \mathfrak{a}^*\rightarrow A_\ell\otimes U\otimes A_r$:
%%%%%%%%%%%%%%%%%%%%%%%%%%%%%%%%%%%%%%%%%%%%%%%%
\begin{defi}\label{Di}
For meromorphic functions $r^{\pm}: \mathfrak{a}^*\rightarrow U^{\otimes 2}$ and 
$\kappa: \mathfrak{a}^*\rightarrow A_\ell\otimes U\otimes A_r$ define the first-order differential
operator $\mathcal{D}_i^{(N)}\in\mathbb{D}_N$ \textup{(}$1\leq i\leq N$\textup{)} by
\[
\mathcal{D}_i^{(N)}:=E_i-
\sum_{s=1}^{i-1}r_{si}^+-\kappa_i-\sum_{s=i+1}^Nr_{is}^-
\]
\textup{(}we suppress the dependence on $r^{\pm}$ and $\kappa$ in the notation for $\mathcal{D}_i^{(N)}$\textup{)}.
\end{defi}
%%%%%%%%%%%%%%%%%%%%%%%%%%%%%%%%%%%%%%%%%%%%%%%%

The commutators $[\mathcal{D}_i^{(N)},\mathcal{D}_j^{(N)}]$ in $\mathbb{D}_N$ can be expressed in terms of meromorphic functions $\textup{CYB}[t](r^+,r^-): \mathfrak{a}^*\rightarrow U^{\otimes 3}$ ($1\leq t\leq 3$) and $\textup{CR}(r^+,r^-;\kappa): \mathfrak{a}^*\rightarrow A_\ell\otimes U^{\otimes 2}\otimes A_r$, defined by
\begin{equation}\label{CYB}
\begin{split}
\textup{CYB}[1](r^+,r^-)&:=[r_{12}^+,r_{13}^+]+[r_{12}^+,r_{23}^+]-[r_{13}^+,r_{23}^-]+E_3(r_{12}^+)-E_2(r_{13}^+),\\
\textup{CYB}[2](r^+,r^-)&:=[r_{12}^-,r_{13}^+]+[r_{12}^-,r_{23}^+]+[r_{13}^-,r_{23}^+]
+E_3(r_{12}^-)-E_1(r_{23}^+),\\
\textup{CYB}[3](r^+,r^-)&:=-[r_{12}^+,r_{13}^-]+[r_{12}^-,r_{23}^-]+[r_{13}^-,r_{23}^-]+
E_2(r_{13}^-)-E_1(r_{23}^-)
\end{split}
\end{equation}
and 
\begin{equation}\label{CR}
\textup{CR}(r^+,r^-;\kappa):=[\kappa_1+r^-,\kappa_2+r^+]+
E_2(r^-+\kappa_1)-E_1(r^++\kappa_2).
\end{equation}
We suppress $r^{\pm},\kappa$ from the notations if no confusion can arise.
%%%%%%%%%%%%%%%%%%%%%%%%%%%%%%%%%%%%%%%%%%%%%%
\begin{prop}\label{commutatorprop}
Let $N\geq 2$ and $1\leq i<j\leq N$. We have
\begin{equation*}
\begin{split}
[\mathcal{D}_i^{(N)},\mathcal{D}_j^{(N)}]&=
\sum_{k=1}^\dd\Bigl([(x_k)_j,r_{ij}^-]-[(x_k)_i,r_{ij}^+]\Bigr)\partial_{\lm_k}+\textup{CR}_{ij}
\\
&+\sum_{s=1}^{i-1}\textup{CYB}[1]_{sij}
+\sum_{s=i+1}^{j-1}\textup{CYB}[2]_{isj}+
\sum_{s=j+1}^N\textup{CYB}[3]_{ijs}
\end{split}
\end{equation*}
in $\mathbb{D}_N$.
\end{prop}
%%%%%%%%%%%%%%%%%%%%%%%%%%%%%%%%%%%%%%%%%%%%%%%
\begin{proof}
This follows by a direct computation (compare with the proof of \cite[Thm. 6.12]{RS}). 
\end{proof}
%%%%%%%%%%%%%%%%%%%%%%%%%%%%%%%%%%%%%%%%%%%%%%%
As an immediate consequence, we have
%%%%%%%%%%%%%%%%%%%%%%%%%%%%%%%%%%%%%%%%%%%%%%%
\begin{cor}\label{commutarotcor}
The following statements are equivalent:
\begin{enumerate}
\item $[\mathcal{D}_i^{(N)},\mathcal{D}_j^{(N)}]=0$ in $\mathbb{D}_N$ for all $N\geq 2$ and all $1\leq i,j\leq N$.
\item $r^{\pm}$ and $\kappa$ satisfy 
\[
[x\otimes 1,r^+]=[1\otimes x,r^-]\qquad \forall\, x\in\mathfrak{a}
\]
\textup{(}$\mathfrak{a}$-compatibility\textup{)}, 
\[
\textup{CYB}[t](r^+,r^-)=0,\qquad 1\leq t\leq 3
\]
\textup{(}the coupled classical dynamical Yang-Baxter equations\textup{)} and
\[
\textup{CR}(r^{+},r^-;\kappa)=0
\]
\textup{(}the classical dynamical reflection equation for $\kappa$ relative to $(r^+,r^-)$\textup{)}.
\end{enumerate}
\end{cor}
%%%%%%%%%%%%%%%%%%%%%%%%%%%%%%%%%%%%%%%%%%%%%%%%
The classical dynamical reflection equation relative to $(r^+,r^-)$ is a classical dynamical analog of the compatibility condition for factorised scattering of quantum particles on a half-line bouncing off the wall, known as the reflection equation \cite{C}. We call a pair $(r^+,r^-)$ satisfying 
$\textup{CYB}[t](r^+,r^-)=0$
($1\leq t\leq 3$) a {\it classical dynamical $r$-matrix pair}, and we call $\kappa$ satisfying 
$\textup{CR}(r^{+},r^-;\kappa)=0$ a {\it classical dynamical $k$-matrix} relative to $(r^+,r^-)$.

We now look at the special case when $r^-=-r^+$, in which case the usual classical dynamical Yang-Baxter equation and classical dynamical reflection equation naturally appear.  For a meromorphic function $r: \mathfrak{a}^*\rightarrow U^{\otimes 2}$
define $\textup{YB}(r): \mathfrak{a}^*\rightarrow U^{\otimes 3}$ by
\begin{equation}\label{YB}
\textup{YB}(r):=[r_{12},r_{13}]+[r_{12},r_{23}]+[r_{13},r_{23}]+E_1(r_{23})-E_2(r_{13})+E_3(r_{12}).
\end{equation}
If in addition $\kappa: \mathfrak{a}^*\rightarrow A_\ell\otimes U\otimes A_r$ is a meromorphic function, then define the function
$\textup{R}(r;\kappa): \mathfrak{a}^*\rightarrow A_\ell\otimes U^{\otimes 2}\otimes A_r$ by
\[
\textup{R}(r;\kappa):=[\kappa_1,\kappa_2+r]+[\kappa_1-r,\kappa_2]+E_2(\kappa_1)-E_1(\kappa_2).
\]
%%%%%%%%%%%%%%%%%%%%%%%%%%%%%%%%%%%%%%%%%%%
\begin{lem}\label{lemAred}
For $r: \mathfrak{a}^*\rightarrow U^{\otimes 2}$ and $\kappa: \mathfrak{a}^*\rightarrow
A_\ell\otimes U\otimes A_r$ we have
\begin{equation*}
\textup{YB}(r)=\textup{CYB}[1](r,-r)+E_1(r_{23})
=-\textup{CYB}[2](r,-r)-E_2(r_{13})
=\textup{CYB}[3](r,-r)+E_3(r_{12})
\end{equation*}
as meromorphic functions $\mathfrak{a}^*\rightarrow U^{\otimes 3}$, and 
\[
\textup{R}(r;\kappa)=\textup{CR}(r,-r;\kappa)+E_1(r)+E_2(r)
\]
as meromorphic functions $\mathfrak{a}^*\rightarrow A_\ell\otimes U^{\otimes 2}\otimes A_r$.
\end{lem}
%%%%%%%%%%%%%%%%%%%%%%%%%%%%%%%%%%%%%%%%%%
\begin{proof}
This follows by a direct computation.
\end{proof}
%%%%%%%%%%%%%%%%%%%%%%%%%%%%%%%%%%%%%%%%%

Define first order differential operators $L_i^{(N)}\in\mathbb{D}_N$ ($1\leq i\leq N$) by
\[
L_i^{(N)}:=
E_i-\sum_{s=1}^{i-1}r_{si}-\kappa_i+\sum_{s=i+1}^N
r_{is}.
\]
%%%%%%%%%%%%%%%%%%%%%%%%%%%%%%%%%%%%%%%%%%%%
\begin{cor}\label{typeAcor}
The following statements are equivalent.
\begin{enumerate}
\item  For all $N\geq 2$ and $1\leq i,j\leq N$,
\[
[L_i^{(N)},L_j^{(N)}]=-\sum_{s=1}^NE_s(r_{ij})
\]
in $\mathbb{D}_N$.
\item $r$ and $\kappa$ satisfy
\[
[x\otimes 1+1\otimes x,r]=0\qquad \forall x\in\mathfrak{a}
\]
\textup{(}$\mathfrak{a}$-invariance\textup{)},
\[
\textup{YB}(r)=0
\]
\textup{(}classical dynamical Yang-Baxter equation\textup{)} and
\[
\textup{R}(r;\kappa)=0
\]
\textup{(}classical dynamical reflection equation relative to $r$\textup{)}.
\end{enumerate}
\end{cor}
%%%%%%%%%%%%%%%%%%%%%%%%%%%%%%%%%%%%%%%%%%%
\begin{proof}
Note that $L_i^{(N)}$ is $\mathcal{D}_i^{(N)}$ with $(r^+,r^-)$ specialised to $(r,-r)$.
Hence Proposition \ref{commutatorprop} gives
\begin{equation}\label{step}
\begin{split}
[L_i^{(N)},L_j^{(N)}]+\sum_{s=1}^NE_s(r_{ij})
&=-\sum_{k=1}^\dd[(x_k)_i+(x_k)_j,r_{ij}]\partial_{\lm_k}
+\underline{\textup{CR}}_{ij}\\
&+\sum_{s=1}^{i-1}\underline{\textup{CYB}}[1]_{sij}
+\sum_{s=i+1}^{j-1}\underline{\textup{CYB}}[2]_{isj}+
\sum_{s=j+1}^N\underline{\textup{CYB}}[3]_{ijs}
\end{split}
\end{equation}
in $\mathbb{D}_N$ with 
\begin{equation*}
\begin{split}
\underline{\textup{CR}}:=&\textup{CR}(r,-r;\kappa)+E_1(r)+E_2(r),\\
\underline{\textup{CYB}}[t]:=&
\textup{CYB}[t](r,-r)+E_t(r_{pq})
\end{split}
\end{equation*}
for $1\leq t\leq 3$, where $1\leq p<q\leq 3$ is such that $\{p,q,t\}=\{1,2,3\}$. By 
Lemma \ref{lemAred}, formula
\eqref{step} thus reduces to
\begin{equation*}
\begin{split}
[L_i^{(N)},L_j^{(N)}]+\sum_{s=1}^NE_s(r_{ij})
&=-\sum_{k=1}^\dd[(x_k)_i+(x_k)_j,r_{ij}]\partial_{\lm_k}
+\textup{R}(r;\kappa)_{ij}\\
&+\sum_{s=1}^{i-1}\textup{YB}(r)_{sij}
-\sum_{s=i+1}^{j-1}\textup{YB}(r)_{isj}+
\sum_{s=j+1}^N\textup{YB}(r)_{ijs}
\end{split}
\end{equation*}
in $\mathbb{D}_N$. The result now follows immediately.
\end{proof}
%%%%%%%%%%%%%%%%%%%%%%%%%%%%%%%%%%%%%%%%%
A meromorphic solution $r: \mathfrak{a}^*\rightarrow U$ of the classical dynamical Yang-Baxter equation is called a {\it classical dynamical $r$-matrix}. For $(U,\mathfrak{a})=(U(\mathfrak{g}),\mathfrak{h})$ as in Section \ref{Se1}, $\rr:\mathfrak{h}^*\rightarrow\mathfrak{g}\otimes\mathfrak{g}$ given by \eqref{Frrr} is an example of an $\mathfrak{h}$-invariant  classical dynamical $r$-matrix.
A meromorphic solution $\kappa: \mathfrak{a}^*\rightarrow A_\ell\otimes U\otimes A_r$ of the classical dynamical reflection equation $\textup{R}(r;\kappa)=0$ 
is called a {\it classical dynamical $k$-matrix} relative to $r$.
%%%%%%%%%%%%%%%%%%%%%%%%%%%%%%%%%
\begin{rema}\label{rrRemark}
(1) 
Corollary \ref{typeAcor} for the trivial classical dynamical $k$-matrix $\kappa\equiv 0$ reduces to \cite[Prop. 6.27]{SR} (which was formulated there for specific choices of $A_\ell, A_r, U$ and $\mathfrak{a}$).\\
(2) For $(U,\mathfrak{a},A_\ell,A_r,r,\kappa)=(U(\mathfrak{g}),\mathfrak{h},\mathbb{C},\mathbb{C},\rr,y)$ with $y: \mathfrak{h}^*\rightarrow \mathfrak{g}$ obtained from \eqref{y}
through the canonical isomorphism $\mathfrak{h}\simeq\mathfrak{h}^*$, 
\begin{equation}\label{ynew}
y(\lm)=\frac{1}{2}\sum_{\alpha\in R^+}\Bigl(\frac{1+e^{-(\alpha,\lm)}}{1-e^{-(\alpha,\lm)}}\Bigr)t_\alpha,
\end{equation}
the asymptotic KZB equations \eqref{KZBformula} state that generalised trace functions provide common eigenfunctions of the commuting operators $L_i^{(N)}$ ($1\leq i\leq N$). Note that indeed $y$ is a classical dynamical $k$-matrix relative to $\rr$ since $\rr$ is $\mathfrak{h}$-invariant and 
\begin{equation}\label{yadd}
E_1(y_2)=\sum_{\alpha\in R^+}\frac{t_\alpha\otimes t_\alpha}{(1-e^\alpha)(1-e^{-\alpha})}=E_2(y_1)
\end{equation}
with $e^\alpha: \mathfrak{h}^*\rightarrow\mathbb{C}$ the exponential map $\lm\mapsto e^{(\alpha,\lm)}$.
\end{rema} 
%%%%%%%%%%%%%%%%%%%%%%%%%%%%%%%%%%%%%%%%%%
%%%%%%%%%%%%%%%%%%%%%%%%%%%%%%%%%%%%%%%%%
\section{Some properties of $\textup{CYB}[t](r^+,r^-)$ and $\textup{CR}(r^+,r^-;\kappa)$}\label{S2}
%%%%%%%%%%%%%%%%%%%%%%%%%%%%%%%%%%%%%%%%%

Let $r^{\pm}: \mathfrak{a}^*\rightarrow U^{\otimes 2}$ be two meromorphic $U^{\otimes 2}$-valued functions.
Consider
\begin{equation}\label{vYB}
\textup{CYB}(r^+,r^-):=[r_{12}^-,r_{13}^-]-[r_{12}^+,r_{23}^-]+[r_{13}^+,r_{23}^+],
\end{equation}
viewed as $U^{\otimes 3}$-valued meromorphic function on $\mathfrak{a}^*$.
%%%%%%%%%%%%%%%%%%%%%%%%%%%%%%%%%%%%%%%%%%%%
\begin{lem}\label{lemtoYB}
We have
\begin{equation}\label{relationformula}
\textup{YB}(r^+-r^-)=\textup{CYB}[1](r^+,r^-)-\textup{CYB}[2](r^+,r^-)+\textup{CYB}[3](r^+,r^-)
+\textup{CYB}(r^+,r^-).
\end{equation}
\end{lem}
%%%%%%%%%%%%%%%%%%%%%%%%%%%%%%%%%%%%%%%%%%%%
\begin{proof}
Write $r:=r^+-r^-$.
The differential contribution to the right hand side of \eqref{relationformula} is
\begin{equation*}
-\bigl(E_2(r_{13}^+)-E_3(r_{12}^+)\bigr)+\bigl(E_1(r_{23}^+)-E_3(r_{12}^-)\bigr)-\bigl(E_1(r_{23}^-)-E_2(r_{13}^-)\bigr)
\end{equation*}
It simplifies to 
\[
E_1(r_{23})-E_2(r_{13})+E_3(r_{12}),
\]
which is the differential contribution to $\textup{YB}(r)$. Hence it remains to show that
\begin{equation*}
\textup{YB}_0(r)=\textup{CYB}_0[1](r^+,r^-)-\textup{CYB}_0[2](r^+,r^-)+\textup{CYB}_0[3](r^+,r^-)+\textup{CYB}(r^+,r^-)
\end{equation*}
with
\begin{equation}\label{zero}
\begin{split}
\textup{YB}_0(r):=&[r_{12},r_{13}]+[r_{12},r_{23}]+[r_{13},r_{23}],\\
\textup{CYB}_0[1](r^+,r^-):=&[r_{12}^+,r_{13}^+]+[r_{12}^+,r_{23}^+]-[r_{13}^+,r_{23}^-],\\
\textup{CYB}_0[2](r^+,r^-):=&[r_{12}^-,r_{13}^+]+[r_{12}^-,r_{23}^+]+[r_{13}^-,r_{23}^+],\\
\textup{CYB}_0[3](r^+,r^-):=&-[r_{12}^+,r_{13}^-]+[r_{12}^-,r_{23}^-]+[r_{13}^-,r_{23}^-]
\end{split}
\end{equation}
the non-dynamical terms of $\textup{YB}(r)$ and $\textup{CYB}[t](r^+,r^-)$ ($1\leq t\leq 3$).
This follows by a direct check. 
\end{proof}
%%%%%%%%%%%%%%%%%%%%%%%%%%%%%%%%%%%%%%%%%%%%%
%%%%%%%%%%%%%%%%%%%%%%%%%%%%%%%%%%%%%%%%%%%%%
\begin{cor}\label{cortoYB}
Let $(r^+,r^-)$ be a classical dynamical $r$-matrix pair.
If $(r^+,r^-)$ satisfies the additional equation
$\textup{CYB}(r^+,r^-)=0$ then $\textup{YB}(r^+-r^-)=0$.
\end{cor}
%%%%%%%%%%%%%%%%%%%%%%%%%%%%%%%%%%%%%%%%%%%%
The function $\textup{CYB}(r^+,r^-)$ will naturally reappear in Section \ref{foldingSection}
when constructing classical dynamical $r$-matrix pairs
by folding classical dynamical $r$-matrices
along an involution.

A classical dynamical $k$-matrix relative to a classical dynamical $r$-matrix pair $(r^+,r^-)$
is said to be a {\it core classical dynamical $k$-matrix} if the boundary algebras $A_\ell$ and $A_r$ are trivial,
$A_\ell=\mathbb{C}=A_r$. 
We now present a method to promote core classical classical dynamical $k$-matrices to classical dynamical $k$-matrices
involving nontrivial boundary algebras $A_\ell$ and $A_r$. 

For $r^{\pm}: \mathfrak{a}^*\rightarrow U^{\otimes 2}$ and $\kappa: \mathfrak{a}^*\rightarrow A_\ell\otimes U\otimes A_r$ 
define $\textup{resCR}(r^+,r^-;\kappa): \mathfrak{a}^*\rightarrow A_\ell\otimes U^{\otimes 2}\otimes A_r$ by
\begin{equation}\label{resBYB}
\textup{resCR}(r^+,r^-;\kappa):=[\kappa_1,r^+]+[\kappa_1,\kappa_2]+[r^-,\kappa_2]+E_2(\kappa_1)-E_1(\kappa_2).
\end{equation}
In the following lemma we describe two elementary properties of $\textup{resCR}$.
%%%%%%%%%%%%%%%%%%%%%%%%%%%%%%%%%%%
\begin{lem}\label{rBYB}
\textup{(1)} 
If $\kappa_\ell: \mathfrak{a}^*\rightarrow A_\ell\otimes U$ and $\kappa_r:
\mathfrak{a}^*\rightarrow U\otimes A_r$ then 
\begin{equation}\label{additive}
\textup{resCR}(r^+,r^-;\kappa_\ell\otimes 1_{A_r}+1_{A_\ell}\otimes \kappa_r)=\textup{resCR}(r^+,r^-;\kappa_\ell\otimes 1_{A_r})+\textup{resCR}(r^+,r^-;1_{A_\ell}\otimes \kappa_r).
\end{equation}
\textup{(2)} For meromorphic $r^{\pm}: \mathfrak{a}^*\rightarrow U^{\otimes 2}$, $\kappa^{\textup{core}}: 
\mathfrak{a}^*\rightarrow U$ and $\kappa: \mathfrak{a}^*\rightarrow A_\ell\otimes U\otimes A_r$ we have
\[
\textup{CR}(r^+,r^-;\kappa^{\textup{core}}+\kappa)
=\textup{CR}(r^+,r^-;\kappa^{\textup{core}})+\textup{resCR}(r^+,r^-;\kappa)
\] 
as functions $\mathfrak{a}^*\rightarrow A_\ell\otimes U^{\otimes 2}\otimes A_r$. Here $\kappa^{\textup{core}}$ is viewed as $A_\ell\otimes U\otimes A_r$-valued function and
$\textup{CR}(r^+,r^-;\kappa^{\textup{core}})$ as $A_\ell\otimes U^{\otimes 2}\otimes A_r$-valued function in the natural manner.
\end{lem}
%%%%%%%%%%%%%%%%%%%%%%%%%%%%%%%%%%%%%%%%%
\begin{proof}
(1) This follows from the fact that $[(\kappa_\ell)_1,(\kappa_r)_2]=0=[(\kappa_\ell)_2,(\kappa_r)_1]$ in $A_\ell\otimes U^{\otimes 2}\otimes A_r$.\\
(2) This is a direct check, using that $[\kappa_1,\kappa_2^{\textup{core}}]=0=[\kappa_1^{\textup{core}},\kappa_2]$ in $A_\ell\otimes U^{\otimes 2}\otimes A_r$.
\end{proof}
%%%%%%%%%%%%%%%%%%%%%%%%%%%%%%%%%%%%%%%
The following result relates $\textup{resCR}$ to $\textup{CYB}$. In this case $A_\ell$ and $A_r$ are both equal to $U$.
%%%%%%%%%%%%%%%%%%%%%%%%%%%%%%%%%%%%%%%
\begin{lem}\label{rellem}
For meromorphic $r^{\pm}: \mathfrak{a}^*\rightarrow U^{\otimes 2}$ we have 
\begin{equation*}
\begin{split}
\textup{resCR}(r^+,r^-;r_{01}^+)
&=\textup{CYB}[1](r^+,r^-)_{012},\qquad
\textup{resCR}(r^+,r^-,r_{0^{\prime}1}^+)=\textup{CYB}[1](r^+,r^-)_{0^{\prime}12},\\
\textup{resCR}(r^+,r^-;r_{10}^-)&=\textup{CYB}[3](r^+,r^-)_{120},\qquad
\textup{resCR}(r^+,r^-;r_{10^{\prime}}^-)=\textup{CYB}[3](r^+,r^-)_{120^\prime}
\end{split}
\end{equation*}
as functions $\mathfrak{a}^*\rightarrow U\otimes U^{\otimes 2}\otimes U$. Here the four
tensor components of $U\otimes U^{\otimes 2}\otimes U$ are labelled by $0,1,2,0^\prime$, and 
$r_{01}^+, r_{0^{\prime}1}^+, r_{10}^-, r_{10^\prime}^-$ are viewed as functions $\mathfrak{a}^*\rightarrow U\otimes U^{\otimes 2}\otimes U$ in the natural way.
\end{lem}
%%%%%%%%%%%%%%%%%%%%%%%%%%%%%%%%%%%%%
\begin{proof}
This follows immediately from the definitions of $\textup{resCR}$ and $\textup{CYB}[t]$.
\end{proof}
%%%%%%%%%%%%%%%%%%%%%%%%%%%%%%%%%%%%
With the same notational conventions as in Lemma \ref{rellem}, we have the following consequence.
\begin{cor}\label{kappaextCor}
Let $r^{\pm}: \mathfrak{a}^*\rightarrow U^{\otimes 2}$ and $\kappa^{\textup{core}}:
\mathfrak{a}^*\rightarrow U$ such that $\textup{CR}(r^+,r^-;\kappa^{\textup{core}})=0$.
\begin{enumerate}
\item If $\textup{CYB}[1](r^+,r^-)=0$ then
\[
\textup{CR}(r^+,r^-;\kappa)=0
\]
in $U\otimes U^{\otimes 2}\otimes U$ 
for $\kappa=\kappa^{\textup{core}}+r^+_{01}$,
$\kappa^{\textup{core}}+r_{0^{\prime}1}^+$ and
$\kappa^{\textup{core}}+r^+_{01}+r_{0^{\prime}1}^+$.
\item If $\textup{CYB}[3](r^+,r^-)=0$ then
\[
\textup{CR}(r^+,r^-;\kappa)=0
\]
in $U\otimes U^{\otimes 2}\otimes U$
for $\kappa=\kappa^{\textup{core}}+r^-_{10}$,
$\kappa^{\textup{core}}+r_{10^\prime}^-$ and
$\kappa^{\textup{core}}+r^-_{10}+r_{10^\prime}^-$.
\item If $\textup{CYB}[1](r^+,r^-)=0=\textup{CYB}[3](r^+,r^-)$ then
\[
\textup{CR}(r^+,r^-;\kappa)=0
\]
in $U\otimes U^{\otimes 2}\otimes U$
for $\kappa=\kappa^{\textup{core}}+r_{01}^++r^-_{10^\prime}$ and
$\kappa^{\textup{core}}+r_{10}^-+r^+_{0^{\prime}1}$.
\end{enumerate}
\end{cor}
%%%%%%%%%%%%%%%%%%%%%%%%%%%%%%%%%%%%%
\begin{proof}
This follows from Lemma \ref{rBYB}, Lemma \ref{rellem} and \eqref{additive}.
\end{proof}
%%%%%%%%%%%%%%%%%%%%%%%%%%%%%%%%%%%%%%%%%%%
\section{Folding along an involution}\label{foldingSection}
%%%%%%%%%%%%%%%%%%%%%%%%%%%%%%%%%%%%%%%%%%

We now construct classical dynamical $r$-matrix pairs and associated classical dynamical $k$-matrices
from a given solution $r$ of the classical dynamical Yang-Baxter equation.
We use for this folding and contraction along an involution.  It is likely that this folding and contraction procedure extend in a natural way to classical dynamical $r$-matrices with spectral parameter, but we do not pursue this here.

Fix an involutive unital algebra automorphism $\theta\in\textup{Aut}(U)$ (if $U=U(\mathfrak{g})$ is the universal enveloping algebra of a Lie algebra
$\mathfrak{g}$ then one typically takes the extension to $U(\mathfrak{g})$ of an involutive automorphism of $\mathfrak{g}$). Write
\[
U=U^+\oplus U^-
\]
for the decomposition of $U$ in ($+1$)-and ($-1$)-eigenspaces of $\theta$. Then $U^+$ is a unital subalgebra, $U^-$ is an $U^+$-bimodule and $U^-\cdot U^-\subseteq U^+$. We assume in this section that $\mathfrak{a}\subseteq U^-$.

For a meromorphic function $r: \mathfrak{a}^*\rightarrow U^{\otimes 2}$ we write
\begin{equation}\label{tilder}
\widetilde{r}:=(\theta\otimes\textup{id})r.
\end{equation}
%%%%%%%%%%%%%%
\begin{defi}
We call $r$ $\theta$-twisted symmetric if $\widetilde{r}$ is symmetric, i.e., if $\widetilde{r}_{21}=\widetilde{r}$. We say that $r$ is
$\mathfrak{a}$-invariant if $r$ takes values in the subspace 
\[
(U\otimes U)^{\mathfrak{a}}:=\{s\in U\otimes U \,\,\, | \,\,\, [x\otimes 1+1\otimes x,s]=0
\quad \forall\, x\in \mathfrak{a}\}.
\]
\end{defi}
%%%%%%%%%%%%%%%%%
\begin{rema}\label{rr2Remark}
(1) Note that $(\theta\otimes\theta)r=r_{21}$ if $r$ is $\theta$-twisted symmetric.\\
(2) Let $L(U)$ be the algebra $U$ viewed as Lie algebra, with Lie bracket the commutator bracket. Let $\widetilde{\mathfrak{a}}$ be the Lie subalgebra of $L(U)$ generated by $\mathfrak{a}$. Then 
$(U\otimes U)^{\mathfrak{a}}=(U\otimes U)^{\widetilde{\mathfrak{a}}}$, and $(U\otimes U)^{\widetilde{\mathfrak{a}}}$ is the 
subspace of elements in $U\otimes U$ that are invariant for the diagonal adjoint $\widetilde{\mathfrak{a}}$-action on
$U\otimes U$.\\
(3) Note that for $(U,\mathfrak{a},\theta)=(U(\mathfrak{g}),\mathfrak{h},\sigma)$ as in Section \ref{Se1}, the $\mathfrak{h}$-invariant classical dynamical $r$-matrix $\rr$, given by \eqref{Frrr}, is $\sigma$-twisted symmetric. Furthermore, $\mathfrak{h}$ is contained in the $(-1)$-eigenspace of $\sigma$.
\end{rema}
%%%%%%%%%%%%%%%%%%

We now consider meromorphic functions $r^{\pm}: \mathfrak{a}^*\rightarrow U^{\otimes 2}$ that are of the form
\begin{equation}\label{folded}
r^{\pm}:=\frac{1}{2}(\widetilde{r}\pm r)
\end{equation}
for some meromorphic $r: \mathfrak{a}^*\rightarrow U^{\otimes 2}$.
Note that $r^{\pm}$ takes values in $U^{\pm}\otimes U$. Furthermore, $\textup{CYB}[1](r^+,r^-)$ takes values in $U^+\otimes U\otimes U$, 
while $\textup{CYB}[2](r^+,r^-)$ and $\textup{CYB}(r^+,r^-)$ take values in $U^-\otimes U\otimes U$ since $\mathfrak{a}\subseteq U^-$. 

Write $\theta_i\in\textup{Aut}(U^{\otimes 3})$ for the action of $\theta$ on the $i$-th tensor component of $U^{\otimes 3}$ ($1\leq i\leq 3$).
%%%%%%%%%%%%%%%%%%%%%%%%%%%%%%%%%%%%%%%%%%%%
\begin{prop}\label{YBrelprop}
Suppose that $\mathfrak{a}\subseteq U^-$. Assume that $r: \mathfrak{a}^*\rightarrow U^{\otimes 2}$ is $\theta$-twisted symmetric. Then $r^{\pm}:=\frac{1}{2}(\widetilde{r}\pm r): \mathfrak{a}^*\rightarrow U^{\pm}\otimes U$ satisfies
\begin{equation*}
\begin{split}
\textup{CYB}[1](r^+,r^-)&=\frac{1}{4}\Bigl(\textup{YB}(r)+
\theta_1(\textup{YB}(r))+\theta_2(\textup{YB}(r)_{213})
-\theta_3(\textup{YB}(r)_{312})\Bigr),\\
\textup{CYB}[2](r^+,r^-)&=\frac{1}{4}\Bigl(-\textup{YB}(r)+
\theta_1(\textup{YB}(r))+\theta_2(\textup{YB}(r)_{213})
+\theta_3(\textup{YB}(r)_{312})\Bigr),\\
\textup{CYB}[3](r^+,r^-)&=\frac{1}{4}\Bigl(\textup{YB}(r)-
\theta_1(\textup{YB}(r))+\theta_2(\textup{YB}(r)_{213})
+\theta_3(\textup{YB}(r)_{312})\Bigr),
\end{split}
\end{equation*}
as well as
\begin{equation}\label{vanishingterm}
\textup{CYB}(r^+,r^-)=\frac{1}{4}\Bigl(
\textup{YB}(r)+
\theta_1(\textup{YB}(r))-\theta_2(\textup{YB}(r)_{213})
+\theta_3(\textup{YB}(r)_{312})\Bigr).
\end{equation}
\end{prop}
%%%%%%%%%%%%%%%%%%%%%%%%%%%%%%%%%%%%%%%%%%%
\begin{proof}
We give the proof of the first equation
\begin{equation}\label{todolink}
\textup{CYB}[1](r^+,r^-)=\frac{1}{4}\Bigl(\textup{YB}(r)+
\theta_1(\textup{YB}(r))+\theta_2(\textup{YB}(r)_{213})
-\theta_3(\textup{YB}(r)_{312})\Bigr),
\end{equation}
the others are proved by a similar computation. Let us first establish that the differential contributions on both sides of \eqref{todolink} match. Since $\mathfrak{a}\subseteq U_-$ and $\widetilde{r}_{21}=\widetilde{r}$, the differential contribution to the right hand side of \eqref{todolink} is
\begin{equation*}
\begin{split}
\frac{1}{4}\bigl(E_1(r_{23})-E_2(r_{13})+&E_3(r_{12})\bigr)
+\frac{1}{4}\bigl(-E_1(r_{23})-E_2(\widetilde{r}_{13})+E_3(\widetilde{r}_{12})\bigr)\\
+&\frac{1}{4}\bigl(-E_2(r_{13})-E_1(\widetilde{r}_{23})+E_3(\widetilde{r}_{12})\bigr)
+\frac{1}{4}\bigl(E_3(r_{12})+E_1(\widetilde{r}_{23})-E_2(\widetilde{r}_{13})\bigr).
\end{split}
\end{equation*}
Collecting terms and using that $r^{\pm}=\frac{1}{2}(\widetilde{r}\pm r)$ this expression reduces to
$E_3(r_{12}^+)-E_2(r_{13}^+)$,
which is the differential contribution to $\textup{CYB}[1](r^+,r^-)$.

To complete the proof of \eqref{todolink} it remains to show that 
\begin{equation}\label{todolink2}
\textup{CYB}_0[1](r^+,r^-)=\frac{1}{4}\Bigl(\textup{YB}_0(r)+
\theta_1(\textup{YB}_0(r))+\theta_2(\textup{YB}_0(r)_{213})
-\theta_3(\textup{YB}_0(r)_{312})\Bigr),
\end{equation}
where we use the notation \eqref{zero}.
Substituting $r^{\pm}=\frac{1}{2}(\widetilde{r}\pm r)$ into $\textup{CYB}_0[1](r^+,r^-)$ we get
\begin{equation*}
\begin{split}
\textup{CYB}_0[1](r^+,r^-)=\frac{1}{4}&\Bigl([r_{12},r_{13}]+[\widetilde{r}_{12},r_{13}]+[r_{12},\widetilde{r}_{13}]
+[\widetilde{r}_{12},\widetilde{r}_{13}]\\
&+[r_{12},r_{23}]+[\widetilde{r}_{12},r_{23}]+[r_{12},\widetilde{r}_{23}]+[\widetilde{r}_{12},\widetilde{r}_{23}]\\
&+[r_{13},r_{23}]+[\widetilde{r}_{13},r_{23}]-[r_{13},\widetilde{r}_{23}]-[\widetilde{r}_{13},\widetilde{r}_{23}]\Bigr).
\end{split}
\end{equation*}
In this expression we rewrite the commutators involving $\widetilde{r}$ in terms of commutators that only involve $r$, for which we sometimes need to use the $\theta$-twisted symmetry $\widetilde{r}=\widetilde{r}_{21}$ of $r$. We get
\begin{equation*}
\begin{split}
\textup{CYB}_0[1](r^+,r^-)=\frac{1}{4}&\Bigl([r_{12},r_{13}]+\theta_2([r_{21},r_{13}])+
\theta_3([r_{12},r_{31}])
+\theta_1([r_{12},r_{13}])\\
&+[r_{12},r_{23}]+\theta_1([r_{12},r_{23}])+\theta_3([r_{12},r_{32}])+\theta_2([r_{21},r_{23}])\\
&+[r_{13},r_{23}]+\theta_1([r_{13},r_{23}])-\theta_2([r_{13},r_{23}])-
\theta_3([r_{31},r_{32}])\Bigr)\\
&=\frac{1}{4}\Bigl(\textup{YB}_0(r)+
\theta_1(\textup{YB}_0(r))+\theta_2(\textup{YB}_0(r)_{213})
-\theta_3(\textup{YB}_0(r)_{312})\Bigr),
\end{split}
\end{equation*}
where the last equality follows by a direct computation.  This completes the proof of \eqref{todolink2}, and hence of \eqref{todolink}.
\end{proof}
%%%%%%%%%%%%%%%%%%%%%%%%%%%%%%%%%%%%%%%%%%
\begin{thm}\label{equivcor}
Suppose that $\mathfrak{a}\subseteq U^-$. Assume that $r: \mathfrak{a}\rightarrow U^{\otimes 2}$ is $\theta$-twisted symmetric, and set
$r^{\pm}:=\frac{1}{2}(\widetilde{r}\pm r): \mathfrak{a}^*\rightarrow U^{\pm}\otimes U$.
The following statements are equivalent.
\begin{enumerate}
\item $r$ is a classical dynamical $r$-matrix. 
\item $(r^+,r^-)$ is a classical dynamical $r$-matrix pair satisfying the additional non-dynamical equation
$\textup{CYB}(r^+,r^-)=0$.
\end{enumerate}
\end{thm}
%%%%%%%%%%%%%%%%%%%%%%%%%%%%%%%%%
\begin{proof}
(1)$\Rightarrow$(2) is immediate from Proposition \ref{YBrelprop}. The implication
(2)$\Rightarrow$(1) is a special case of Corollary \ref{cortoYB}.
\end{proof}
%%%%%%%%%%%%%%%%%%%%%%%%%%%%%%%%%
We next construct an accompanying core classical dynamical $k$-matrix by a $\theta$-twisted
contraction procedure. Denote by $m: U\otimes U\rightarrow U$ the multiplication map of $U$.  
%%%%%%%%%%%%%%%%%%%%%%%%%%%%%%%%%%%%%%%%
\begin{prop}\label{mainrelation}
Suppose that $\mathfrak{a}\subseteq U^-$. Let $r: \mathfrak{a}^*\rightarrow U\otimes U$ be $\theta$-twisted symmetric and $\mathfrak{a}$-invariant, and write $r^{\pm}:=\frac{1}{2}(\widetilde{r}\pm r): \mathfrak{a}^*\rightarrow U^{\pm}\otimes U$.
Then 
\begin{equation}\label{invariantplusminus}
[x\otimes 1,r^+]-[1\otimes x,r^-]=0\qquad \forall\, x\in\mathfrak{a}
\end{equation}
and 
\begin{equation}\label{BYBfolded}
\begin{split}
\textup{CR}\Bigl(r^+,r^-;\frac{m(\widetilde{r})}{2}\Bigr)
=\frac{1}{2}(m\otimes m)\Bigl(&\textup{CYB}[1](r^+,r^-)_{123}+\textup{CYB}[2](r^+,r^-)_{123}\\
&+\textup{CYB}[2](r^+,r^-)_{234}+\textup{CYB}[3](r^+,r^-)_{234}\Bigr).
\end{split}
\end{equation}
\end{prop}
%%%%%%%%%%%%%%%%%%%%%%%%%%%%%%%%%%%%%%%
\begin{proof}
Formula \eqref{invariantplusminus} follows from the $\mathfrak{a}$-invariance of $r$
and the fact that $\mathfrak{a}\subseteq U^-$.

Write $\kappa^{\textup{core}}:=m(\widetilde{r})/2$. We first check that the differential contributions to the left and right hand side of \eqref{BYBfolded} match. The differential contribution to
\[
\frac{1}{2}\Bigl(\textup{CYB}[1](r^+,r^-)_{123}+\textup{CYB}[2](r^+,r^-)_{123}
+\textup{CYB}[2](r^+,r^-)_{234}+\textup{CYB}[3](r^+,r^-)_{234}\Bigr)
\]
is
\begin{equation*}
\frac{1}{2}\bigl(E_3(r_{12}^+)-E_2(r_{13}^+)\bigr)
+\frac{1}{2}\bigl(E_3(r_{12}^-)-E_1(r_{23}^+)\bigr)
+\frac{1}{2}\bigl(E_4(r_{23}^-)-E_2(r_{34}^+)\bigr)
+\frac{1}{2}\bigl(E_3(r_{24}^-)-E_2(r_{34}^-)\bigr).
\end{equation*}
Reordering the terms and using that $\widetilde{r}=r^++r^-$ it can be written as
\begin{equation*}
\frac{1}{2}\bigl(E_3(\widetilde{r}_{12})-E_2(\widetilde{r}_{34})\bigr)
+\frac{1}{2}\sum_{k=1}^\dd\Bigl(\partial_{\lm_k}(r_{23}^-)(x_k)_4+(x_k)_3\partial_{\lm_k}(r_{24}^-)
-\partial_{\lm_k}(r_{13}^+)(x_k)_2-(x_k)_1\partial_{\lm_k}(r_{23}^+)\Bigr).
\end{equation*}
Applying $m\otimes m$ we obtain the expression
\begin{equation*}
E_2(\kappa_1^{\textup{core}})-E_1(\kappa_2^{\textup{core}})
+\frac{1}{2}\sum_{k=1}^\dd\Bigl(\partial_{\lm_k}(r^-)(x_k)_2+(x_k)_2\partial_{\lm_k}(r^-)-
\partial_{\lm_k}(r^+)(x_k)_1-(x_k)_1\partial_{\lm_k}(r^+)\Bigr).
\end{equation*}
By \eqref{invariantplusminus} this reduces to $E_2(\kappa_1^{\textup{core}})-E_1(\kappa_2^{\textup{core}})+E_2(r^-)-E_1(r^+)$,
which is the differential contribution to the left hand side of \eqref{BYBfolded}. Hence it remains to show
that
\begin{equation}\label{BYBfoldedzero}
[\kappa_1^{\textup{core}}+r^-,\kappa_2^{\textup{core}}+r^+]=\frac{1}{2}
(m\otimes m)\bigl(\textup{CYB}_0[1]_{123}+\textup{CYB}_0[2]_{123}
+\textup{CYB}_0[2]_{234}+\textup{CYB}_0[3]_{234}\bigr)
\end{equation}
with $\textup{CYB}_0[t]=\textup{CYB}_0[t](r^+,r^-)$ defined by \eqref{zero}. 

Since $\kappa^{\textup{core}}
=\frac{1}{2}m(r^++r^-)$ the left hand side of \eqref{BYBfoldedzero} is
\begin{equation}\label{todozero}
[\kappa_1^{\textup{core}}+r^-,\kappa_2^{\textup{core}}+r^+]=\frac{1}{2}[m(r^++r^-)_1,r^+]+\frac{1}{2}[r^-,m(r^++r^-)_2]
+[r^-,r^+].
\end{equation}
Now for $s,s^\prime\in U^{\otimes 2}$ one verifies by direct computations that
\begin{equation*}
\begin{split}
[m(s)_1,s^\prime]&=(m\otimes \textup{id}_U)\bigl([s_{12},s^\prime_{13}+s_{23}^\prime]\bigr),\\
[m(s)_2,s^\prime]&=(\textup{id}_U\otimes m)\bigl([s_{23},s_{12}^\prime+s_{13}^\prime]
\bigr),\\
[s,s^\prime]&=(m\otimes m)\bigl([s_{23},s_{24}^\prime]+[s_{23},s_{13}^\prime]\bigr)
\end{split}
\end{equation*}
in $U^{\otimes 2}$. Substituting into \eqref{todozero} we get
\[
[\kappa_1^{\textup{core}}+r^-,\kappa_2^{\textup{core}}+r^+]=\frac{1}{2}
(m\otimes m)\Bigl(
[r_{12}^++r_{12}^-,r_{13}^++r_{23}^+]-[r_{34}^++r_{34}^-,r_{23}^-+r_{24}^-]
+2[r_{23}^-,r_{24}^+]+2[r_{23}^-,r_{13}^+]\Bigr).
\]
Rewriting the commutators in the right hand side in terms of $\textup{CYB}_0[t]$ (see \eqref{zero}), one gets
\begin{equation*}
\begin{split}
[\kappa_1^{\textup{core}}+r^-,\kappa_2^{\textup{core}}+r^+]&=\frac{1}{2}(m\otimes m)\bigl(
\textup{CYB}_0[1]_{123}+\textup{CYB}_0[2]_{123}
+\textup{CYB}_0[2]_{234}+\textup{CYB}_0[3]_{234}\bigr)\\
&+\frac{1}{2}(m\otimes m)\bigl([r_{23}^+,r_{24}^-]+[r_{23}^-,r_{24}^+]-[r_{13}^+,r_{23}^-]
-[r_{13}^-,r_{23}^+]\bigr).
\end{split}
\end{equation*}
So to prove \eqref{BYBfoldedzero} it suffices to show that
\begin{equation}\label{todozero2}
(m\otimes m)\bigl([r_{23}^+,r_{24}^-]+[r_{23}^-,r_{24}^+]-[r_{13}^+,r_{23}^-]
-[r_{13}^-,r_{23}^+]\bigr)=0.
\end{equation}
Substitute $r^{\pm}=\frac{1}{2}(\widetilde{r}\pm r)$ in the left hand side of \eqref{todozero2} and apply the multiplication map if one of its
tensor legs does not contain a component of a commutator.
Then
\begin{equation*}
\begin{split}
&(m\otimes m)\bigl([r_{23}^+,r_{24}^-]+[r_{23}^-,r_{24}^+]-[r_{13}^+,r_{23}^-]
-[r_{13}^-,r_{23}^+]\bigr)=\\
&\,\,=\frac{1}{2}\Bigl((m\otimes\textup{id}_U)[r_{13},r_{23}]-
(\textup{id}_U\otimes m)[r_{12},r_{13}]\Bigr)
-\frac{1}{2}\Bigl((m\otimes\textup{id}_U)[\widetilde{r}_{13},\widetilde{r}_{23}]-(\textup{id}_U\otimes m)[\widetilde{r}_{12},\widetilde{r}_{13}]\Bigr)
\end{split}
\end{equation*}
(the commutators involving $r$ and $\widetilde{r}$ cancel out).
The right hand side of this expression vanishes since for $s\in U^{\otimes 2}$,
\[
(m\otimes\textup{id}_U)[s_{13},s_{23}]-(\textup{id}_U\otimes m)[s_{12},s_{13}]=0
\]
in $U^{\otimes 2}$. This completes the proof of the proposition.
\end{proof}
%%%%%%%%%%%%%%%%%%%%%%%%%%%%%%%%%%%%%%%%%
\begin{thm}\label{BYB1}
Suppose that $\mathfrak{a}\subseteq U^-$. Let $r: \mathfrak{a}^*\rightarrow U^{\otimes 2}$ be a $\theta$-twisted symmetric, $\mathfrak{a}$-invariant classical dynamical $r$-matrix.
Write $(r^+,r^-)$ for the associated classical dynamical $r$-matrix pair, where $r^{\pm}:=\frac{1}{2}(\widetilde{r}\pm r)$ 
\textup{(}see Theorem \ref{equivcor}\textup{)}. Then $\kappa^{\textup{core}}:=m(\widetilde{r})/2: \mathfrak{a}^*\rightarrow U$
is a core classical dynamical $k$-matrix relative to $(r^+,r^-)$.
\end{thm}
%%%%%%%%%%%%%%%%%%%%%%%%%%%%%%%%%%%%%%%%%
\begin{proof}
We have to show that 
\[
\textup{CR}(r^+,r^-;\kappa^{\textup{core}})=0.
\]
This follows immediately from Proposition \ref{mainrelation} and Theorem \ref{equivcor}.
\end{proof}
%%%%%%%%%%%%%%%%%%%%%%%%%%%%%%%%%%%%%%%%%
\begin{rema}\label{rr3Remark}
Let $(U,\mathfrak{a},\theta)=(U(\mathfrak{g}),\mathfrak{h},\sigma)$. 
By Remark \ref{rr2Remark}(3), $\rr:\mathfrak{h}^*\rightarrow\mathfrak{g}\otimes\mathfrak{g}$ (see \eqref{Frrr}) is an $\mathfrak{h}$-invariant, $\sigma$-twisted symmetric classical dynamical $r$-matrix. By Theorem \ref{BYB1} the triple $(\rr^+,\rr^-)$ is a coupled classical dynamical $r$-matrix pair and $m(\widetilde{\rr})/2$ is 
a core classical dynamical $k$-matrix relative to $(\rr^+,\rr^-)$. We have for $s\in\mathbb{C}$,
\[
\textup{CR}\Bigl(\rr^+,\rr^-;\frac{m(\widetilde{\rr})}{2}+sy\Bigr)=\textup{resCR}(\rr^+,\rr^+;sy)=0
\]
by Lemma \ref{rBYB}, \eqref{invariantplusminus} and \eqref{yadd}, so $\frac{m(\widetilde{\rr})}{2}+sy$ is a classical dynamical $k$-matrix relative to $(\rr^+,\rr^-)$ for all $s\in\mathbb{C}$. The boundary KZB equations \eqref{bKZBformula} then state that $N$-point spherical functions produce common eigenfunctions of
the commuting operators $\mathcal{D}_i^{(N)}$ associated to the triple $(\rr^+,\rr^-,(m(\widetilde{\rr})-y)/2)$.
\end{rema}
%%%%%%%%%%%%%%%%%%%%%%%%%%%%%%%%%%%%%%%%

In the following corollary we construct classical dynamical $k$-matrices taking values in the algebra $U^+\otimes U\otimes U^{+}$. 
We label the tensor components in $U^+\otimes U\otimes U^{+}$ by $0,1$ and $0^\prime$.
%%%%%%%%%%%%%%%%%%%%%%%%%%%%%%%%%%%%%%%%%%%%%
\begin{cor}\label{BYB2}
Suppose that $\mathfrak{a}\subseteq U^-$. Let $r: \mathfrak{a}^*\rightarrow U^{\otimes 2}$ be a $\theta$-twisted symmetric and $\mathfrak{a}$-invariant 
classical dynamical $r$-matrix.
Write
$r^{\pm}:=\frac{1}{2}(\widetilde{r}\pm r): \mathfrak{a}^*\rightarrow U^{\pm}\otimes U$ and
$\kappa^{\textup{core}}:=m(\widetilde{r})/2: \mathfrak{a}^*\rightarrow U$ for the associated classical dynamical $r$-matrix pair
and core classical dynamical $k$-matrix, respectively.

Then $\kappa: \mathfrak{a}^*\rightarrow U^+\otimes U\otimes U^+$, with $\kappa$ one of the following three meromorphic functions
\[
\kappa^{\textup{core}}_1+r_{01}^+,\,\,\kappa^{\textup{core}}_1+r_{0^{\prime}1}^+,\,\,\kappa^{\textup{core}}_1+r_{01}^++r_{0^{\prime}1}^+,
\]
is a classical dynamical $k$-matrix relative to $(r^+,r^-)$.
\end{cor}
%%%%%%%%%%%%%%%%%%%%%%%
\begin{proof}
We have to show that $\textup{CR}(r^+,r^-;\kappa)=0$ for $\kappa$ either one of the three meromorphic functions listed in the corollary.
This follows from Theorem \ref{BYB1} and Corollary \ref{kappaextCor}(1).
\end{proof}
%%%%%%%%%%%%%%%%%%%%%%%%
\begin{rema}
For the example discussed in Remark \ref{rr3Remark}, the extensions of the core classical dynamical $k$-matrix provided by Corollary \ref{BYB2} correspond representation theoretically to adding spin boundary behaviour to the $N$-point spherical functions, see \cite[\S 6]{SR}. For this particular example 
additional extensions of the
associated core classical dynamical $k$-matrix can be made, see Proposition \ref{ktwist}.
\end{rema}
%%%%%%%%%%%%%%%%%%%%%%%%
\begin{rema}\label{remtowardsGaudin}
Assume that $r: \mathfrak{a}^*\rightarrow U^{\otimes 2}$ is a $\theta$-twisted symmetric and $\mathfrak{a}$-invariant 
classical dynamical $r$-matrix.
Write
$r^{\pm}:=\frac{1}{2}(\widetilde{r}\pm r): \mathfrak{a}^*\rightarrow U^{\pm}\otimes U$ and suppose that $\kappa: \mathfrak{a}^*\rightarrow
A_\ell\otimes U\otimes A_r$ is a classical dynamical $k$-matrix with relative to $(r^+,r^-)$ .
The resulting commuting first order
differential operators $\{\mathcal{D}_i^{(N)}\,\, | \,\, 1\leq i\leq N\}$ (see Definition \ref{Di}) can then be expressed as
\[
\mathcal{D}_i^{(N)}=E_i
-A_i^{(N)}
\]
with $A_i^{(N)}: \mathfrak{a}^*\rightarrow A_\ell\otimes U^{\otimes N}\otimes A_r$ given by 
\begin{equation}\label{AiN}
A_i^{(N)}:=\kappa_i+\frac{1}{2}\sum_{s=1}^{i-1}r_{si}-\frac{1}{2}\sum_{s=i+1}^Nr_{is}+
\frac{1}{2}\sum_{s\not=i}\widetilde{r}_{si}.
\end{equation}
If $r$ is quasi-unitary with coupling constant $t$, i.e.,
\begin{equation}\label{tu}
r+r_{21}=t\varpi,
\end{equation}
then
\[
A_i^{(N)}=-\frac{t}{2}\sum_{s=i+1}^N\varpi_{is}+
\Bigl(\kappa_{i}+\frac{1}{2}\sum_{s\not=i}r_{si}+\frac{1}{2}\sum_{s\not=i}\widetilde{r}_{si}
\Bigr).
\]
This applies in particular to the asymptotic boundary KZB operators, when $r$ is Felder's classical dynamical $r$-matrix \eqref{Frrr}, which is quasi-unitary with coupling constant $-1$ (see \eqref{qu}). Note that if $r$ is skew-symmetric, then
\[
A_i^{(N)}=\kappa_{i}+\frac{1}{2}\sum_{s\not=i}r_{si}+\frac{1}{2}\sum_{s\not=i}\widetilde{r}_{si}.
\]
\end{rema}
%%%%%%%%%%%%%%%%%%%%%%%%

%%%%%%%%%%%%%%%%%%%%%%%
\section{Examples}\label{ExampleSection}
%%%%%%%%%%%%%%%%%%%%%%%
\subsection{Schiffmann's classical dynamical $r$-matrices}
%%%%%%%%%%%%%%%%%%%%%%%%
We freely use the notations from Section \ref{Se1}. We assume though $\mathfrak{g}$ to be simple, and we will allow different normalisations for the root vectors $e_\alpha$.

Recall that $\mathfrak{h}\subset\mathfrak{g}$ is a fixed Cartan subalgebra. The pair $(\mathfrak{g},\mathfrak{h})$ is the complexification of a pair $(\mathfrak{g}_{\mathbb{R}},\mathfrak{h}_{\mathbb{R}})$ with $\mathfrak{g}_{\mathbb{R}}$ a split real simple Lie algebra and Cartan subalgebra $\mathfrak{h}_{\mathbb{R}}=\textup{span}_{\mathbb{R}}\{t_\beta\,\, | \,\, \beta\in R\}$. The restriction $(\cdot,\cdot)|_{\mathfrak{h}_{\mathbb{R}}\times\mathfrak{h}_{\mathbb{R}}}$ is positive definite. We write $\mathfrak{g}_{\mathbb{R},\alpha}:=\mathfrak{g}_{\mathbb{R}}\cap\mathfrak{g}_\alpha$ for
$\alpha\in R$. 

We call a complex subspace $\mathfrak{a}\subseteq\mathfrak{h}$ of real type if it has a real form contained in $\mathfrak{h}_{\mathbb{R}}$.
We write 
$\mathfrak{t}$ for its orthogonal complement in $\mathfrak{h}$ with respect to $(\cdot,\cdot)$. Note that if $\mathfrak{a}$ is of real type, then $\mathfrak{t}$ is also of real type and $\mathfrak{h}=\mathfrak{a}\oplus\mathfrak{t}$.
For subspaces $\mathfrak{a}\subseteq\mathfrak{h}$ of real type we embed $\mathfrak{a}^*\hookrightarrow\mathfrak{h}^*$ by extending $\lambda\in\mathfrak{a}^*$ to a linear functional on $\mathfrak{h}$ by $\lambda(x+y):=\lambda(x)$ for $x\in\mathfrak{a}$ and $y\in\mathfrak{t}$.

Let $\Delta$ be a basis of $R$.
Write $R^{\pm}$ for the corresponding set of positive and negative roots, and $\mathfrak{b}^{\pm}$ for the associated positive and negative Borel sub-algebras of $\mathfrak{g}$.  For a subset $\Gamma\subseteq\Delta$ of simple roots write $R_{\Gamma}:=\mathbb{Z}\Gamma\cap R$, which is a root subsystem with basis $\Gamma$. We write $R_{\Gamma}^{\pm}$ for the corresponding positive and negative roots in $R_{\Gamma}$. Let $\mathfrak{g}_\Gamma\subseteq \mathfrak{g}$ be the (semisimple)
Lie subalgebra generated by $\mathfrak{g}_\alpha$ ($\alpha\in\Gamma$).

We now recall the classical dynamical $r$-matrices $r: \mathfrak{a}^*\rightarrow \mathfrak{g}\otimes \mathfrak{g}$ for subspaces $\mathfrak{a}\subseteq\mathfrak{h}$ of real type constructed in \cite{S,ES} (to make contact with the setup in the earlier sections, one views $r$ as function $r: \mathfrak{a}^*\rightarrow U(\mathfrak{g})\otimes U(\mathfrak{g})$ using the canonical embedding of $\mathfrak{g}$ into $U(\mathfrak{g})$). 
We follow here \cite[App. A]{ES2}, which is slightly more general than the setup in \cite{S}. 

{\it Generalised Belavin-Drinfeld triples} are triples $(\Gamma_1,\Gamma_2,\tau)$ with
$\Gamma_1$ and $\Gamma_2$ subsets of $\Delta$ and 
$\tau: \Gamma_1\overset{\sim}{\longrightarrow}\Gamma_2$ a bijection satisfying
$(\tau(\alpha),\tau(\beta))=(\alpha,\beta)$ for all $\alpha,\beta\in\Gamma_1$.
Note that $\tau$ uniquely extends to an 
isomorphism $R_{\Gamma_1}\overset{\sim}{\longrightarrow} R_{\Gamma_2}$
of root systems (see, e.g., \cite[Prop. 11.1]{Hu}), which we again denote by $\tau$. Set
\[
R_{\Gamma_1,\tau}:=\{\beta\in R_{\Gamma_1} \,\, | \,\, \tau^i(\beta)\in R_{\Gamma_1}\,\, \forall\, i\in\mathbb{Z}_{\geq 0}\}.
\]
%%%%%%%%%%%%%%%%%%%%%%%%%%%%%%
\begin{defi}\label{admissibledef}
Let $\mathfrak{a}\subseteq\mathfrak{h}$ be a subspace of real type. A generalised Belavin-Drinfeld triple
$(\Gamma_1,\Gamma_2,\tau)$ is said to be $\mathfrak{a}$-admissible if
\begin{enumerate}
\item[{\textup{(1)}}] $\tau(\alpha)|_{\mathfrak{a}}=\alpha|_{\mathfrak{a}}$ for all $\alpha\in R_{\Gamma_1}$, 
\item[{\textup{(2)}}] $(\alpha+\tau(\alpha)+\cdots+\tau^{k-1}(\alpha))|_{\mathfrak{t}}=0$
for $\alpha\in R_{\Gamma_1,\tau}$ 
and $k\in\mathbb{Z}_{>0}$ such that 
$\tau^k(\alpha)=\alpha$.
\end{enumerate}
\end{defi}
%%%%%%%%%%%%%%%%%%%%%%%%%%%%%%%%%%%

%%%%%%%%%%%%
\begin{rema}
(1) $\{(\Gamma,\Gamma,\textup{id}) \,\, | \,\, \Gamma\subseteq\Delta\}$ is the set of 
generalised $\mathfrak{h}$-admissible Belavin-Drinfeld triples.\\
(2) A {\it Belavin-Drinfeld triple} is a generalised Belavin-Drinfeld triple $(\Gamma_1,\Gamma_2,\tau)$ satisfying the nilpotency condition $R_{\Gamma_1,\tau}=\emptyset$
(see \cite{BD2}).
Note that the Belavin-Drinfeld triples are exactly the $\{0\}$-admissible generalised Belavin-Drinfeld triples (use here that $\tau$ maps positive roots to positive roots in order to show that condition (2) in Definition \ref{admissibledef} implies $R_{\Gamma_1,\tau}=\emptyset$ when $\mathfrak{a}=\{0\}$).
\end{rema}
%%%%%%%%%%%%%

Fix a subspace $\mathfrak{a}\subseteq\mathfrak{h}$ of real type and an $\mathfrak{a}$-admissible generalised Belavin-Drinfeld triple $(\Gamma_1,\Gamma_2,\tau)$.
For $\alpha\in\Delta$ choose $e_\alpha\in\mathfrak{g}_{\alpha}$ and $e_{-\alpha}\in \mathfrak{g}_{-\alpha}$ such that $[e_\alpha,e_{-\alpha}]=t_\alpha$. The set $\{e_\alpha,e_{-\alpha},t_\alpha\}_{\alpha\in\Delta}$ generates the Lie algebra $\mathfrak{g}$.
By the isomorphism theorem \cite[Thm. 14.2]{Hu} there exists a unique isomorphism $\mathfrak{g}_{\Gamma_1}\overset{\sim}{\longrightarrow}\mathfrak{g}_{\Gamma_2}$ such that
$t_\alpha\mapsto t_{\tau(\alpha)}$, $e_{\pm\alpha}\mapsto e_{\pm\tau(\alpha)}$ for $\alpha\in\Gamma_1$. We denote the Lie algebra isomorphism by $\tau$ again. Choose additional root vectors $e_\beta\in\mathfrak{g}_\beta$ ($\beta\in R\setminus
\{\pm\alpha \,\, | \,\, \alpha\in\Gamma_1\})$ such that $[e_\beta,e_{-\beta}]=t_\beta$ for all $\beta\in R$
and $\tau(e_\gamma)=e_{\tau(\gamma)}$ for all $\gamma\in R_{\Gamma_1}$ (this is possible using Chevalley bases of $\mathfrak{g}$, see, e.g., \cite[\S 2.9]{Sa}). We extend $\tau$ to a linear endomorphism of $\mathfrak{g}$ by setting $\tau(t_\beta)=0$ for $\beta\in\Delta\setminus\Gamma_1$
and $\tau(e_\beta)=0$ for $\beta\in R\setminus R_{\Gamma_1}$. 

For $\alpha\in R_{\Gamma_1}$ define $\varphi_\alpha: \mathfrak{a}^*\rightarrow \mathfrak{g}$ by
\begin{equation}\label{varphi}
\varphi_\alpha(\lm):=\sum_{j>0}e^{-j(\alpha,\lm)}\tau^j(e_\alpha) \qquad\quad (\lm\in\mathfrak{a}^*).
\end{equation}
This formula should be interpreted as follows. If $\alpha\in R\setminus R_{\Gamma_1,\tau}$ then $\tau^k(\alpha)\in R\setminus R_{\Gamma_1}$ for some $k\in\mathbb{Z}_{>0}$, and hence $\tau^j(e_\alpha)=0$ for $j\geq k$. In this case the sum \eqref{varphi} terminates and 
defines an analytic $\mathfrak{g}$-valued function on $\mathfrak{a}^*$.
If $\alpha\in R_{\Gamma_1,\tau}$ then  
write $k<\ell$ for the strictly positive integers with $k+\ell$ as small as possible such that $\tau^k(\alpha)=\tau^\ell(\alpha)$. Then $\varphi_\alpha(\lm)$ should be read as
\begin{equation}\label{varphi2}
\varphi_\alpha(\lm)=\sum_{j=1}^{k-1}e^{-j(\alpha,\lm)}e_{\tau^j(\alpha)}+
\frac{1}{1-e^{(k-\ell)(\alpha,\lm)}}
\sum_{j=k}^{\ell-1}e^{-j(\alpha,\lm)}e_{\tau^j(\alpha)},
\end{equation}
which is a meromorphic $\mathfrak{g}_{\Gamma_1}$-valued function on $\mathfrak{a}^*$.

For a complex subspace $\mathfrak{a}\subseteq\mathfrak{h}$ of real type write $\varpi\in (S^2\mathfrak{g})^{\mathfrak{g}}$, $\varpi_{\mathfrak{h}}\in S^2\mathfrak{h}$, $\varpi_{\mathfrak{a}}\in S^2\mathfrak{a}$ and $\varpi_{\mathfrak{t}}\in S^2\mathfrak{t}$ for the elements representing the nondegenerate symmetric bilinear forms $K(\cdot,\cdot)$, $(\cdot,\cdot)$, $(\cdot,\cdot)|_{\mathfrak{a}\times\mathfrak{a}}$ and
$(\cdot,\cdot)|_{\mathfrak{t}\times\mathfrak{t}}$, respectively. Note that
\begin{equation}\label{varpi}
\varpi=\varpi_{\mathfrak{h}}+\sum_{\alpha\in R}e_{-\alpha}\otimes e_\alpha,\qquad
\varpi_{\mathfrak{h}}=\varpi_{\mathfrak{a}}+\varpi_{\mathfrak{t}}.
\end{equation}
We need the following lemma, which is essentially \cite[Thm. 10.1(i)]{ES2}.
%%%%%%%%%%%%%%%%%%%%%%%%%%%%%%%%%%%%%%%%%%%
\begin{lem}\label{rt}
Let $\mathfrak{a}\subseteq\mathfrak{h}$ be a subspace of real type and $(\Gamma_1,\Gamma_2,\tau)$ an $\mathfrak{a}$-admissible generalised Belavin-Drinfeld triple. 
Write $\mathcal{S}\subseteq \mathfrak{t}\otimes\mathfrak{t}$ for the set of elements
$r_{\mathfrak{t}}\in\mathfrak{t}\otimes\mathfrak{t}$ such that 
\begin{equation}\label{tpart}
((\alpha-\tau(\alpha))\otimes 1)r_{\mathfrak{t}}=\frac{1}{2}((\alpha+\tau(\alpha))\otimes 1)\varpi_{\mathfrak{t}}\qquad \forall\, \alpha\in \Gamma_1.
\end{equation}
Then
\begin{enumerate}
\item $\mathcal{S}\subseteq \mathfrak{t}\otimes\mathfrak{t}$ is an affine subspace such that
$\mathcal{S}\cap\wedge^2\mathfrak{t}\not=\emptyset$.
\item $\mathcal{S}\cap S^2\mathfrak{t}\not=\emptyset$ $\Leftrightarrow$ $\Gamma_1=\Gamma_2$ and $\tau=\textup{id}$. In that case, $\mathcal{S}=\mathfrak{t}\otimes\mathfrak{t}$.
\end{enumerate}
\end{lem}
%%%%%%%%%%%%%%%%%%%%%%%%%%%%%%%%%%%%%%%%
\begin{proof}
For (1) see \cite[Thm. 10.1(ii)]{ES2}. The proof of \cite[Thm. 10.1(ii)]{ES2} shows that $\mathcal{S}$ can only contain a symmetric element if $\tau(\alpha)=\alpha$ for all $\alpha\in\Gamma_1$, showing the if part of (2). For the converse, if $(\Gamma_1,\Gamma_2,\tau)=(\Gamma,\Gamma,\textup{id})$ then 
$\alpha|_{\mathfrak{t}}=0$ for $\alpha\in\Gamma$ due to $\mathfrak{a}$-admissibility, hence both sides of \eqref{tpart} are zero for all $r_{\mathfrak{t}}\in\mathfrak{t}\otimes\mathfrak{t}$.
\end{proof}
%%%%%%%%%%%%%%%%%%%%%%%%%%%%%%%%%%%%%%%%%
Write $x\wedge y:=x\otimes y-y\otimes x$ for $x,y\in\mathfrak{g}$. By \cite[\S 5.1]{S} and \cite[Thm. 10.1(ii)]{ES2} we have the following result.
%%%%%%%%%%%%%%%%%%%%%%%%%%%%%%%%%%%%%%%%%%
\begin{thm}\label{ESthm}
Let $\mathfrak{a}\subseteq\mathfrak{h}$ be a subspace of real type, $(\Gamma_1,\Gamma_2,\tau)$ an $\mathfrak{a}$-admissible generalised Belavin-Drinfeld triple and $r_{\mathfrak{t}}\in\mathcal{S}$.
The meromorphic function $r^{\textup{Sch}}: \mathfrak{a}^*\rightarrow \mathfrak{g}^{\otimes 2}$ defined by
\begin{equation}\label{rES}
r^{\textup{Sch}}(\lm):=\frac{1}{2}\varpi+r_{\mathfrak{t}}+\sum_{\alpha\in R_{\Gamma_1}^+}\varphi_\alpha(\lm)\wedge 
e_{-\alpha}+\frac{1}{2}\sum_{\alpha\in R^+}e_\alpha\wedge e_{-\alpha}
\end{equation}
is an $\mathfrak{a}$-invariant classical dynamical $r$-matrix.
\end{thm}
%%%%%%%%%%%%%%%%%%%%%%%%%%%%%%%%%%%%%%%%%%%%%
Note that $r^{\textup{Sch}}$ does not depend on the particular choice of root vectors $e_\alpha\in\mathfrak{g}_\alpha$ ($\alpha\in R$) as long as $[e_\alpha,e_{-\alpha}]=t_\alpha$ for all $\alpha\in R$. Furthermore, note that $r^{\textup{Sch}}$ is quasi-unitary with coupling constant one, i.e.,
\begin{equation}\label{uc}
r^{\textup{Sch}}+r_{21}^{\textup{Sch}}=\varpi,
\end{equation}
iff $r_{\mathfrak{t}}\in\mathcal{S}\cap\wedge^2\mathfrak{t}$. 

Note that the $\mathfrak{a}$-invariance of $r^{\textup{Sch}}$ is due to the fact that $\tau(\alpha)|_{\mathfrak{a}}=\alpha|_{\mathfrak{a}}$ for all $\alpha\in R_{\Gamma_1}$,
see Definition \ref{admissibledef}(1).
Up to an appropriate notion of gauge equivalence, the $r^{\textup{Sch}}$ with 
$r_{\mathfrak{t}}\in\mathcal{S}\cap\wedge^2\mathfrak{t}$ are all the $\mathfrak{a}$-invariant quasi-unitary 
classical dynamical $r$-matrices with values in $\mathfrak{g}^{\otimes 2}$ having coupling constant $1$,
see \cite[App. A]{ES2}. 

%%%%%%%%%%%%%%%%%%%%%%%%%%%%%%%%%%%%%%%%%%%%%%%%
\begin{rema}\label{relationtoSR}
For the $\mathfrak{h}$-admissible generalised Belavin-Drinfeld triple $(\Gamma_1,\Gamma_2,\tau)=(\Delta,\Delta,\textup{id})$ we have
\[
\varphi_\alpha(\lm)=\frac{e_\alpha}{e^{(\alpha,\lm)}-1}\qquad \forall\, \alpha\in R
\]
and hence
\[
r^{\textup{Sch}}(\lm)=\frac{1}{2}\varpi_{\mathfrak{h}}+\sum_{\alpha\in R}\frac{e_{-\alpha}\otimes e_\alpha}{1-e^{(\alpha,\lm)}}=-\rr(-\lm))\qquad (\lm\in\mathfrak{h}^*),
\]
with $\rr$ Felder's trigonometric classical dynamical $r$-matrix (see \eqref{Frrr}). 
\end{rema}

%%%%%%%%%%%%%%%%%%%%%%%%%%%%%%%%%%%%%%%%%%%%%%%
\subsection{Coupled classical dynamical $r$-and $k$-matrices associated to $r^{\textup{Sch}}$}\label{Sl}
%%%%%%%%%%%%%%%%%%%%%%%%%%%%%%%%%%%%%%%%%%%%%%%
Let $\textup{Aut}(\mathfrak{g},\mathfrak{h})$ be the group of automorphisms of $\mathfrak{g}$ stabilising the Cartan subalgebra $\mathfrak{h}$. Denote by $\textup{Aut}^+(\mathfrak{g},\mathfrak{h})$ the normal subgroup consisting of
 automorphisms $\theta\in\textup{Aut}(\mathfrak{g},\mathfrak{h})$ satisfying $\theta|_{\mathfrak{h}}=
 \textup{id}_{\mathfrak{h}}$. We denote the extension of  an automorphism $\theta\in\textup{Aut}(\mathfrak{g})$ to an automorphism of $U(\mathfrak{g})$ by $\theta$ again. By the isomorphism theorem \cite[Thm. 14.2]{Hu} we have
 \[
 \mathfrak{h}/2\pi i P^\vee\simeq \textup{Aut}^+(\mathfrak{g},\mathfrak{h}),\qquad 
 \overline{y}:=y+2\pi i P^\vee\mapsto \textup{Ad}_y,
 \]
with $P^\vee\subset\mathfrak{h}$ the co-weight lattice and $\textup{Ad}_y\in\textup{Aut}^+(\mathfrak{g},\mathfrak{h})$ for $y\in\mathfrak{h}$ characterised be $\textup{Ad}_y|_{\mathfrak{g}_\alpha}=e^{\alpha(y)}\textup{id}_{\mathfrak{g}_\alpha}$ for all $\alpha\in R$. The isomorphism theorem also gives the group isomorphism
\begin{equation}\label{isoAut}
\textup{Aut}(\mathfrak{g},\mathfrak{h})/\textup{Aut}^+(\mathfrak{g},\mathfrak{h})\overset{\sim}{\longrightarrow}\textup{Aut}(R), \qquad\theta\textup{Aut}^+(\mathfrak{g},\mathfrak{h})\mapsto
{}^t\theta^{-1}
\end{equation}
with ${}^t\theta\in\textup{Gl}(\mathfrak{h}^*)$ the transpose of $\theta|_{\mathfrak{h}}\in\textup{Gl}(\mathfrak{h})$. Recall furthermore that 
\[
\textup{Aut}(R)\simeq \Xi\ltimes W
\]
with $W$ the Weyl group of $R$ and $\Xi$ the group of Dynkin diagram automorphisms (extended to automorphisms of $R$ in the natural way), see, e.g., \cite[\S 12.2]{Hu}.

Write $\textup{Inv}(\mathfrak{g},\mathfrak{h})\subseteq\textup{Aut}(\mathfrak{g},\mathfrak{h})$ for
the subset of involutive automorphisms of $\mathfrak{g}$ stabilising $\mathfrak{h}$. Note that
$\textup{Inv}(\mathfrak{g},\mathfrak{h})$ is invariant under conjugation by $\textup{Aut}(\mathfrak{g},\mathfrak{h})$.
For $\theta\in\textup{Inv}(\mathfrak{g},\mathfrak{h})$ denote by $\mathfrak{g}^{\pm}_\theta\subseteq\mathfrak{g}$ the $(\pm 1)$-eigenspaces of $\theta$. Then
\[
\mathfrak{g}=\mathfrak{g}^+_\theta\oplus\mathfrak{g}^-_\theta,\qquad \phi(\mathfrak{g}^{\pm}_\theta)=\mathfrak{g}^{\pm}_{\phi\theta\phi^{-1}} \quad (\phi\in\textup{Aut}(\mathfrak{g},\mathfrak{h})).
\]
Writing $\mathfrak{h}^{\pm}_\theta:=\mathfrak{g}^{\pm}_\theta\cap\mathfrak{h}$, we also have $\mathfrak{h}=\mathfrak{h}^+_\theta\oplus\mathfrak{h}^-_\theta$ and $\phi(\mathfrak{h}^{\pm}_\theta)=\mathfrak{h}^{\pm}_{\phi\theta\phi^{-1}}$. A special class of involutions are the Chevalley involutions:
%%%%%%%%%%%%%%%%%%%%%%%%%%%%%%%%%%%%%%%%%%%%%
\begin{defi}
An involution $\theta\in\textup{Inv}(\mathfrak{g},\mathfrak{h})$ is called a Chevalley involution of $\mathfrak{g}$ relative to $\mathfrak{h}$ if $\mathfrak{h}^-_\theta=\mathfrak{h}$.
The set of Chevalley involutions relative to $\mathfrak{h}$ will be denoted by $\textup{Ch}(\mathfrak{g},\mathfrak{h})$.
\end{defi}
%%%%%%%%%%%%%%%%%%%%%%%%%%%%%%%%%%%%%%%%%%%%%
The following lemma is well known (it is an easy consequence of the isomorphism theorem \cite[Thm. 14.2]{Hu}).
%%%%%%%%%%%%%%%%%%%%%%%%%%%%%%%%%%%%%%%%%%%%%%
\begin{lem}\label{descriptionCh}
$\textup{Ch}(\mathfrak{g},\mathfrak{h})$ is the $\textup{Aut}^+(\mathfrak{g},\mathfrak{h})$-coset in $\textup{Aut}(\mathfrak{g},\mathfrak{h})$
corresponding to $-1\in\textup{Aut}(R)$ under the isomorphism \eqref{isoAut}. Furthermore, $\textup{Ch}(\mathfrak{g},\mathfrak{h})$ is a single $\textup{Aut}^+(\mathfrak{g},\mathfrak{h})$-orbit 
for the conjugation action of $\textup{Aut}^+(\mathfrak{g},\mathfrak{h})$ on $\textup{Inv}(\mathfrak{g},\mathfrak{h})$.
\end{lem}
%%%%%%%%%%%%%%%%%%%%%%%%%%%%%%%%%%%%%%%%%%%%%%

We consider the following special class of classical dynamical $r$-matrix pairs and related classical dynamical $k$-matrices.
%%%%%%%%%%%%%%%%%%%%%%%%%%%%%%%%%%%%%%%%%%%%%%%%
\begin{defi}\label{gaugedef}
Let $\theta\in\textup{Inv}(\mathfrak{g},\mathfrak{h})$, and fix a subspace
$\mathfrak{a}\subseteq\mathfrak{h}^-_{\theta}$.

We denote by $\mathcal{M}(\theta,\mathfrak{a})$ the set of triples
$(r^+,r^-,\kappa)$ of meromorphic functions $r^{\pm}:\mathfrak{a}^*\rightarrow \mathfrak{g}^{\pm}_{\theta}\otimes \mathfrak{g}$ and $\kappa: \mathfrak{a}^*\rightarrow U(\mathfrak{g}_\theta^+)\otimes U(\mathfrak{g})\otimes U(\mathfrak{g}_\theta^+)$
satisfying the following two conditions:
\begin{enumerate}
\item $(r^+,r^-)$ is a classical dynamical $r$-matrix pair satisfying the $\mathfrak{h}$-compatibility condition
\[
[x\otimes 1,r^+]=[1\otimes x,r^-]\qquad \forall\, x\in\mathfrak{h}.
\]
\item $\kappa$ is a classical dynamical $k$-matrix relative to $(r^+,r^-)$.
\end{enumerate}
\end{defi}
%%%%%%%%%%%%%%%%%%%%%%%%%%%%%%%%%%%%%%%%%%%%%%%%%
We give here two natural classes of gauge transformations for triples $(r^+,r^-,\kappa)$ in $\mathcal{M}(\theta,\mathfrak{a})$ (a more thorough analysis of gauge transformations and the corresponding classification problems, as in \cite{ES,S} for the ordinary classical dynamical Yang-Baxter equations, will not be pursued further in this paper).
%%%%%%%%%%%%%%%%%%%%%%%%%%%%%%%%%%%%%%%%%%%%%%%%%
\begin{lem}\label{gaugelem}
Fix $\theta\in\textup{Inv}(\mathfrak{g},\mathfrak{h})$ and a subspace $\mathfrak{a}\subseteq\mathfrak{h}^-_\theta$.
For $(r^+,r^-,\kappa)\in\mathcal{M}(\theta,\mathfrak{a})$ we have
\begin{enumerate}
\item $({}^{\phi}r^+, {}^{\phi}r^-,{}^{\phi}\kappa)\in\mathcal{M}(\phi\theta\phi^{-1},\phi(\mathfrak{a}))$ for $\phi\in\textup{Aut}(\mathfrak{g},\mathfrak{h})$,
where 
\[
{}^{\phi}r^{\pm}:=(\phi\otimes\phi)\circ r^{\pm}\circ {}^t(\phi\vert_{\mathfrak{a}}),\qquad
{}^{\phi}\kappa:=(\phi\otimes\phi\otimes\phi)\circ\kappa\circ {}^t(\phi\vert_{\mathfrak{a}}),
\]
with ${}^t(\phi\vert_{\mathfrak{a}}): \phi(\mathfrak{a})^*\rightarrow\mathfrak{a}^*$ the transpose of the linear isomorphism $\phi\vert_{\mathfrak{a}}: \mathfrak{a}\overset{\sim}{\longrightarrow}\phi(\mathfrak{a})$.
\item $(r^+_{\epsilon,\mu},r^-_{\epsilon,\mu},\kappa_{\epsilon,\mu})\in\mathcal{M}(\theta,\mathfrak{a})$ for $(\epsilon,\mu)\in\{\pm 1\}\times\mathfrak{a}^*$, where 
\[
r^{\pm}_{\epsilon,\mu}(\lm):=\epsilon r^\pm(\epsilon\lm+\mu),\qquad
\kappa_{\epsilon,\mu}(\lm):=\epsilon\kappa(\epsilon\lm+\mu).
\]
\end{enumerate}
\end{lem}
%%%%%%%%%%%%%%%%%%%%%%%%%%
\begin{proof}
(1) By a direct computation one verifies that
\begin{equation*}
\begin{split}
(\phi\otimes\phi\otimes\phi)\circ\textup{CYB}[t](r^+,r^-)\circ{}^t(\phi\vert_{\mathfrak{a}})&=
\textup{CYB}[t]({}^{\phi}r^+,{}^{\phi}r^-),\\
(\phi\otimes\phi\otimes\phi\otimes\phi)\circ\textup{CR}(r^+,r^-;\kappa)\circ{}^t(\phi\vert_{\mathfrak{a}})&=
\textup{CR}({}^{\phi}r^+,{}^{\phi}r^-;{}^{\phi}\kappa)
\end{split}
\end{equation*}
for $1\leq t\leq 3$, from which the result immediately follows.\\
(2) This follows similarly to (1), now using that 
\begin{equation*}
\begin{split}
\textup{CYB}[t](r^+,r^-)(\epsilon\lm+\mu)&=
\textup{CYB}[t](r_{\epsilon,\mu}^+,r_{\epsilon,\mu}^-)(\lm),\\
\textup{CR}(r^+,r^-;\kappa)(\epsilon\lm+\mu)&=
\textup{CR}(r^+_{\epsilon,\mu},r^-_{\epsilon,\mu};\kappa_{\epsilon,\mu})(\lm)
\end{split}
\end{equation*}
for $1\leq t\leq 3$.
\end{proof}
%%%%%%%%%%%%%%%%%%%%%%%%%

We now construct explicit solutions $(r^+,r^-,\kappa)\in\mathcal{M}(\theta,\mathfrak{a})$ by folding and contracting
Schiffmann's \cite{S} classical dynamical $r$-matrices $r^{\textup{Sch}}$. As we have seen in Section \ref{foldingSection}, the key additional requirement for this to work is that $r^{\textup{Sch}}$ is $\theta$-twisted symmetric.

First note that $\varpi_{\mathfrak{h}}$ is $\theta$-twisted symmetric for any involution
$\theta\in\textup{Inv}(\mathfrak{g},\mathfrak{h})$. Indeed, the Killing form $K(\cdot,\cdot)$ is $\theta$-invariant, hence $\mathfrak{h}=\mathfrak{h}^+_\theta\oplus\mathfrak{h}^-_\theta$ is an orthogonal direct sum with respect to $(\cdot,\cdot)$. Then
\[
\varpi_{\mathfrak{h}}=\varpi_{\mathfrak{h}^+_\theta}+\varpi_{\mathfrak{h}^-_\theta}
\]
with $\varpi_{\mathfrak{h}_\theta^{\pm}}\in S^2\mathfrak{h}^{\pm}_\theta$ representing the nondegenerate symmetric bilinear form $(\cdot,\cdot)|_{\mathfrak{h}^{\pm}_\theta\times\mathfrak{h}^{\pm}_\theta}$.  Hence $\widetilde{\varpi}_{\mathfrak{h}}=(\theta\otimes\textup{id})(\varpi_{\mathfrak{h}})=
\varpi_{\mathfrak{h}^+_\theta}-\varpi_{\mathfrak{h}^-_\theta}$ is symmetric. 

Let $\Gamma\subseteq\Delta$. Write $V_\Gamma$ for the orthocomplement of the subspace
$\bigcap_{\alpha\in\Gamma}\textup{Ker}(\alpha)\subseteq\mathfrak{h}$ with respect to $(\cdot,\cdot)$. Note that $V_\Gamma\subseteq\mathfrak{h}$ is a subspace of real type.
%%%%%%%%%%%%%%%%%%%%%%%%%%%%%%%%%%%%%%%%%%%%%
\begin{thm}\label{thmExplicit} Let $\theta\in\textup{Inv}(\mathfrak{g},\mathfrak{h})$ and fix a subspace $\mathfrak{a}\subseteq\mathfrak{h}$ of real type. Suppose that $(\Gamma_1,\Gamma_2,\tau)$ is an $\mathfrak{a}$-admissible generalised Belavin-Drinfeld triple. Let
$r_{\mathfrak{t}}\in\mathcal{S}$. Write $r^{\textup{Sch}}: \mathfrak{a}^*\rightarrow (\mathfrak{g}\otimes\mathfrak{g})^{\mathfrak{a}}$ for the associated 
classical dynamical $r$-matrix, see \eqref{rES}. The following two statements are equivalent:
\begin{enumerate}
\item $r^{\textup{Sch}}$ is $\theta$-twisted symmetric.
\item  $(\Gamma_1,\Gamma_2,\tau)=(\Gamma,\Gamma,\textup{id})$ with $\Gamma\subseteq\Delta$ a subset such that $V_\Gamma\subseteq\mathfrak{a}$,
$r_{\mathfrak{t}}\in S^2\mathfrak{t}$ and $\theta\in\textup{Ch}(\mathfrak{g},\mathfrak{h})$ \textup{(}i.e., $\theta$ is a Chevalley involution relative to $\mathfrak{h}$\textup{)}.
\end{enumerate}
\end{thm}
%%%%%%%%%%%%%%%%%%%%%%%%%%%%%%%%%%%%%%%%%%%%%
\begin{proof}
(1) $\Rightarrow$ (2): rewrite $r^{\textup{Sch}}$ as 
\begin{equation}\label{rESstart}
\begin{split}
r^{\textup{Sch}}(\lm)=\frac{1}{2}\varpi_{\mathfrak{h}}&+r_{\mathfrak{t}}+
\sum_{\alpha\in R_{\Gamma_1}^+}\bigl(\varphi_\alpha(\lm)+e_\alpha\bigr)\otimes e_{-\alpha}\\
&-\sum_{\alpha\in R_{\Gamma_1}^+}e_{-\alpha}\otimes\varphi_\alpha(\lm)+
\sum_{\alpha\in R^+\setminus R_{\Gamma_1}^+}e_\alpha\otimes e_{-\alpha}
\end{split}
\end{equation}
using the explicit expression for $\varpi$. 
Note that ${}^t\theta\in\textup{Aut}(R)$ is the involutive automorphism such that $\theta(\mathfrak{g}_\alpha)=\mathfrak{g}_{{}^t\theta(\alpha)}$ for all $\alpha\in R$. Recalling that $\varpi_{\mathfrak{h}}$ is $\theta$-twisted symmetric and noting that $\varphi_\alpha$ takes values in $\bigoplus_{\beta\in R^+}\mathfrak{g}_\beta$ for $\alpha\in R_{\Gamma_1}^+$, it follows from \eqref{rESstart} that the $\theta$-twisted symmetry of $r^{\textup{Sch}}$
implies that 
\begin{enumerate}
\item[(a)] $r_{\mathfrak{t}}$ is $\theta$-twisted symmetric,
\item[(b)] $\varphi_\alpha(\lm)+e_\alpha\in\mathfrak{g}_{-{}^t\theta(\alpha)}$ for $\alpha\in
R_{\Gamma_1}^+$,
\item[(c)] $\varphi_\alpha(\lm)\in\mathfrak{g}_{-{}^t\theta(\alpha)}$ for $\alpha\in R_{\Gamma_1}^+$,
\item[(d)] ${}^t\theta(\alpha)=-\alpha$ for $\alpha\in R^+\setminus R_{\Gamma_1}^+$. 
\end{enumerate}
Suppose that $\alpha\in R_{\Gamma_1}^+\setminus R_{\Gamma_1,\tau}^+$. Let $k\in\mathbb{Z}_{>0}$ be the smallest positive
integer such that $\tau^k(\alpha)\in R^+\setminus R_{\Gamma_1}^+$.  
Then
\[
\varphi_\alpha(\lm)+e_{\alpha}=\sum_{j=0}^ke^{-j(\alpha,\lm)}e_{\tau^j(\alpha)}
\]
and $\alpha,\tau(\alpha),\tau^2(\alpha),\ldots,\tau^k(\alpha)$ are pairwise different roots in $R^+$. This violates property (b). 
Hence $R_{\Gamma_1}^+=R_{\Gamma_1,\tau}^+$, in particular, $\Gamma_2=\Gamma_1$ (which we denote by $\Gamma$ from now on). 

Fix $\alpha\in R_{\Gamma}^+$ and denote by
$k<\ell$ the positive integers with $k+\ell$ as small as possible such that  $\tau^k(\alpha)=\tau^\ell(\alpha)$. Then $\tau(\alpha),\ldots,\tau^{\ell-1}(\alpha)$ are pairwise distinct
roots in $R_{\Gamma}^+$, hence by \eqref{varphi2} and property (c) we must have $(k,\ell)=(1,2)$.
In particular, $\tau^2(\alpha)=\tau(\alpha)$,
\[
\varphi_\alpha(\lm)=\frac{e^{-(\alpha,\lm)}}{1-e^{-(\alpha,\lm)}}\,e_{\tau(\alpha)}
\]
and $\tau(\alpha)=-{}^t\theta(\alpha)$. Property (b) then implies that $\tau(\alpha)=\alpha$. 

We conclude that $(\Gamma_1,\Gamma_2,\tau)=(\Gamma,\Gamma,\textup{id})$
and ${}^t\theta(\alpha)=-\alpha$ for all $\alpha\in R^+_\Gamma$. By property (d) we then have
${}^t\theta(\alpha)=-\alpha$ for all $\alpha\in R$. Hence $\theta|_{\mathfrak{h}}=-\textup{id}_{\mathfrak{h}}$, showing that $\theta\in\textup{Ch}(\mathfrak{g},\mathfrak{h})$ 
and that $\widetilde{r}_{\mathfrak{t}}=-r_{\mathfrak{t}}$. Property (a) now forces $r_{\mathfrak{t}}\in S^2\mathfrak{t}$. Finally, the $\mathfrak{a}$-admissibility of $(\Gamma,\Gamma,\textup{id})$
is equivalent to the condition that $\alpha|_{\mathfrak{t}}=0$ for all $\alpha\in\Gamma$, hence 
$V_\Gamma\subseteq\mathfrak{a}$.\\
(2) $\Rightarrow$ (1): $(\Gamma,\Gamma,\textup{id})$ is $\mathfrak{a}$-admissible by the assumption that $V_\Gamma\subseteq\mathfrak{a}$.
 A direct computation using the explicit expression for $\varpi$ and the fact that
\[
\varphi_\alpha(\lm)=\frac{e^{-(\alpha,\lm)}}{1-e^{-(\alpha,\lm)}}\,e_\alpha,\qquad
\alpha\in R_\Gamma,
\]
shows that 
\[
r^{\textup{Sch}}(\lm)=\frac{1}{2}\varpi_{\mathfrak{h}}+
r_{\mathfrak{t}}+\sum_{\alpha\in R_\Gamma}\frac{e_\alpha\otimes e_{-\alpha}}{1-e^{-(\alpha,\lm)}}
+\sum_{\alpha\in R^+\setminus R_\Gamma^+}e_\alpha\otimes e_{-\alpha}\qquad (\lm\in\mathfrak{a}^*).
\]
Since $\theta\in\textup{Ch}(\mathfrak{g},\mathfrak{h})$ we have $\theta(e_\alpha)=c_\alpha e_{-\alpha}$ for 
some $c_\alpha\in\mathbb{C}^\times$ ($\alpha\in R$), hence $\widetilde{r}^{\textup{Sch}}=(\theta\otimes\textup{id})r^{\textup{Sch}}$ is explicitly given by
\[
\widetilde{r}^{\textup{Sch}}(\lm)=
-\frac{1}{2}\varpi_{\mathfrak{h}}-
r_{\mathfrak{t}}+\sum_{\alpha\in R_\Gamma}\frac{c_\alpha e_{-\alpha}\otimes e_{-\alpha}}{1-e^{-(\alpha,\lm)}}
+\sum_{\alpha\in R^+\setminus R_\Gamma^+}c_\alpha e_{-\alpha}\otimes e_{-\alpha},
\]
which is symmetric since $r_{\mathfrak{t}}\in S^2\mathfrak{t}$.
\end{proof}
%%%%%%%%%%%%%%%%%%%%%%%%%%%%%%%%%%%%%%%%%%%%
\begin{rema}
Theorem \ref{thmExplicit} shows that the folding and contraction procedure of Section \ref{foldingSection} in case $U=U(\mathfrak{g})$ 
is naturally related to the split real form of $\mathfrak{g}$. 
This should be compared with \cite{RS}, where it was shown that the asymptotic KZB equations for $N$-point spherical functions for real semisimple Lie groups $G$ are governed by triples $(r^+,r^-,\kappa)$ satisfying {\it glued} versions of the coupled classical dynamical Yang-Baxter and reflection equations, which only decouple when $G$ is split. In the non-split case the triple $(r^+,r^-,\kappa)$ is obtainable from a single $r: \mathfrak{a}^*\rightarrow\mathfrak{g}\otimes\mathfrak{g}$ by folding and contraction along the Cartan involution $\theta$, but $r$ does not satisfy the classical dynamical Yang-Baxter equation.
\end{rema}
%%%%%%%%%%%%%%%%%%%%%%%%%%%%%%%%%%%%%%%%%%%%
It is convenient to make the dependence of Schiffmann's classical dynamical $r$-matrix $r^{\textup{Sch}}$ 
on the initial data explicit for the special class of generalised Belavin-Drinfeld triples arising in Theorem \ref{thmExplicit}.
Recall from Lemma \ref{rt} that $S^2\mathfrak{t}\subseteq\mathcal{S}$
for $\mathfrak{a}$-admissible Belavin-Drinfeld triples of the form $(\Gamma,\Gamma,\textup{id})$. 

%%%%%%%%%%%%%%%%%%%%%%%%%%%%%%%%%%%%%%%%%%%%%%
\begin{defi}\label{tripledef}
Let $(\Gamma,\mathfrak{a},r_{\mathfrak{t}})$ be a triple consisting of
\begin{enumerate}
\item a subset $\Gamma\subseteq\Delta$, 
\item a subspace $\mathfrak{a}\subseteq\mathfrak{h}$ of real type containing $V_\Gamma$,
\item an element $r_{\mathfrak{t}}\in S^2\mathfrak{t}$.
\end{enumerate}
We write
\begin{equation}\label{TheExamples}
r_{(\Gamma,\mathfrak{a},r_{\mathfrak{t}})}(\lm)=\frac{1}{2}\varpi_{\mathfrak{h}}+
r_{\mathfrak{t}}+\sum_{\alpha\in R_\Gamma}\frac{e_\alpha\otimes e_{-\alpha}}{1-e^{-(\alpha,\lm)}}
+\sum_{\alpha\in R^+\setminus R_\Gamma^+}e_\alpha\otimes e_{-\alpha} \qquad (\lm\in\mathfrak{a}^*)
\end{equation}
for the $\mathfrak{a}$-invariant classical dynamical $r$-matrix $r^{\textup{Sch}}: \mathfrak{a}^*\rightarrow (\mathfrak{g}\otimes\mathfrak{g})^{\mathfrak{a}}$
associated to the $\mathfrak{a}$-admissible generalised
Belavin-Drinfeld triple $(\Gamma,\Gamma,\textup{id})$ and the element $r_{\mathfrak{t}}\in S^2\mathfrak{t}$.
\end{defi}
%%%%%%%%%%%%%%%%%%%%%%%%%%%%%%%%%%%%%%%%%%%%%%
Note that 
\[
r_{(\Gamma,\mathfrak{a},r_{\mathfrak{t}})}+(r_{(\Gamma,\mathfrak{a},r_{\mathfrak{t}})})_{21}=\varpi+2r_{\mathfrak{t}},
\]
in particular
$r_{(\Gamma,\mathfrak{a},r_{\mathfrak{t}})}$ is quasi-unitary with coupling constant $1$ iff $r_{\mathfrak{t}}=0$. Note furthermore that all terms in the sum over $\alpha\in R_{\Gamma}$ 
in \eqref{TheExamples} depend nontrivially on $\lm\in\mathfrak{a}^*$ since $\alpha|_{\mathfrak{a}}\not=0$ for all
$\alpha\in R_\Gamma$ by the assumption that 
$V_\Gamma\subseteq\mathfrak{a}$. Finally, by Theorem \ref{thmExplicit}, the $r_{(\Gamma,\mathfrak{a},r_{\mathfrak{t}})}$ form the subclass of 
classical dynamical $r$-matrices
$r^{\textup{Sch}}$ 
that are $\theta$-twisted symmetric for all $\theta\in\textup{Ch}(\mathfrak{g},\mathfrak{h})$.
Note that for $\Gamma=\Delta$ we are forced to take 
$\mathfrak{a}=\mathfrak{h}$, hence $r_{\mathfrak{t}}=0$.
Then $r_{(\Delta,\mathfrak{h},0)}: \mathfrak{h}^*\rightarrow (\mathfrak{g}\otimes\mathfrak{g})^{\mathfrak{h}}$ is essentially Felder's classical dynamical $r$-matrix $\rr$,
see Remark \ref{relationtoSR}. Recall that $r_{(\Gamma,\mathfrak{a},r_{\mathfrak{t}})}$ does not depend on the choice of $e_\alpha\in\mathfrak{g}_\alpha$ ($\alpha\in R$) such that $[e_\alpha,e_{-\alpha}]=t_\alpha$
for all $\alpha\in R$. 

We next explicitly describe the dependence of the classical dynamical $r$-matrix pairs and the associated classical dynamical $k$-matrices, obtained from
$r^{(\Gamma,\mathfrak{a},r_{\mathfrak{t}})}$ by folding and contracting, on the normalisation of the root vectors and on the Chevalley involution. 
For this it is convenient
to choose the $e_\alpha$'s to be well behaved with respect to a fixed Chevalley involution:
%%%%%%%%%%%%%%%%%%%%%%%%%%%%%%%%%%%%%%%%%%%%%%
\begin{conv}
In the remainder of this section we fix a Chevalley involution $\sigma\in\textup{Ch}(\mathfrak{g},\mathfrak{h})$ and fix root vectors $e_\alpha\in\mathfrak{g}_\alpha$ \textup{(}$\alpha\in R$\textup{)} such that $[e_\alpha,e_{-\alpha}]=t_\alpha$ and $\sigma(e_\alpha)=-e_{-\alpha}$ for all $\alpha\in R$ \textup{(}this is the same convention as used in Section \ref{Se1}\textup{)}.
We denote the other Chevalley involutions in $\textup{Aut}(\mathfrak{g},\mathfrak{h})$ by
\[
\sigma_{\overline{y}}:=\sigma\circ\textup{Ad}_y,\qquad y\in\mathfrak{h}
\]
\textup{(}cf. Lemma \ref{descriptionCh}\textup{)}. 
\end{conv}
%%%%%%%%%%%%%%%%%%%%%%%%%%%%%%%%%%%%%%%%%%%%%%
We now come to the following main result of this subsection.
%%%%%%%%%%%%%%%%%%%%%%%%%%%%%%%%%%%%%%%%%%%%%%%
\begin{thm}\label{mainTHMexplicit}
Fix $(\Gamma,\mathfrak{a},r_{\mathfrak{t}})$ satisfying the conditions \textup{(1)-(3)} from Definition \ref{tripledef}, and fix
$\overline{y}\in\mathfrak{h}/2\pi i P^\vee$.
Define $r_{(\Gamma,\mathfrak{a},r_{\mathfrak{t}})}^{\pm}(\cdot;\overline{y}):
\mathfrak{a}^*\rightarrow \mathfrak{g}_{\sigma_{\overline{y}}}^{\pm}\otimes\mathfrak{g}$ and 
$\kappa_{(\Gamma,\mathfrak{a},r_{\mathfrak{t}})}^{\textup{core}}(\cdot;\overline{y}): \mathfrak{a}^*\rightarrow U(\mathfrak{g})$ by
\[
r_{(\Gamma,\mathfrak{a},r_{\mathfrak{t}})}^{\pm}(\lm;\overline{y}):=
\frac{1}{2}\bigl(\widetilde{r}_{(\Gamma,\mathfrak{a},r_{\mathfrak{t}})}(\lm;\overline{y})\pm
r_{(\Gamma,\mathfrak{a},r_{\mathfrak{t}})}(\lm)\bigr),\qquad
\kappa_{(\Gamma,\mathfrak{a},r_{\mathfrak{t}})}^{\textup{core}}(\lm;\overline{y}):=
\frac{1}{2}m\bigl(\widetilde{r}_{(\Gamma,\mathfrak{a},r_{\mathfrak{t}})}(\lm,\overline{y})\bigr),
\]
where 
\[
\widetilde{r}_{(\Gamma,\mathfrak{a},r_{\mathfrak{t}})}(\lm;\overline{y}):=
(\sigma_{\overline{y}}\otimes\textup{id})(r_{(\Gamma,\mathfrak{a},r_{\mathfrak{t}})}(\lm))\qquad\quad (\lm\in\mathfrak{a}^*).
\]
Then 
\[
\bigl(r_{(\Gamma,\mathfrak{a},r_{\mathfrak{t}})}^{+}(\cdot;\overline{y}),
r_{(\Gamma,\mathfrak{a},r_{\mathfrak{t}})}^{-}(\cdot;\overline{y}),
\kappa_{(\Gamma,\mathfrak{a},r_{\mathfrak{t}})}^{\textup{core}}(\cdot;\overline{y})\bigr)\in
\mathcal{M}(\sigma_{\overline{y}},\mathfrak{a}).
\]
\end{thm}
%%%%%%%%%%%%%%%%%%%%%%%%%%%%%%%%%%%%%%%%%%%%%%
\begin{proof}
Since $\sigma_{\overline{y}}$ is a Chevalley involution we have $\mathfrak{a}\subseteq\mathfrak{h}_{\sigma_{\overline{y}}}^-=\mathfrak{h}$. Furthermore, $r_{(\Gamma,\mathfrak{a},r_{\mathfrak{t}})}$ is $\sigma_{\overline{y}}$-twisted symmetric by Theorem \ref{thmExplicit}, and $\mathfrak{a}$-invariant. Hence the folding and contracting results from Section
\ref{foldingSection} can be applied to $r_{(\Gamma,\mathfrak{a},r_{\mathfrak{t}})}$ relative to the Chevalley involution $\sigma_{\overline{y}}$. The corollary then follows from Theorem \ref{equivcor} and Theorem \ref{BYB1}.
\end{proof}
%%%%%%%%%%%%%%%%%%%%%%%%%%%%%%%%%%%%%%%%%%%%%%
%%%%%%%%%%%%%%%%%%%%%%%%%%%%%%%%%%%%%%%%%%%
Concretely, we have
\begin{equation*}
\mathfrak{g}^+_{\sigma_{\overline{y}}}=\bigoplus_{\alpha\in R^+}\mathbb{C}(e_\alpha-e^{\alpha(y)}e_{-\alpha}),\qquad
\mathfrak{g}^-_{\sigma_{\overline{y}}}=\mathfrak{h}\oplus\bigoplus_{\alpha\in R^+}\mathbb{C}(e_\alpha+e^{\alpha(y)}e_{-\alpha})
\end{equation*}
and
$\widetilde{r}_{(\Gamma,\mathfrak{a},r_{\mathfrak{t}})}(\cdot;\overline{y}): \mathfrak{a}^*\rightarrow\mathfrak{g}
\otimes\mathfrak{g}$
is explicitly given
by
\begin{equation}\label{TheExamplestilde}
\widetilde{r}_{(\Gamma,\mathfrak{a},r_{\mathfrak{t}})}(\lm;\overline{y})=-\frac{1}{2}\varpi_{\mathfrak{h}}-
r_{\mathfrak{t}}-\sum_{\alpha\in R_\Gamma}\frac{e^{\alpha(y)}e_{-\alpha}\otimes e_{-\alpha}}{1-e^{-(\alpha,\lm)}}
-\sum_{\alpha\in R^+\setminus R_\Gamma^+}e^{\alpha(y)}e_{-\alpha}\otimes e_{-\alpha}.
\end{equation}
The folded $r$-matrices $r_{(\Gamma,\mathfrak{a},r_{\mathfrak{t}})}^{\pm}(\cdot;\overline{y}): \mathfrak{a}^*
\rightarrow\mathfrak{g}_{\sigma_{\overline{y}}}^{\pm}\otimes\mathfrak{g}$ are given by 
\begin{equation}\label{rpm}
\begin{split}
r_{(\Gamma,\mathfrak{a},r_{\mathfrak{t}})}^{+}(\lm;\overline{y})&=
\frac{1}{2}\sum_{\alpha\in R_{\Gamma}}\frac{(e_\alpha-e^{\alpha(y)}e_{-\alpha})\otimes e_{-\alpha}}
{1-e^{-(\alpha,\lm)}}+\frac{1}{2}\sum_{\alpha\in R^+\setminus R_\Gamma^+}(e_\alpha-e^{\alpha(y)}e_{-\alpha})\otimes e_{-\alpha},\\
r_{(\Gamma,\mathfrak{a},r_{\mathfrak{t}})}^{-}(\lm;\overline{y})&=
-\frac{1}{2}\varpi_{\mathfrak{h}}-r_{\mathfrak{t}}\\
&-\frac{1}{2}\sum_{\alpha\in R_\Gamma}
\frac{(e_\alpha+e^{\alpha(y)}e_{-\alpha})\otimes e_{-\alpha}}{1-e^{-(\alpha,\lm)}}-
\frac{1}{2}\sum_{\alpha\in R^+\setminus R_\Gamma^+}(e_\alpha+e^{\alpha(y)}e_{-\alpha})\otimes e_{-\alpha},
\end{split}
\end{equation}
and the associated core classical dynamical $k$-matrix $\kappa_{(\Gamma,\mathfrak{a},r_{\mathfrak{t}})}^{\textup{core}}(\cdot;\overline{y}): \mathfrak{a}^*\rightarrow
U(\mathfrak{g})$ is
\begin{equation}\label{kk}
\kappa_{(\Gamma,\mathfrak{a},r_{\mathfrak{t}})}^{\textup{core}}(\lm;\overline{y})=
-\frac{1}{4}\Omega_{\mathfrak{h}}-\frac{1}{2}\Omega_{\mathfrak{t}}-
\frac{1}{2}\sum_{\alpha\in R_{\Gamma}}\frac{e^{\alpha(y)}e_{-\alpha}^2}{1-e^{-(\alpha,\lm)}}-
\frac{1}{2}\sum_{\alpha\in R^+\setminus R_{\Gamma}^+}e^{\alpha(y)}e_{-\alpha}^2
\end{equation}
with $\Omega_{\mathfrak{h}}:=m(\varpi_{\mathfrak{h}})$ and 
$\Omega_{\mathfrak{t}}:=m(r_{\mathfrak{t}})$. 

A different choice of Chevalley involution results in gauge equivalent solutions (see Lemma \ref{gaugelem} for the notations):
%%%%%%%%%%%%%%%%%%%%%%%%%%%%%%%%%%%%%%%%%%%%%
\begin{cor}\label{xind}
Fix $(\Gamma,\mathfrak{a},r_{\mathfrak{t}})$ satisfying the conditions \textup{(1)-(3)} from Definition \ref{tripledef}, and fix
$\overline{y},\overline{z}\in\mathfrak{h}/2\pi i P^\vee$.
Then
\[
{}^{\textup{Ad}_z}r_{(\Gamma,\mathfrak{a},r_{\mathfrak{t}})}^{\pm}(\cdot;\overline{y})=
r^{(\Gamma,\mathfrak{a},r_{\mathfrak{t}}),\pm}(\cdot;\overline{y-2z}),\qquad
{}^{\textup{Ad}_z}\kappa_{(\Gamma,\mathfrak{a},r_{\mathfrak{t}})}^{\textup{core}}(\cdot;\overline{y})=
\kappa_{(\Gamma,\mathfrak{a},r_{\mathfrak{t}})}^{\textup{core}}(\cdot;\overline{y-2z}).
\]
\end{cor}
%%%%%%%%%%%%%%%%%%%%%%%%%%%%%%%%%%%%%%%%%%%%%
\begin{proof}
This follows by direct inspection using the explicit expressions \eqref{rpm} and \eqref{kk}.
\end{proof}
%%%%%%%%%%%%%%%%%%%%%%%%%%%%%%%%%%%%%%%%%%%%
The core classical dynamical $k$-matrix $\kappa_{(\Gamma,\mathfrak{a},r_{\mathfrak{t}})}^{\textup{core}}(\cdot;\overline{y})$ 
can be dressed up with additional terms involving nontrivial components in $U(\mathfrak{g}_{\sigma_{\overline{y}}}^+)$, see Corollary \ref{BYB2}. The resulting triples will be denoted by
\begin{equation}\label{addK}
\bigl(r_{(\Gamma,\mathfrak{a},r_{\mathfrak{t}})}^{+}(\cdot;\overline{y}),
r_{(\Gamma,\mathfrak{a},r_{\mathfrak{t}})}^{-}(\cdot;\overline{y}),
\kappa_{(\Gamma,\mathfrak{a},r_{\mathfrak{t}})}(\cdot;\overline{y})\bigr)\in
\mathcal{M}(\sigma_{\overline{y}},\mathfrak{a}).
\end{equation}

As we have seen before (see Remark \ref{relationtoSR}), 
the classical dynamical $r$-matrix pair 
\[
(r_{(\Delta,\mathfrak{h},0)}^{+}(\cdot;\overline{y}),r_{(\Delta,\mathfrak{h},0)}^{-}(\cdot;\overline{y}))
\]
and the associated classical dynamical $k$-matrices $\kappa_{(\Delta,\mathfrak{h},0)}(\cdot;\overline{y})$ 
play an important role in the theory of $N$-point spherical functions for split real semisimple Lie groups, 
see Subsection \ref{22} and \cite[\S 6]{SR}. But in \cite[Thm. 6.31]{SR} 
additional terms
to $\kappa_{(\Delta,\mathfrak{h},0)}^{\textup{core}}(\cdot;\overline{y})$ appeared,
which cannot be justified by Corollary \ref{BYB2}. In the following proposition we discuss such twisting algebraically for Felder's classical dynamical $r$-matrix only (we take $y=0$ without loss of generality).
%%%%%%%%%%%%%%%%%%%%%%%%%%%%%%%%
\begin{prop}\label{ktwist}
Write $r:=r_{(\Delta,\mathfrak{h},0)}$ and denote by
$(r^{+},r^-,\kappa^{\textup{core}})\in\mathcal{M}(\sigma,\mathfrak{h})$ the associated triple,
as constructed in Theorem \ref{mainTHMexplicit}. Define twisted meromorphic
functions $\check{r}^{\,\pm}: \mathfrak{h}^*\rightarrow \mathfrak{g}^{\pm}_{\sigma}\otimes\mathfrak{g}$ by
\[
\check{r}^{\,\pm}(\lm):=(\textup{id}\otimes\textup{Ad}_{t_\lm/2})r^{\pm}(\lm).
\]
Then, with the notational conventions as in Lemma \ref{rellem} and Corollary \ref{BYB2},
\begin{enumerate}
\item $\textup{resCR}\bigl(r^+,r^-,\frac{1}{2}(\varpi_{\mathfrak{h}})_{10^\prime}+\check{r}^{\,-}_{10^\prime}\bigr)=0$,
\item 
$\check{r}^{\,+}=\frac{1}{2}\varpi_{\mathfrak{h}}+\check{r}^{\,-}_{21}$,
\item $(r^+,r^-,\kappa_1+\check{r}^{\,+}_{0^{\prime}1}), (r^+,r^-,\kappa_1+r^+_{01}+\check{r}^{\,+}_{0^{\prime}1})\in\mathcal{M}(\sigma,\mathfrak{h})$.
\end{enumerate}
\end{prop}
%%%%%%%%%%%%%%%%%%%%%%%%%%%%%%%%%%%%%%%%%%%
\begin{proof}
(1) Fix a linear basis $\{x_k\}_{k=1}^n$ of $\mathfrak{h}$ satisfying $(x_k,x_\ell)=\delta_{k,l}$. Denote the corresponding dual basis of $\mathfrak{h}^*$ by
$\{\lambda_k\}_{k=1}^n$. Then $\lambda_k=(x_k,\cdot)$, and hence $t_{\lambda_k}=x_k$ ($1\leq k\leq n$).
By a direct computation we have
\begin{equation*}
\begin{split}
\textup{resCR}&\Bigl(r^+,r^-,\frac{1}{2}(\varpi_{\mathfrak{h}})_{10^\prime}+\check{r}^{\,-}_{10^\prime}\Bigr)(\lm)
=(\textup{Ad}_{t_\lm/2})_{0^\prime}\Bigl([r_{10^\prime}^-(\lm),r_{12}^+(\lm)]+[r_{10^\prime}^-(\lm),r_{20^\prime}^-(\lm)]+[r_{12}^-(\lm),r_{20^\prime}^-(\lm)]
\Bigr)\\
&+\frac{1}{2}(\textup{Ad}_{t_\lm/2})_{0^\prime}\Bigl(
[(\varpi_{\mathfrak{h}})_{10^\prime},r_{12}^+(\lm)]+
[(\varpi_{\mathfrak{h}})_{10^\prime},r_{20^\prime}^-(\lm)]-
[(\varpi_{\mathfrak{h}})_{20^\prime},r_{10^\prime}^-(\lm)]-
[(\varpi_{\mathfrak{h}})_{20^\prime},r_{12}^-(\lm)]\Bigr)\\
&+\sum_{k=1}^n\Bigl((x_k)_2(\partial_{\lm_k}\check{r}^{\,-}_{10^\prime})(\lm)-
(x_k)_1(\partial_{\lm_k}\check{r}^{\,-}_{20^\prime})(\lm)\Bigr).
\end{split}
\end{equation*}
Since
\[
\bigl(\partial_{\lm_k}\check{r}^{\,-}\bigr)(\lm)=(\textup{Ad}_{t_\lm/2})_2
\Bigl(\frac{1}{2}[(x_k)_2,r^-(\lm)]+(\partial_{\lm_k}r^-)(\lm)\Bigr)
\]
we have
\begin{equation*}
\begin{split}
\sum_{k=1}^n\Bigl((x_k)_2(\partial_{\lm_k}\check{r}^{\,-}_{10^\prime})(\lm)-
(x_k)_1(\partial_{\lm_k}\check{r}^{\,-}_{20^\prime})(\lm)\Bigr)=
&(\textup{Ad}_{t_\lm/2})_{0^\prime}\Bigl\{\frac{1}{2}[(\varpi_{\mathfrak{h}})_{20^\prime},r_{10^\prime}^-(\lm)]-
\frac{1}{2}[(\varpi_{\mathfrak{h}})_{10^\prime},r_{20^\prime}^-(\lm)]\\
&\quad+\sum_{k=1}^n\Bigl((x_k)_2(\partial_{\lm_k}r_{10^\prime}^-)(\lm)-
(x_k)_1(\partial_{\lm_k}r_{20^\prime}^-)(\lm)\Bigr)\Bigr\}
\end{split}
\end{equation*}
and hence
\begin{equation*}
\begin{split}
&(\textup{Ad}_{t_\lm/2}^{-1})_{0^\prime}\Bigl(\textup{resCR}\bigl(r^+,r^-,\frac{1}{2}(\varpi_{\mathfrak{h}})_{10^\prime}+\check{r}^{\,-}_{10^\prime}\bigr)(\lm)\Bigr)
=\frac{1}{2}\Bigl([(\varpi_{\mathfrak{h}})_{10^\prime},r_{12}^+(\lm)]-[(\varpi_{\mathfrak{h}})_{20^\prime},r_{12}^-(\lm)]\Bigr)\\
&+[r_{10^\prime}^-(\lm),r_{12}^+(\lm)]+[r_{10^\prime}^-(\lm),r_{20^\prime}^-(\lm)]+[r_{12}^-(\lm),r_{20^\prime}^-(\lm)]+\sum_{k=1}^n\Bigl((x_k)_2(\partial_{\lm_k}r_{10^\prime}^-)(\lm)-
(x_k)_1(\partial_{\lm_k}r_{20^\prime}^-)(\lm)\Bigr).
\end{split}
\end{equation*}
By \eqref{invariantplusminus} we have
\[
[(\varpi_{\mathfrak{h}})_{10^\prime},r_{12}^+(\lm)]-[(\varpi_{\mathfrak{h}})_{20^\prime},r_{12}^-(\lm)]=0
\]
and hence
\[
(\textup{Ad}_{t_\lm/2}^{-1})_{0^\prime}\Bigl(\textup{resCR}\bigl(r^+,r^-,\frac{1}{2}(\varpi_{\mathfrak{h}})_{10^\prime}+\check{r}^{\,-}_{10^\prime}\bigr)(\lm)\Bigr)
=\textup{CYB}[3](r^+,r^-)_{120^\prime}(\lm)=0,
\]
as desired.\\
(2) A direct computation shows that
\begin{equation*}
\begin{split}
r^+(\lm)&=\frac{1}{2}\sum_{\alpha\in R^+}\frac{(e_\alpha-e_{-\alpha})\otimes (e_{-\alpha}+
e^{-(\alpha,\lm)}e_\alpha)}{1-e^{-(\alpha,\lm)}},\\
r^-(\lm)&=-\frac{1}{2}\varpi_{\mathfrak{h}}+\frac{1}{2}\sum_{\alpha\in R^+}
\frac{(e_\alpha+e_{-\alpha})\otimes (e^{-(\alpha,\lm)}e_\alpha-e_{-\alpha})}{1-e^{-(\alpha,\lm)}},
\end{split}
\end{equation*}
and consequently
\begin{equation*}
\begin{split}
\check{r}^{\,+}(\lm)&=\frac{1}{2}\sum_{\alpha\in R^+}\frac{(e_\alpha-e_{-\alpha})\otimes
(e_\alpha+e_{-\alpha})}{e^{(\alpha,\lm)/2}-e^{-(\alpha,\lm)/2}},\\
\check{r}^{\,-}(\lm)&=-\frac{1}{2}\varpi_{\mathfrak{h}}+
\frac{1}{2}\sum_{\alpha\in R^+}\frac{(e_\alpha+e_{-\alpha})\otimes (e_\alpha-e_{-\alpha})}
{e^{(\alpha,\lm)/2}-e^{-(\alpha,\lm)/2}},
\end{split}
\end{equation*}
from which the result immediately follows.\\ 
(3) This follows from 
Lemma \ref{rBYB} and parts (1) and (2) of the lemma. 
\end{proof}
%%%%%%%%%%%%%%%%%%%%%%%%%%%%%%%%%
\section{Type $C_N$ asymptotic Gaudin Hamiltonians}\label{AsGa}
%%%%%%%%%%%%%%%%%%%%%%%%%%%%%%%%%%
We keep the notations from the previous section.

Note that $(\emptyset,\mathfrak{a},r_{\mathfrak{t}})$ satisfies the conditions \textup{(1)-(3)} from Definition \ref{tripledef} when $\mathfrak{a}\subseteq\mathfrak{h}$ is a subspace
of real type and $r_{\mathfrak{t}}\in S^2\mathfrak{t}$.
The resulting classical dynamical $r$-matrix
$r_{(\emptyset,\mathfrak{a},r_{\mathfrak{t}})}: \mathfrak{a}^*\rightarrow (\mathfrak{g}\otimes\mathfrak{g})^{\mathfrak{h}}$ is constant,
\[
r_{(\emptyset,\mathfrak{a},r_{\mathfrak{t}})}=\frac{1}{2}\varpi_{\mathfrak{h}}+r_{\mathfrak{t}}+
\sum_{\alpha\in R^+}e_\alpha\otimes e_{-\alpha}=\Omega_++r_{\mathfrak{t}}.
\]
and hence satisfies the classical (non-dynamical) Yang-Baxter equation
\[
[r_{12},r_{13}]+[r_{12},r_{23}]+[r_{13},r_{23}]=0\qquad (r=r_{(\emptyset,\mathfrak{a},r_{\mathfrak{t}})}).
\]
We focus from now on 
\[
r_{\textup{Ga}}:=r_{(\emptyset,\{0\},0)}=\Omega_+,
\]
which is a quasi-unitary classical $r$-matrix with coupling constant one
that lies in $\mathfrak{b}^+\otimes\mathfrak{b}^-$.
Theorem \ref{mainTHMexplicit} applied to $r_{\textup{Ga}}$ gives
\[
(r^+_{\textup{Ga}}(\overline{y}),r^-_{\textup{Ga}}(\overline{y}),\kappa^{\textup{core}}_{\textup{Ga}}(\overline{y}))\in\mathcal{M}(\sigma_{\overline{y}},\{0\})
\]
for $\overline{y}\in\mathfrak{h}/2\pi i P^\vee$ with 
\[
\kappa_{\textup{Ga}}^{\textup{core}}(\overline{y}):=\frac{1}{2}m(\widetilde{r}_{\textup{Ga}}(\overline{y})).
\]
Concretely, we have
\begin{equation}\label{TheExamplestildeG}
\widetilde{r}_{\textup{Ga}}(\overline{y})=-\frac{1}{2}\varpi_{\mathfrak{h}}
-\sum_{\alpha\in R^+}e^{\alpha(y)}e_{-\alpha}\otimes e_{-\alpha},
\end{equation}
and 
\begin{equation}\label{rpmG}
\begin{split}
r_{\textup{Ga}}^{+}(\overline{y})&=
\frac{1}{2}\sum_{\alpha\in R^+}(e_\alpha-e^{\alpha(y)}e_{-\alpha})\otimes e_{-\alpha},\\
r_{\textup{Ga}}^{-}(\overline{y})&=
-\frac{1}{2}\varpi_{\mathfrak{h}}-
\frac{1}{2}\sum_{\alpha\in R^+}(e_\alpha+e^{\alpha(y)}e_{-\alpha})\otimes e_{-\alpha}.
\end{split}
\end{equation}
The associated core classical dynamical $k$-matrix $\kappa_{\textup{Ga}}^{\textup{core}}(\overline{y})\in
U(\mathfrak{g})$ is
\begin{equation}\label{kkG}
\kappa_{\textup{Ga}}^{\textup{core}}(\overline{y})=
-\frac{1}{4}\Omega_{\mathfrak{h}}-
\frac{1}{2}\sum_{\alpha\in R^+}e^{\alpha(y)}e_{-\alpha}^2
\end{equation}
with (recall) $\Omega_{\mathfrak{h}}=m(\varpi_{\mathfrak{h}})$. As in the previous section, 
the classical dynamical $k$-matrix $\kappa_{\textup{Ga}}^{\textup{core}}(\overline{y})$ relative to 
$(r^+_{\textup{Ga}}(\overline{y}),r^-_{\textup{Ga}}(\overline{y}))$ can be dressed up with additional terms, giving rise to triples
\begin{equation}\label{Gaudintriple}
(r^+_{\textup{Ga}}(\overline{y}),r^-_{\textup{Ga}}(\overline{y}),\kappa_{\textup{Ga}}(\overline{y}))\in\mathcal{M}(\sigma_{\overline{y}},\{0\}),
\end{equation}
cf. \eqref{addK}. 
%%%%%%%%%%%%%%%%%%%%%%%%%%%%%%%%%%%%%%%%%%%%%%%%%%%%%%%%
\begin{rema}\label{limitGaudin}
Triples \eqref{Gaudintriple} can also be obtained as limit cases of triples 
\[
\bigl(r_{(\Delta,\mathfrak{h},0)}^{+}(\cdot;\overline{y}),
r_{(\Delta,\mathfrak{h},0)}^{-}(\cdot;\overline{y}),
\kappa_{(\Delta,\mathfrak{h},0}(\cdot;\overline{y})\bigr)\in
\mathcal{M}(\sigma_{\overline{y}},\mathfrak{h}),
\]
see Subsection \ref{DegScheme} and \cite[Rem. 6.4]{SR}.
\end{rema}
%%%%%%%%%%%%%%%%%%%%%%%%%%%%%%%%%%%%%%%%%%%%%%%%%%%%%%%

%%%%%%%%%%%%%%%%%%%%%%%%%%%%%%%%%%%%%%%%%%%%%%%%%%%%%%%%%%
\begin{thm}
Fix $\overline{y}\in\mathfrak{h}/2\pi i P^\vee$. With the notations as above, 
\[
A_{\textup{Ga};i}^{(N)}(\overline{y})\in U(\mathfrak{g}_{\sigma_{\overline{y}}}^+)\otimes
U(\mathfrak{g})^{\otimes N}\otimes U(\mathfrak{g}_{\sigma_{\underline{y}}}^+)\qquad (i=1,\ldots,N)
\]
defined by
\[
A_{\textup{Ga};i}^{(N)}(\overline{y}):=-\frac{1}{2}\sum_{s=i+1}^N\varpi_{is}+
\Bigl(\kappa_{\textup{Ga};i}(\overline{y})+\frac{1}{2}\sum_{s\not=i}r_{\textup{Ga};si}+\frac{1}{2}\sum_{s\not=i}\widetilde{r}_{\textup{Ga};si}(\overline{y})
\Bigr)
\]
pairwise commute.
\end{thm}
%%%%%%%%%%%%%%%%%%%%%%%%%%%%%%%%%%%%%%%%%%%%%%%%%%%%%%%%%
\begin{proof}
This follows from the observations in Remark \ref{remtowardsGaudin} applied to
\[
(U,\mathfrak{a}, \theta,A_\ell,A_r,r,\kappa)=\bigl(U(\mathfrak{g}), \{0\}, \sigma_{\overline{y}}, U(\mathfrak{g}_{\sigma_{\overline{y}}}^+),
U(\mathfrak{g}_{\sigma_{\overline{y}}}^+), r_{\textup{Ga}}, \kappa_{\textup{Ga}}\bigr).
\]
\end{proof}
%%%%%%%%%%%%%%%%%%%%%%%%%%%%%%%%%%%%%%%%%%%%%%%%%%%%%%%%%%

%%%%%%%%%%%%%
\begin{rema}
(1) By Remark \ref{limitGaudin} and Subsection \ref{DegScheme} the $A_{\textup{Ga};i}^{(N)}(\overline{0})$ are obtained from the commuting trigonometric
asymptotic boundary KZB-operators on $\mathfrak{h}$ by sending the real part of the {\it dynamical} parameters $x\in\mathfrak{h}$ to infinity in the appropriate Weyl chamber.
As a consequence the $A_{\textup{Ga};i}^{(N)}(\overline{0})$ ($1\leq i\leq N$)
are asymptotic trigonometric boundary Gaudin Hamiltonians (see Subsection \ref{DegScheme}). Note that the $A_{\textup{Ga};i}^{(N)}(\overline{y})$ 
are gauge-equivalent to the boundary Gaudin Hamiltonians
$A_{\textup{Ga};i}^{(N)}(\overline{0})$, due to Corollary \ref{xind}.\\
(2) Note the similarity of the asymptotic boundary Gaudin Hamiltonians $A_{\textup{Ga};i}^{(N)}(\overline{y})$ and the ones constructed in \cite[\S II C \& Thm. 2.3]{Skr2}.
The latter Hamiltonians
are derived from Lax operators involving non-unitary generalised classical $r$-matrices with spectral parameter. These generalised classical $r$-matrices are obtained from a unitary classical $r$-matrix by some folding procedure along a compatible involution.  

Skrypnyk's \cite{Skr1,Skr2} folding theory is significant different from the folding procedure introduced in this paper. Skrypnyk's theory is for classical unitary $r$-matrices with spectral parameter, leading to a {\it single} generalised non-unitary symmetric $r$-matrix that satisfies a {\it single} generalised classical Yang-Baxter equation. The folding theory in the present paper is for (quasi-unitary) classical {\it dynamical} $r$-matrices, leading to classical dynamical $r$-matrix pairs satisfying {\it three} coupled classical dynamical Yang-Baxter equations and a fourth classical non-dynamical Yang-Baxter type equation. 

In the degeneration scheme \eqref{dscheme},
Skrypnyk's folding theory takes place in the "Gaudin" box and the present folding theory in the "asymptotic KZB" box. It is an open problem whether these folding theories 
are degenerations of (a single) folding theory in the "KZB" box, where both the dynamical and spectral parameters are present.
\end{rema}
%%%%%%%%%%%%%%%%%%%%%%%%%%%%%%%%%%%%%%%%

%%%%%%%%%%%%%%%%%%%%%%%%%%%%%%%%%%%%%%%%%%%%%%%%%%%%%%%%%%%

 %%%%%%%%%%%%%%%%%%%%%%%%

\end{document}